\documentclass{amsart}
\usepackage{mathabx}
\usepackage{graphicx}
\usepackage{amsfonts}
\usepackage{amsmath}
\usepackage{amscd}
\usepackage{amssymb}
\usepackage{verbatim}
\usepackage{hyperref}
\usepackage{marginnote}
\usepackage{tikz}
\usepackage{cite}
\usepackage{hyperref}
\usepackage{cleveref}
\usepackage{mathdots}

\newenvironment{proof*}{\noindent\emph{Proof}}{$\square$\smallskip}

\newtheorem{theorem}{Theorem}[section]

\newtheorem{Definition}[theorem]{Definition}
\newtheorem{lemma}[theorem]{Lemma}
\newtheorem{Example}[theorem]{Example}

\newtheorem{corollary}[theorem]{Corollary}
\newtheorem{Remark}[theorem]{Remark}
\newtheorem{proposition}[theorem]{Proposition}

\newtheorem{Exercise}[theorem]{Exercise}
\newtheorem{Exercises}[theorem]{Exercises}
\newtheorem{Notation}[theorem]{Notation}
\newtheorem{Convention}[theorem]{Convention}
\newtheorem{standing assumption}[theorem]{Standing Assumption}

\newenvironment{definition}{\begin{Definition}\normalfont}{\end{Definition}}
\newenvironment{example}{\begin{Example}\normalfont}{\end{Example}}
\newenvironment{remark}{\begin{Remark}\normalfont}{\end{Remark}}

\newenvironment{convention}{\begin{Convention}\normalfont}{\end{Convention}}

\newcommand{\lk}{\ensuremath{{\rm lk}}} 
                            
\title[Finiteness properties via expansion sets]{Finiteness properties of generalized Thompson groups via expansion sets}  
\author[D.~S.~Farley]{Daniel S. Farley}
\address{Department of Mathematics\\ Miami University\\ Oxford, OH 45056 U.S.A.}
\email{farleyds@miamioh.edu}

\date{\today}

\begin{document}

\begin{abstract} 
We outline a general procedure that builds classifying spaces for generalized Thompson groups 
$\Gamma$. The construction depends on a small number of choices: (1) an inverse semigroup $S$
of partial transformations that ``locally determine" $\Gamma$; (2) an equivalence relation on 
certain pairs $(f,D)$, and (3) an ``expansion" rule $\mathcal{E}$. These choices determine an \emph{expansion set} $\mathcal{B}$, which is a combinatorial device that outputs a simplicial complex $\Delta^{f}_{\mathcal{B}}$ upon which 
$\Gamma$ acts. Under favorable conditions, often achieved in practice, $\Delta^{f}_{\mathcal{B}}$ is contractible, and the action of $\Gamma$ has small stabilizers.

The definition of $\Delta^{f}_{\mathcal{B}}$ is such that ascending and descending links in $\Delta^{f}_{\mathcal{B}}$ can be described via formulas that depend only on the expansion rule 
$\mathcal{E}$. The result is to facilitate the usual computations of the connectivity of the descending link. Under natural hypotheses, one can prove that the acting group has type $F_{\infty}$. The net effect of our results is to automate results of this kind.

Several applications are given; in particular, we sketch unified proofs that $V$, $nV$, R\"{o}ver's group $G$, and the Lodha-Moore group have type $F_{\infty}$.

\end{abstract}



\maketitle

\setcounter{tocdepth}{2} \tableofcontents

\section{Introduction}

\subsection{Generalized Thompson groups and finiteness properties} This paper considers the problem of finding topological finiteness properties of 
generalized Thompson groups.
 
A group $G$ has type $F_{n}$ if there is a $K(G,1)$ complex with finitely many cells 
of dimension $n$ or less. A group $G$ has type $F_{\infty}$ if it has type $F_{n}$ for all $n$ (or, equivalently, if there is a $K(G,1)$ complex with finitely many cells of each dimension). The properties $F_{n}$ (and $F_{\infty}$) are \emph{topological finiteness properties}; i.e., topological properties of a group that hold for all finite groups. 

 The basic examples of Thompson groups are $F$, $T$, and $V$. 
 The group $F$ can be described as the group of piecewise linear homeomorphisms $h$ of the unit interval such that: i) $h'(x)$ is an integral power of $2$, wherever $h'$
 is defined; ii) if $h'(x)$ is undefined, then $x$ is a dyadic rational number. The groups $T$ and $V$ are similarly defined groups of homeomorphisms of the circle and the Cantor set (respectively). A standard introduction
 to $F$, $T$, $V$, and some related groups, is \cite{CFP}. 

In the present paper, 
 the property of being a generalized Thompson group is understood to be the result of the \emph{piecewise}
 nature of the group's definition. Thus, just as members of Thompson's group $F$ are defined by ``local" behavior on certain particular pieces, the groups in our class of generalized Thompson groups will act on spaces $X$ in ways that are locally determined by a collection of partial transformations $S$. Certain generalizations of the Thompson groups, like the braided Thompson groups \cite{BrinBraid, Dehornoy}, lie outside the scope of this paper, although the problem of extending our methods to those groups remains an interesting one. 
 
 The study of the topological finiteness properties of Thompson's groups and their generalizations began in the 1980s. 
Brown and Geoghegan \cite{BrownGeoghegan} showed that Thompson's group $F$ has type $F_{\infty}$. Brown \cite{Brown} subsequently showed that Thompson's groups $T$ and $V$, and the generalized $n$-ary Thompson groups
$F_{n,r}$, $T_{n,r}$, and $V_{n,r}$, have type $F_{\infty}$.  

In recent years, there have been numerous studies of generalized Thompson groups and their topological finiteness properties \cite{QV, BelkMatucci, BraidedV, F(LM), nV, Lodha, Nucinkis,  Skipper, SkipperZaremsky}. In addition to the above results, several authors have proposed general frameworks for establishing finiteness properties of generalized Thompson groups \cite{FH2, FH1, Thumann, Witzel, WZcloning}. We will now describe one such framework, which
was proposed by Hughes and the author in \cite{FH2}.

\subsection{A framework for establishing finiteness properties} \label{subsection:framework} 
Let $\Gamma$ be a generalized Thompson group. From the perspective of \cite{FH2}, this means that $\Gamma$ is a group of bijections of a set $X$, and there is an inverse semigroup $S$ of partial transformations of $X$ such that each $\gamma \in \Gamma$ is expressible as a finite disjoint union of members from $S$; i.e., we have an equality
\[ \gamma = \coprod_{i=1}^{n} s_{i}, \]
where each $s_{i}$ is in $S$, and the domains (and images) of the $s_{i}$ do not overlap. (By an \emph{inverse semigroup}, we will always mean a collection $S$ of partial bijections $s: A \rightarrow B$, where $A,B \subseteq X$, such that $S$ is closed under inverses and compositions. The composition $s_{1} \circ s_{2}$ is defined ``on overlaps":  $dom(s_{1} \circ s_{2})$ is the set of all $x \in dom(s_{2})$ such that $s_{2}(x) \in dom(s_{1})$. The empty function serves as a zero element.) The main argument of \cite{FH2} builds a classifying space for $\Gamma$ in two steps: first, the authors build a natural directed set upon which $\Gamma$ acts, and, second, the authors introduce \emph{expansion schemes}, which offer a systematic way to simplify the directed set construction. The latter, simplified classifying spaces lead to proofs that various generalized Thompson groups have type $F_{\infty}$. Here, the directed set construction resembles the ones due to Brown \cite{Brown}, while the simplified constructions (employing expansion schemes) are modelled after the ones found in Stein \cite{Stein}.

The directed set construction begins with the set $\mathcal{D}^{+}_{S}$ of non-empty \emph{domains}, which are simply the domains of transformations from $S$. 
We then introduce an \emph{$S$-structure} $(\mathbb{S},\mathbb{P})$, where $\mathbb{S}$ and $\mathbb{P}$ are functions, called the \emph{structure} and \emph{pattern} functions, respectively. The function $\mathbb{S}$ associates to each pair
$(D_{1}, D_{2}) \in \mathcal{D}^{+}_{S} \times \mathcal{D}^{+}_{S}$ a (possibly empty) collection of transformations $\mathbb{S}(D_{1}, D_{2})$. Each set $\mathbb{S}(D_{1}, D_{2})$ is a subset of $S$ whose members have $D_{1}$ as domain and $D_{2}$ as image. The sets $\mathbb{S}(D_{1},D_{2})$ are further required to satisfy various ``groupoid-like" properties. The function $\mathbb{P}$ associates to each domain $D$ a collection of preferred finite partitions into subdomains. Thus, $\mathbb{P}$ assigns a kind of subdivision rule to the domains of $\mathcal{D}^{+}_{S}$. 

Let $\widehat{S}$ denote the set of all finite disjoint unions 
of members of $S$.
For pairs $(f,D), (g,D') \in \widehat{S} \times \mathcal{D}^{+}_{S}$ such that
$D \subseteq dom(f)$ and $D' \subseteq dom(g)$, we write
\[ (f,D) \sim (g,D') \]
if there is $h \in \mathbb{S}(D',D)$ such that $f \circ h = g$. The relation $\sim$ is an equivalence relation. We let $[f,D]$ denote the equivalence class of $(f,D)$.

A \emph{vertex} is a finite collection
\[ v= \{ [f_{1}, D_{1}], \ldots, [f_{m},D_{m}] \} \]
such that the images $f_{i}(D_{i})$ ($i=1, \ldots, m$) partition $X$. We then define an \emph{expansion partial order}, writing $v \leq v'$ if there is a sequence of vertices 
\[ v = v_{0}, v_{1}, v_{2}, \ldots, v_{k} = v' \]
such that, for $i=1, \ldots, k$, $v_{i}$ is the result of replacing some member $[g,E]$ of $v_{i-1}$
by a collection 
\[ [g_{\mid}, E_{1}], \ldots, [g_{\mid}, E_{\ell}], \]
where $\{ E_{1}, \ldots, E_{\ell}\}$ is one of the preferred partitions of $E$ determined by the function $\mathbb{P}$. [This definition of expansion is a slight oversimplification, but it is sufficient for this brief overview.] Under suitable assumptions on $\mathbb{S}$ and $\mathbb{P}$, as laid out in \cite{FH2}, the set of vertices $\mathcal{V}$ becomes a directed set under the expansion partial order. It follows that the simplicial realization $\Delta(\mathcal{V})$ of the directed set is contractible.

The simplicial complex $\Delta(\mathcal{V})$ has various properties that make it difficult to work with in practice. For instance, $\Delta(\mathcal{V})$ almost always fails to be locally finite, and indeed it is generally not even locally finite-dimensional. In order to simplify the complex $\Delta(\mathcal{V})$, \cite{FH2} introduces an \emph{expansion scheme} $\mathcal{E}$, which restricts the admissible expansions from each pair $[f,D]$. We then define a complex $\Delta^{\mathcal{E}}(\mathcal{V})$, which consists of the simplices 
\[ v_{0} < v_{1} < \ldots < v_{k} \]
in $\Delta(\mathcal{V})$ such that each vertex $v_{i}$ may be obtained from 
$v_{0}$ using only the expansions allowed by $\mathcal{E}$. Under suitable hypotheses, the complex 
$\Delta^{\mathcal{E}}(\mathcal{V})$ remains contractible, and in many cases it is (for instance) locally finite.

Once the complex $\Delta^{\mathcal{E}}(\mathcal{V})$ has been produced, we apply Brown's Finiteness Criterion \cite{Brown}. The complex $\Delta^{\mathcal{E}}(\mathcal{V})$ has a natural filtration $\{ \Delta^{\mathcal{E}}(\mathcal{V})_{n}: n \in \mathbb{N} \}$, where $\Delta^{\mathcal{E}}(\mathcal{V})_{n}$ denotes the subcomplex of 
$\Delta^{\mathcal{E}}(\mathcal{V})$ spanned by vertices of cardinality no greater than $n$. (The assignment
$v \mapsto |v|$ thus plays the role of a height function.) The latter part of \cite{FH2} gives general criteria that determine 
(1) when $\Gamma$ acts cocompactly 
on the complexes $\Delta^{\mathcal{E}}(\mathcal{V})_{n}$;
(2) when the stabilizers of vertices have type $F_{\infty}$ (or type $F_{n}$);
(3) when the descending link of a vertex in $\Delta^{\mathcal{E}}(\mathcal{V})$ is highly connected.
As a result, it is possible to conclude, in many cases, that the group $\Gamma$ has type $F_{\infty}$.

\subsection{A new approach via ``expansion sets"}
The Lodha-Moore group of \cite{Lodha,LM}, here denoted $G$, offered a useful test case of the approach from \cite{FH2}. We note that Lodha has shown \cite{Lodha} that $G$ has type $F_{\infty}$. The author attempted to use the results of \cite{FH2} to give another proof that $G$ has type $F_{\infty}$. This attempt, while ultimately successful \cite{F(LM)}, revealed a few shortcomings of \cite{FH2}. Most notably, the initial decision of \cite{FH2} to work with the full set of domains $\mathcal{D}^{+}_{S}$ proved to be very inconvenient. 

The author's experience in writing \cite{F(LM)} suggested that, to facilitate further developments, it may be useful to loosen some of the dependence on a specific list of initial assumptions.
In this paper, we will therefore, in effect, tell the story of \cite{FH2} backwards. Our starting point is a topological construction, into which algebraic data may be input. Under favorable conditions, the outcome is a proof that the relevant group has type $F_{\infty}$.

We begin with a new abstract object called an \emph{expansion set}, which is a $4$-tuple $(\mathcal{B}, X, supp, \mathcal{E})$, where $\mathcal{B}$ and $X$ are sets and
$supp$ and $\mathcal{E}$ are functions. We will frequently refer to $\mathcal{B}$
itself as the expansion set, since the remaining components play supporting roles.
Each member $b$ of $\mathcal{B}$ takes up a certain (non-empty) amount of space, called its \emph{support} and denoted $supp(b)$, within the set $X$. Thus, $supp$ is a function from $\mathcal{B}$ to $\mathcal{P}(X)$. Each $b \in \mathcal{B}$ is also assigned a partially ordered set $\mathcal{E}(b)$ of \emph{expansions}. 
The set $\mathcal{E}(b)$ has $\{ b \}$ as its least element. Every other element of $\mathcal{E}(b)$ takes the form $\{ b_{1}, \ldots, b_{m} \} \subseteq \mathcal{B}$
(for varying $2 \leq m < \infty$), where $\{ supp(b_{1}), \ldots, supp(b_{m}) \}$ is a partition of $supp(b)$. More generally, if two members 
$v _{1} = \{ b_{1}, \ldots, b_{m} \}$ and $v_{2} = \{ b'_{1}, \ldots, b'_{\ell} \}$ of $\mathcal{E}(b)$ are such that $v_{1} \lneq v_{2}$, then $\{ supp(b'_{1}), \ldots, supp(b'_{\ell}) \}$ is a proper refinement 
of $\{ supp(b_{1}), \ldots, supp(b_{m}) \}$.  (Here we omit discussion of a third property, Definition \ref{definition:expansionsets}(3), which can be viewed as a compatibility condition between different sets $\mathcal{E}(b)$ and $\mathcal{E}(b')$.) We will often think of $\mathcal{E}(b)$ as a simplicial complex in which the simplices are finite ascending chains. 

The expansion set $(\mathcal{B}, X, supp, \mathcal{E})$ determines a simplicial complex $\Delta_{\mathcal{B}}$. A \emph{vertex} in $\Delta_{\mathcal{B}}$ is a finite  subset $v = \{ b_{1}, \ldots, b_{m} \}$ of $\mathcal{B}$ (for some $m$) such that $supp(b_{i}) \cap supp(b_{j}) = \emptyset$ for $1 \leq i < j \leq m$. The cardinality of $v$ is its \emph{height}.
The simplicial structure in $\Delta_{\mathcal{B}}$ is completely determined by the ascending stars of the vertices, and each ascending star $st_{\uparrow}(v)$ is essentially defined to be the product of the individual sets $\mathcal{E}(b_{i})$ for $i=1, \ldots, m$, where 
$v = \{ b_{1}, \ldots, b_{m} \}$. As a result, each ascending link $lk_{\uparrow}(v)$ in $\Delta_{\mathcal{B}}$ has a natural join structure. The descending star $st_{\downarrow}(v)$ is more complicated, but it admits a cover by subcomplexes $st_{\downarrow}^{\mathcal{P}}(v)$, each of which has a natural factorization as a product. Analogous statements are true of the descending link $lk_{\downarrow}(v)$: it does not directly admit a join structure, but there are subcomplexes $lk_{\downarrow}^{\mathcal{P}}(v)$ that cover
$lk_{\downarrow}(v)$, such that each $lk_{\downarrow}^{\mathcal{P}}(v)$ has a natural join structure.

The complex of greatest interest is the \emph{full-support complex}, denoted $\Delta^{f}(\mathcal{B})$, which is spanned by the vertices 
$v = \{ b_{1}, \ldots, b_{m} \}$ such that $\{ supp(b_{1}), \ldots, supp(b_{m}) \}$
partitions $X$. It is this complex that can be most relevantly compared to the complexes in Stein's work \cite{Stein}. (We note, however, that this type of comparison depends on the size of the sets $\mathcal{E}(b)$: for very large sets $\mathcal{E}(b)$, the construction is more like the ones in \cite{Brown}, while, for small sets $\mathcal{E}(b)$, the construction is like the ones in \cite{Stein}.)

Within the above combinatorial-topological set-up, it is possible to prove that most parts of Brown's Finiteness Criterion (Theorem \ref{theorem:Brown}) apply to 
$\Delta^{f}_{\mathcal{B}}$, under suitable (but rather general) hypotheses. Perhaps most notably, the above set-up gives us a very useful and general inductive approach to analyzing the connectivity of the descending link (Corollary \ref{corollary:connectivityeasy}), which is essential in proving that the connectivity of the members of the natural filtration
\[ \{ \Delta^{f}_{\mathcal{B},n} \mid n \in \mathbb{N} \} \]
 tends to infinity with $n$. (The complex $\Delta^{f}_{\mathcal{B},n}$ is the subcomplex of $\Delta^{f}_{\mathcal{B}}$ spanned by vertices of height no more than $n$.)
 
All that is missing at this point is a group action. But it is straightforward enough to create a suitable definition of a group action on the expansion set $\mathcal{B}$. Our approach actually defines the (partial) action of an inverse semigroup $\widehat{S}$
on $\Delta_{\mathcal{B}}$ (Definition \ref{definition:actionsonexpansionsets}). Using our definitions, we are able to set down general conditions under which cell stabilizers have type $F_{n}$ (Corollary \ref{corollary:Fnstabilizers}), and under which the action of $\Gamma$ on
$\Delta^{f}_{\mathcal{B},n}$ is cocompact (Proposition \ref{proposition:cocompact}). With these elements in place, one has a general schema (Theorem \ref{theorem:template}) for proving that a generalized Thompson group has type $F_{\infty}$.
(While we do not directly concern ourselves here with groups of $F_{n}$ that are not of type $F_{n+1}$, we believe that our methods can be adapted to such groups without too much modification.)  

The abstract definition of expansion sets confers the usual benefits of abstraction: generality and clarity. The current version of the theory applies seamlessly to the Lodha-Moore group, for example. Putting expansion sets at the center of our exposition also leads to a streamlining of the topological arguments. An additional benefit is that we are able to handle ``$F$"- and ``$T$''-like groups, where \cite{FH2} considers only ``$V$"-like groups. 

Finally, in Section \ref{section:construction}, we offer examples of expansion sets. The first example, which we simply call the basic construction in Subsection \ref{subsection:basic}, can be applied to any group $\Gamma$ that is locally determined by an inverse semigroup $S$ acting on $X$. The price of such generality is that the complex $\Delta^{f}_{\mathcal{B}}$ is difficult to analyze directly. We view the basic construction as a direct descendant of the complexes considered by Brown in \cite{Brown}. Here the sets $\mathcal{E}(b)$ are very large. In subsequent examples, we show how one can strategically reduce the size of $\mathcal{E}(b)$ and achieve much more economical complexes. We explicitly consider Thompson's group $V$, the Brin-Thompson groups $nV$, R\"{o}ver's group, and the Lodha-Moore group, sketching proofs that all have type $F_{\infty}$. (The latter groups were first proved to have type $F_{\infty}$ in
\cite{Brown, BrownGeoghegan}, \cite{nV}, \cite{BelkMatucci}, and \cite{Lodha}, respectively.) The complexes that we build for these groups may be considered versions of the ones first constructed by Stein \cite{Stein}. 

The underlying message of Section \ref{section:construction} is that the proof of $F_{\infty}$ can largely be reduced to a single choice, namely the choice of function $\mathcal{E}$. (The other choices are fairly standard. However, the example of R\"{o}ver's group also shows the potential benefits of tinkering with the basic definition of the equivalence relation on pairs $(f,D)$.) The desired topological properties of $\Delta^{f}_{\mathcal{B}}$ follow directly, based on the work of earlier sections. The proof of $F_{\infty}$ for most of the groups listed above becomes a straightforward application of the theory. A partial exception is the Lodha-Moore group, which still requires an involved combinatorial analysis \cite{F(LM)} that will not be reproduced here.  

  Let us summarize the paper section-by-section. Section \ref{section:expansion}
  contains an exposition of expansion sets, beginning in Subsection \ref{subsection:expdef} with their definition and the definitions of the associated complexes $\Delta^{f}_{\mathcal{B}}$ and $\Delta_{\mathcal{B}}$. Subsection \ref{subsection:upanddown} establishes basic facts about upward and downward stars and links in $\Delta_{\mathcal{B}}$. The upward stars and links have straightforward factorizations as products and joins (respectively) of simpler complexes, while the downward stars and links are covered by subcomplexes with such factorizations. Subsection \ref{subsection:nconncrit} considers the connectivity of $\Delta^{f}_{\mathcal{B}}$ as a whole, culminating in Theorem \ref{theorem:nconnectivity}, which establishes a criterion for $n$-connectivity of $\Delta^{f}_{\mathcal{B}}$. Subsection \ref{subsection:filtration} defines a natural filtration 
  $\{ \Delta^{f}_{\mathcal{B},n} \}$ of $\Delta^{f}_{\mathcal{B}}$ and characterizes the connectivity of the former in terms of the descending link. This subsection also marks the end of the general theory of expansion sets and the topology of their associated complexes. Subsection \ref{subsection:threetypes} introduces two commonly encountered types of expansion set: linear and cyclic. These correspond
  to ``$F$"- and ``$T$"-like groups, respectively. The subsection goes on to consider, in each case, a combinatorial analysis of the connectivity of the descending links, under commonly encountered (but not completely general) hypotheses.  

Section \ref{section:groupactions} introduces group actions on expansion sets in the linear, cyclic, and ``permutational" (i.e., ``$V$''-like) cases. Here we consider definitions (Subsection \ref{subsection:actions}), cell stabilizers (Subsection \ref{subsection:stabilizers}),
cocompactness of the members of the filtration (Subsection \ref{subsection:cocompact}), and review Brown's Finiteness Criterion (Subsection \ref{subsection:Brown}). Section \ref{section:groupactions}
culminates in Theorem \ref{theorem:template}, which assembles all previous results into a set of sufficient conditions for the group $\Gamma$ to have type $F_{\infty}$. 

Section \ref{section:construction} contains examples. Subsection \ref{subsection:basic} considers the most generic examples, while Subsections \ref{subsection:thompson}, \ref{subsection:brin}, \ref{subsection:rover}, and \ref{subsection:LM} consider Thompson's group
$V$, the Brin-Thompson group $nV$, R\"{o}ver's group, and the Lodha-Moore group, respectively. The treatment of $V$ in Subsection \ref{subsection:thompson} is the most complete, and is intended to show that most of the necessary verifications are routine.

I would like to thank Yuya Kodama for suggesting a correction to an earlier version of this paper. Remark \ref{remark:final} previously stated that Moawad \cite{Moawad} had shown that certain groups defined by Kodama \cite{Kodama} have type $F_{\infty}$. In fact, Moawad's paper concerns slightly different groups (as Moawad himself is careful to note).


\section{Expansion sets} \label{section:expansion} 

In this section, we define expansion sets and consider the topological properties of the associated simplicial complexes
$\Delta_{\mathcal{B}}$ and $\Delta^{f}_{\mathcal{B}}$. Subsection \ref{subsection:expdef} contains the basic definitions. Subsection \ref{subsection:upanddown} describes the topology of upward and downward stars and links. The latter description results in a criterion  for the $n$-connectivity of $\Delta^{f}_{\mathcal{B}}$. This is established in Subsection\ref{subsection:nconncrit}. Subsection \ref{subsection:filtration} introduces a natural finite-height filtration of $\Delta^{f}_{\mathcal{B}}$, and also relates the connectivity of the latter to the descending link. Subsection \ref{subsection:threetypes} specializes the discussion somewhat, introducing special types of expansion set, and providing criteria that allow one to conclude that the descending link is highly connected, for sufficiently large height.

\subsection{Definition of expansion sets and associated simplicial complexes}
\label{subsection:expdef}


\begin{definition} \label{definition:expansionsets}
(Expansion sets) 
An \emph{expansion set} is a $4$-tuple $(\mathcal{B}, X, supp, \mathcal{E})$, where $\mathcal{B}$ and $X$ are sets, and $supp: \mathcal{B} \rightarrow \mathcal{P}(X)$ and $\mathcal{E}$ are functions. We may sometimes let $\mathcal{B}$ itself denote an expansion set, for the sake of brevity. It is occasionally convenient to call $\mathcal{B}$ an \emph{expansion set over $X$} when we want to emphasize the role of the set $X$.

For each $b \in \mathcal{B}$, $supp(b)$ is required to be non-empty. The function $supp$ is called the \emph{support function}, and $supp(b)$ is the \emph{support} of $b$.  

A \emph{vertex} is a finite subset $v = \{ b_{1}, \ldots, b_{k}\} \subseteq \mathcal{B}$ such that $supp(b_{i}) \cap supp(b_{j}) = \emptyset$ when $i \neq j$. The set of all vertices is denoted $\mathcal{V}_{\mathcal{B}}$. For each $v \in \mathcal{V}_{\mathcal{B}}$, we define 
\[ supp(v) = \bigcup_{\ell =1}^{k} supp(b_{\ell}); \quad \quad  
P(v) = \{ supp(b_{\ell}) \mid \ell \in \{ 1, \ldots, k \} \}.  \]
The collection $P(v)$ is the \emph{partition induced by $v$}. It is a partition of $supp(v)$.

 The function $\mathcal{E}$ assigns a partially ordered set of vertices, denoted $\mathcal{E}(b)$, to each $b \in \mathcal{B}$. 
 The set $\mathcal{E}(b)$  and partial order $\leq$ are required to satisfy the following three conditions:
\begin{enumerate}
\item The vertex $\{ b \}$ is the least member of $\mathcal{E}(b)$; i.e., $\{ b \} \in \mathcal{E}(b)$ and, if $v \in \mathcal{E}(b)$, then $\{ b \} \leq v$;
\item if $v_{1}, v_{2} \in \mathcal{E}(b)$ and $v_{1} < v_{2}$, then $P(v_{2})$ is a proper refinement of $P(v_{1})$.
\end{enumerate}

Let $v_{1}, v_{2} \in \mathcal{E}(b)$. For  $b \in v_{1}$, we define
\[ r_{b}(v_{2}) = \{ b' \in v_{2} \mid supp(b') \subseteq supp(b) \}. \]
(Thus, $r_{b}(v_{2})$ is the portion of the vertex $v_{2}$ whose support is contained in that of $b$.)
With these conventions, we can state the final condition:
\begin{enumerate}
\item[(3)] If $v_{1}, v_{2} \in \mathcal{E}(b)$, $v_{1} \leq v_{2}$ if and only if $r_{b'}(v_{2}) \in \mathcal{E}(b')$, for each $b' \in v_{1}$.
\end{enumerate}  
\end{definition}

\begin{remark} \label{remark:posetremark}
We note that each set $\mathcal{E}(b)$ is a separate partially ordered set. We are denoting a (usually infinite) set of partial orders by $\leq$, avoiding notation like $\leq_{b}$ for the sake of simplicity.
\end{remark}

\begin{figure}[!b]
\begin{center}
\begin{tikzpicture}
\filldraw[lightgray] (0,0) -- (1,1) -- (0,2) -- cycle;
\draw[black] (0,0) -- (0,2);
\draw[black] (0,0) -- (1,1);
\draw[black] (1,1) -- (0,2);

\draw[black] (5.5,.25) -- (5.5,1.75);

\draw[black] (-1.4,-.2) rectangle (-.4,.2);
\node at (-.9, 0){\footnotesize $b$};

\draw[black] (-1.4,1.8) rectangle (-.4,2.2);
\draw[black] (-1.07,1.8) rectangle (-.73,2.2);
\node at (-1.235,2){\footnotesize $b_{1}$};
\node at (-.905,2){\footnotesize $b_{2}$};
\node at (-.55,2){\footnotesize $b_{3}$};

\draw[black] (1.4,.8) rectangle (2.4,1.2);
\draw[black] (1.4,.8) rectangle (1.73,1.2);
\node at (1.6,1){\footnotesize $b_{1}$};
\node at (2.065,1){\footnotesize $b'$};

\filldraw [black] (0,2) circle (1.5pt);
\filldraw [black] (1,1) circle (1.5pt);
\filldraw [black] (0,0) circle (1.5pt);

\filldraw [black] (5.5,.25) circle (1.5pt);
\filldraw [black] (5.5,1.75) circle (1.5pt);

\draw[black](5.9,.05) rectangle (6.57,.45); 
\node at (6.235,.25){\footnotesize $b'$};
\draw[black](5.9,1.55) rectangle (6.57,1.95); 
\draw[black](6.235,1.55) -- (6.235,1.95);
\node at (6.08,1.75){\footnotesize $b_{2}$};
\node at (6.42,1.75){\footnotesize $b_{3}$};

\node at (-.3,1){$\mathbf{e_{2}}$};
\node at (.8,.3){$\mathbf{e_{1}}$};
\node at (.8,1.6){$\mathbf{e_{3}}$};
\node at (.4,1){$\mathbf{f}$};

\end{tikzpicture}
\end{center}
\caption{The figure illustrates possible definitions of $\mathcal{E}(b)$ (left) and $\mathcal{E}(b')$ (right).}
\label{expansionsetpicture}
\end{figure}
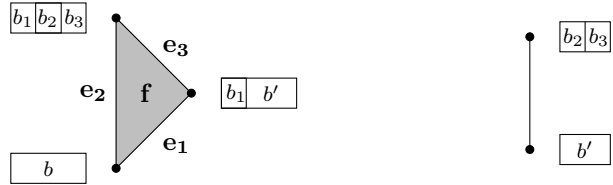

\begin{remark} \label{remark:Ebasasimplicialcomplex}(the partially ordered sets $\mathcal{E}(b)$ as abstract simplicial complexes)
We can also view the partially ordered set $(\mathcal{E}(b), \leq)$ as the vertex set for an abstract simplicial complex, in which the simplices are finite non-empty strictly ascending chains. We will generally do this without additional comment. 
\end{remark}

\begin{example} \label{example:figure} 
Figure \ref{expansionsetpicture} depicts two different partially ordered sets,
$\mathcal{E}(b)$ and $\mathcal{E}(b')$, on the left and right respectively, which are part of some larger expansion set $\mathcal{B}$. The vertices of the simplicial complex
$\mathcal{E}(b)$ are $\{ b \}$, $\{ b_{1}, b' \}$, and 
$\{ b_{1}, b_{2}, b_{3} \}$ (reading counterclockwise from the bottom). The vertices of $\mathcal{E}(b')$ are $\{ b' \}$ and 
$\{ b_{2}, b_{3} \}$.

The rectangles labelling the vertices are intended to describe the supports of the elements in question, where the base of a rectangle represents the support of its label. Thus, for instance,
\[ supp(b) = supp(b_{1}) \cup supp(b') = supp(b_{1}) \cup supp(b_{2}) \cup supp(b_{3}). \]
(The height of the rectangles has no specific meaning.) Properties (1) and (2) from Definition \ref{definition:expansionsets} are straightforward: property (1) simply says that $\{ b \}$ and $\{ b' \}$ 
are the least vertices in $\mathcal{E}(b)$ and $\mathcal{E}(b')$, while (2) forces $P(v_{2})$ to refine $P(v_{1})$, whenever $v_{1}$ and $v_{2}$ are members of the same $\mathcal{E}(b'')$, and $v_{1} \leq v_{2}$ in $\mathcal{E}(b'')$. Both of these properties are clearly satisfied. 

We consider property (3) from Definition \ref{definition:expansionsets} in more detail. Let us assume that $\mathcal{E}(b')$ is fixed, and consider certain variant definitions of $\mathcal{E}(b)$. 

First, suppose that $\mathcal{E}(b)$ consists only of the edges $e_{1}$ and $e_{2}$, and their endpoints. By (3),
\[ v:= \{ b_{1}, b' \} \leq \{ b_{1}, b_{2}, b_{3} \} \]
in $\mathcal{E}(b)$
if and only if $r_{b_{1}}(v) = \{ b_{1} \} \in \mathcal{E}(b_{1})$ and 
$r_{b'}(v) = \{ b_{2}, b_{3} \} \in \mathcal{E}(b')$. The latter two conditions are clear: $\{ b_{2}, b_{3} \} \in \mathcal{E}(b')$ by the right half of the figure, while $\{ b_{1} \} \in \mathcal{E}(b_{1})$ by Definition \ref{definition:expansionsets}(1). Thus, the absence of $e_{3}$ violates
(3). This reasoning shows not only that the presence of $e_{1}$ and $e_{2}$ forces $e_{3}$ to be an edge, but also that the $2$-simplex 
$f$ must be present, since the simplices of $\mathcal{E}(b)$ are finite chains.

Note that, if only $e_{1}$ (and endpoints) were present in $\mathcal{E}(b)$, $\mathcal{E}(b)$ would satisfy (3). This means, in particular, that (3) does not force $\mathcal{E}(b)$ to be ``closed under expansion", except in a restricted sense (as indicated in the previous paragraph). Similarly, the assignment that puts only $e_{2}$ into $\mathcal{E}(b)$ is also consistent with Definition \ref{definition:expansionsets}.   
\end{example}

\begin{definition} \label{definition:deltaB} (The abstract simplicial complex $\Delta_{\mathcal{B }}$) 
Let $\mathcal{B}$ be an expansion set. We define 
$\Delta_{\mathcal{B}} = (\mathcal{V}_{\mathcal{B}}, \mathcal{S}_{\mathcal{B}})$ as follows. 
\begin{enumerate}
\item $\mathcal{V}_{\mathcal{B}}$ is as defined in Definition \ref{definition:expansionsets};
\item Let $\{ v_{1}, v_{2}, \ldots, v_{n} \} \subseteq \mathcal{V}_{\mathcal{B}}$. 
The collection $\{ v_{1}, \ldots, v_{n} \}$ is in
$\mathcal{S}_{\mathcal{B}}$ if and only if, possibly after a (fixed) reordering of the $v_{i}$,
\[ \{ b \} = r_{b}(v_{1}) \leq r_{b}(v_{2}) \leq \ldots \leq r_{b}(v_{n})\] is a chain
in $\mathcal{E}(b)$ for each $b \in v_{1}$. where $r_{b}$ is  as in Definition \ref{definition:expansionsets}.
\end{enumerate}
\end{definition}

\begin{remark} (Canonical ordering) \label{remark:canonical}
If $S = \{ v_{1}, \ldots, v_{n} \} \in \mathcal{S}_{\mathcal{B}}$, then we say that $S$ is in \emph{canonical order} if 
$P(v_{j})$ refines $P(v_{i})$, for all $i < j$. It is a consequence of Definition \ref{definition:expansionsets}(2) that the canonical order exists and is unique. The canonical order is the ordering referenced in Definition \ref{definition:deltaB}(2).
\end{remark}

\begin{lemma} \label{lemma:complex}
Let $v_{1} \leq \ldots \leq v_{n}$ be a chain
in $\mathcal{E}(b)$, for some $b \in \mathcal{B}$. If $j \in \{1, 2, \ldots, n \}$ and $b' \in v_{j}$, then 
\[ r_{b'}(v_{j}) \leq r_{b'}(v_{j+1}) \leq \ldots \leq r_{b'}(v_{n}) \]
is a chain in $\mathcal{E}(b')$.
\end{lemma}

\begin{proof}
First, we note that the general case follows from the case $j=1$. Indeed, assume that, whenever $v_{1} \leq v_{2} \leq \ldots \leq v_{n}$ is an arbitrary chain in some $\mathcal{E}(b)$ and $b' \in v_{1}$, then $r_{b'}(v_{1}) \leq \ldots \leq r_{b'}(v_{n})$ is a chain in
$\mathcal{E}(b')$. Let $v_{1} \leq \ldots \leq v_{n}$ and let $b' \in v_{j}$. We can apply our assumption to the chain $v_{j} \leq \ldots \leq v_{n}$ in $\mathcal{E}(b)$, concluding that $r_{b'}(v_{j}) \leq \ldots \leq r_{b'}(v_{n})$ is a chain in $\mathcal{E}(b')$, as required.

We turn to the $j=1$ case. Consider a chain 
$v_{1} \leq \ldots \leq v_{n}$ in $\mathcal{E}(b)$, and let $b' \in v_{1}$. Since
$v_{1} \leq v_{k}$ in $\mathcal{E}(b)$ for $k = 1, \ldots, n$, we can conclude from Definition \ref{definition:expansionsets}(3) that $r_{b'}(v_{k}) \in \mathcal{E}(b')$ for $k=1, \ldots, n$. The final step is to prove that 
$r_{b'}(v_{k_{1}}) \leq r_{b'}(v_{k_{2}})$ in $\mathcal{E}(b')$ whenever $1 \leq k_{1} \leq k_{2} \leq n$. By Definition \ref{definition:expansionsets}(3), it suffices to show that
\[ r_{b''}(r_{b'}(v_{k_{2}})) \in \mathcal{E}(b'') \]
for all $b'' \in r_{b'}(v_{k_{1}})$. 

Let $b'' \in r_{b'}(v_{k_{1}})$ be arbitrary.
We note that $b'' \in v_{k_{1}}$ (since $r_{b'}(v_{k_{1}}) \subseteq v_{k_{1}}$), so
$r_{b''}(v_{k_{2}}) \in \mathcal{E}(b'')$, by Definition \ref{definition:expansionsets}(3)
(since $v_{k_{1}} \leq v_{k_{2}}$ in $\mathcal{E}(b)$ and $b'' \in v_{k_{1}}$). The proof is then concluded by noting that, since $supp(b'') \subseteq supp(b')$, $r_{b''}(r_{b'}(v_{k_{2}})) = r_{b''}(v_{k_{2}}) \in \mathcal{E}(b'')$.
\end{proof}

\begin{proposition} \label{proposition:deltaB}
($\Delta_{\mathcal{B}}$ is a simplicial complex)
The pair $\Delta_{\mathcal{B}} = (\mathcal{V}_{\mathcal{B}}, \mathcal{S}_{\mathcal{B}})$
is an abstract simplicial complex.
\end{proposition}

\begin{proof}
It is clear that $\{ v \} \in \mathcal{S}_{\mathcal{B}}$, for all $v \in \mathcal{V}_{\mathcal{B}}$.

We need to show that if $n \geq 2$ and $\{ v_{1}, \ldots, v_{n} \} \in \mathcal{S}_{\mathcal{B}}$, then so is
$\{ v_{1}, \ldots, \hat{v}_{j}, \ldots, v_{n} \}$, for $j=1, \ldots, n$, where $\hat{v}_{j}$ indicates that $v_{j}$ is omitted. We assume that $\{ v_{1}, \ldots, v_{n} \}$ is in canonical order (Remark \ref{remark:canonical}). Our hypothesis implies that 
\[ \{ b \} = r_{b}(v_{1}) \leq r_{b}(v_{2}) \leq \ldots \leq r_{b}(v_{n}) \]
is a chain in $\mathcal{E}(b)$,
for each $b \in v_{1}$.

First assume that $j \geq 2$. It follows from the hypothesis that
\[ r_{b}(v_{1}) \leq \ldots \leq r_{b}(v_{j-1}) \leq r_{b}(v_{j+1}) \leq \ldots \leq r_{b}(v_{n}) \]
is a chain in $\mathcal{E}(b)$, for each $b \in v_{1}$. Thus,
$\{ v_{1}, \ldots, \hat{v}_{j}, \ldots, v_{n} \} \in \mathcal{S}_{\mathcal{B}}$, as desired.

Finally, assume that $j=1$. We need to show that 
$\{ v_{2}, \ldots, v_{n} \}$ is a simplex in $\Delta_{\mathcal{B}}$. According to Definition \ref{definition:deltaB}, it suffices to show that 
\[ r_{b'}(v_{2}) \leq r_{b'}(v_{3}) \leq \ldots \leq r_{b'}(v_{n}) \]
is a chain in $\mathcal{E}(b')$, for each $b' \in v_{2}$. This  follows directly from Lemma \ref{lemma:complex}.
\end{proof}


\begin{definition} \label{definition:fullsupportcomplex}
(The full-support subcomplex of $\Delta_{\mathcal{B}}$) 
We let $\Delta^{f}_{\mathcal{B}}$ denote the subcomplex of $\Delta_{\mathcal{B}}$ spanned by vertices $v$ such that $supp(v) = X$. This is the \emph{full-support subcomplex of $\Delta_{\mathcal{B}}$}. 
\end{definition}


\begin{remark}
If two vertices $v_{1}$ and $v_{2}$ are such that $supp(v_{1}) \neq supp(v_{2})$, then $v_{1}$ and $v_{2}$ are necessarily in distinct path components of $\Delta_{\mathcal{B}}$. Thus, $\Delta_{\mathcal{B}}$ will generally have infinitely many path components.

The object of ultimate interest for us is $\Delta^{f}_{\mathcal{B}}$, but it will be useful to follow the example of \cite{FH2} 
and treat the vertices of $\Delta_{\mathcal{B}} - \Delta^{f}_{\mathcal{B}}$ more or less equally. (Such vertices were called ``pseudovertices" in \cite{FH2}.) This will allow us to adopt a ``piecewise" approach to the (local) topology of $\Delta_{\mathcal{B}}$, as exemplified in Subsection \ref{subsection:upanddown}.  
 \end{remark}

\subsection{Upward and downward links and stars in $\Delta_{\mathcal{B}}$}
\label{subsection:upanddown}
In this subsection, we will show that the complexes $\Delta_{\mathcal{B}}$ and 
$\Delta^{f}_{\mathcal{B}}$ have natural height functions, and that the ascending star of any vertex $v \in \Delta_{\mathcal{B}}$ has a natural product structure. It follows directly that the ascending link has a natural join structure. We will also establish ``relativized" product and link structures for the applications in Subsection \ref{subsection:nconncrit}. 

The situation for descending stars and links is more complicated. We will describe subcomplexes $st_{\downarrow}^{\mathcal{P}}(v)$ and $lk_{\downarrow}^{\mathcal{P}}(v)$
of the descending star and link (respectively) that retain the product and link structures (respectively). We will later show that the descending star and link can be covered by these subcomplexes, which will allow us to apply the Nerve Theorem (Theorem \ref{theorem:Nerve}) to compute the connectivity of the downward link.

(We note that ``upward link" and ``ascending link" are used interchangeably throughout the paper. Similarly for ``downward link" and ``descending link", etc.)

\subsubsection{The height function on $\Delta_{\mathcal{B}}$; upward and downward links and stars}

\begin{definition} \label{definition:height}
(the height function on $\Delta_{\mathcal{B}}$)
For a vertex $v \in \mathcal{V}_{\mathcal{B}}$, we refer to $|v|$ (the cardinality of $v$) as the \emph{height} of the vertex $v$.
\end{definition}

\begin{definition} \label{definition:starsandlinks}
(ascending and descending stars and links)
For any vertex $v \in \mathcal{V}_{\mathcal{B}}$, we let $st(v)$ denote the subcomplex of $\Delta_{\mathcal{B}}$ consisting of those simplices that contain $v$, and all faces of such simplices. We let $lk(v)$ denote the subcomplex of $st(v)$ consisting of all simplices of which $v$ is not a vertex. These are the \emph{star} and \emph{link} of $v$, respectively.

The \emph{ascending star} of $v$, denoted $st_{\uparrow}(v)$, is the subcomplex of $st(v)$ consisting of simplices whose vertices have height no less than $|v|$. Similarly, the \emph{descending star} of $v$, denoted $st_{\downarrow}(v)$, is the subcomplex of $st(v)$ determined by simplices whose vertices have height no more than $|v|$. 

The \emph{ascending and descending links} of $v$, denoted $lk_{\uparrow}(v)$ and $lk_{\downarrow}(v)$ respectively, are the subcomplexes of the corresponding stars consisting of simplices that do not contain $v$.
\end{definition}

\begin{remark}
In what follows, we will sometimes denote the vertex $\{ b \}$ (for $b \in \mathcal{B}$) simply by $b$, in order to simplify notation, and we shall do so without further comment.
\end{remark}

\subsubsection{Topology of upward links and stars in $\Delta_{\mathcal{B}}$}

\begin{definition} \label{definition:productofposets}
(The product partial order)
Let $S_{1}, \ldots, S_{\ell}$ be partially ordered sets. The
\emph{product partial order} is defined as follows:
\[ (s'_{1}, \ldots, s'_{\ell}) \leq (s_{1}, \ldots, s_{\ell}) \quad \text{if and only if} \quad s'_{i} \leq s_{i} \]
for each $i =1 \ldots, \ell$.
\end{definition}

\begin{definition} \label{definition:ascstarasposet}
(The partial order on ascending stars)
Let $v \in \mathcal{V}_{\mathcal{B}}$.
Define a relation $\leq$ on the vertices of $st_{\uparrow}(v)$ by the rule
\[ u \leq w \quad \text{if and only if} \quad r_{b}(u) \leq r_{b}(w) \]
for all $b \in v$, where the latter inequality is relative to the partial order on $\mathcal{E}(b)$.
\end{definition}

\begin{proposition} \label{proposition:ascstarasproduct} 
(The ascending star as a product) Let $v = \{ b_{1}, \ldots, b_{k} \} \in \mathcal{V}_{\mathcal{B}}$.
The function $p: st_{\uparrow}(v) \rightarrow \mathcal{E}(b_{1}) \times \ldots \times \mathcal{E}(b_{k})$ defined by
\[ p(w) = (r_{b_{1}}(w), \ldots, r_{b_{k}}(w)) \]
is an isomorphism between the partially ordered sets
$st_{\uparrow}(v)$ and $\mathcal{E}(b_{1}) \times \ldots \times \mathcal{E}(b_{k})$.
\end{proposition}

\begin{proof}
Let $w$ be a vertex of $st_{\uparrow}(v)$. By property (2) from Definition \ref{definition:deltaB} and the definition of ascending star, 
$\{ b \} = r_{b}(v) \leq r_{b}(w)$ is a chain in $\mathcal{E}(b)$, for each $b \in v$. In particular, $r_{b}(w) \in \mathcal{E}(b)$ for each $b \in v$, so 
$(r_{b_{1}}(w), \ldots, r_{b_{k}}(w)) \in \mathcal{E}(b_{1}) \times \ldots \times
\mathcal{E}(b_{k})$. This shows that $p$ is well-defined.

Next we prove surjectivity of $p$. Let $(w_{1}, \ldots, w_{k}) \in \mathcal{E}(b_{1}) \times \ldots \times \mathcal{E}(b_{k})$. We consider the union $w = w_{1} \cup w_{2} \cup \ldots \cup w_{k}$. Note first that $w$ is a vertex (i.e., no two $b_{i}$ in $w$ have overlapping supports)
since each individual $w_{\beta}$ is a vertex, and $supp(w_{i}) \cap supp(w_{j}) = \emptyset$ for $i \neq j$. Clearly, $w \in st_{\uparrow}(v)$ (since $r_{b_{i}}(w) = w_{i}$ for $i=1, \ldots, k$).  It follows immediately that 
$p(w) = (w_{1}, \ldots, w_{k})$, so $p$ is surjective.

Suppose that $p(w) = p(w')$, for some $w, w' \in st_{\uparrow}(v)$.
This means that $r_{b_{\beta}}(w) = r_{b_{\beta}}(w')$, for $\beta = 1, \ldots, k$. It follows directly that $w=w'$ since 
$w = r_{b_{1}}(w) \cup \ldots \cup r_{b_{k}}(w)$ and $w' = r_{b_{1}}(w') \cup \ldots \cup r_{b_{k}}(w')$. (The latter expressions for $w$ and $w'$ follow from the fact that 
$P(w)$ and $P(w')$ are both refinements of $P(v)$, a consequence of Definition \ref{definition:expansionsets}(2).)  Thus $p$ is injective.

The isomorphism property follows directly from a comparison of Definitions \ref{definition:productofposets} and \ref{definition:ascstarasposet}. 
\end{proof}

\begin{definition}\label{definition:DeltaHat}
(Expansion sequences and the expansion partial order on $\mathcal{V}_{\mathcal{B}}$)
Let $v', v'' \in \mathcal{V}_{\mathcal{B}}$. We write $v' \nearrow v''$ if there is a sequence
\[ v' = v_{0}, v_{1}, \ldots, v_{n} = v'' \]
where $v_{i+1}$ is a vertex of $st_{\uparrow}(v_{i})$ for $i=0, \ldots, n-1$. A sequence with the latter property is called an \emph{expansion sequence}.
The relation $\nearrow$ is a partial order on $\mathcal{V}_{\mathcal{B}}$, called the \emph{expansion partial order}.
\end{definition}

\begin{lemma} \label{lemma:expchar} 
(Piecewise property of the expansion partial order)
Assume that $v' = \{ b_{1}, \ldots, b_{k} \}$ and 
$v''$ are vertices in $\mathcal{V}_{\mathcal{B}}$. 
\begin{enumerate}
\item If 
$v' = v_{0}, v_{1}, \ldots, v_{n} = v''$
is an expansion sequence, then,  
for each $i \in \{ 1, \ldots, k \}$,
\[ b_{i} = r_{b_{i}}(v_{0}), r_{b_{i}}(v_{1}), \ldots,
r_{b_{i}}(v_{n}) = r_{b_{i}}(v'') \]
is an expansion sequence.

\item If, for $i = 1, \ldots, k$, the sequence
$\displaystyle b_{i} = v_{i0}, v_{i1}, \ldots, v_{in} = r_{b_{i}}(v'')$
is an expansion sequence, then 
$\displaystyle v' = v_{0}, \ldots, v_{n} = v''$
is an expansion sequence, where $v_{j} = \cup_{i=1}^{k} v_{ij}$.
\end{enumerate}
\end{lemma}

\begin{proof}
We first prove (1). Let $v' = v_{0}, v_{1}, \ldots, v_{n} = v''$ be an expansion sequence. Letting $i \in \{ 1, \ldots, k \}$ and $j \in \{ 0, \ldots, n-1 \}$ be arbitrary, we must show that $r_{b_{i}}(v_{j+1}) \in
st_{\uparrow}(r_{b_{i}}(v_{j}))$. By Definition \ref{definition:deltaB}(2), the latter is equivalent to showing that 
$r_{b'}(r_{b_{i}}(v_{j+1})) \in \mathcal{E}(b')$, for each 
$b' \in r_{b_{i}}(v_{j})$. We thus let $b' \in r_{b_{i}}(v_{j})$ be arbitrary. Since $v_{j+1} \in st_{\uparrow}(v_{j})$, by our assumption, and $b' \in r_{b_{i}}(v_{j}) \subseteq v_{j}$, Definition \ref{definition:deltaB} implies that $r_{b'}(v_{j+1}) \in \mathcal{E}(b')$. Now since 
$supp(b') \subseteq supp(b_{i})$, $r_{b'}(r_{b_{i}}(v_{j+1}) = r_{b'}(v_{j+1})$, and the proof of (1) is complete.

Assume that we are given expansion sequences 
$b_{i} = v_{i0}, v_{i1}, \ldots, v_{in} = r_{b_{i}}(v'')$ for each $i \in \{ 1, \ldots, k \}$. We wish to show that $v' = v_{0}, \ldots, v_{n} = v''$ is an expansion sequence, where the $v_{j}$ are as defined in the lemma. Let $j \in \{ 0, \ldots, n-1 \}$ be arbitrary, and let $b' \in v_{j}$. We must show that $r_{b'}(v_{j+1}) \in \mathcal{E}(b')$. By the definition of $v_{j}$, $b' \in v_{ij}$, for some $i$. Our assumption then  implies that $r_{b'}(v_{i(j+1)}) \in \mathcal{E}(b')$. Note that $v_{i(j+1)} = r_{b_{i}}(v_{j+1})$ and $supp(b') \subseteq supp(b_{i})$ (by the assumption that $b' \in v_{ij}$). 
Thus, $r_{b'}(v_{j+1}) = r_{b'}(r_{b_{i}}(v_{j+1})) = r_{b'}(v_{i(j+1)}) \in \mathcal{E}(b')$, as desired.
\end{proof}

\begin{example} \label{example:piecewise}
Figure \ref{expset} offers a schematic picture of vertices in $\Delta_{\mathcal{B}}$, and specifically illustrates Lemma \ref{lemma:expchar}. On the left, we have an expansion sequence 
\[ v_{0}, v_{1}, v_{2}, v_{3}. \]
These are the vertices 
\[ \{ b_{1}, b_{2}, \}, \{ b_{1}, b_{3}, b_{4}, \}, \{ b_{3}, b_{5}, b_{6}, b_{7}, b_{8} \}, \{ \hat{b}, \tilde{b}, b_{3}, b_{6}, b_{7}, b_{8} \}, \]
 respectively. Thus, each row in the left-hand rectangle represents a vertex in $\Delta_{\mathcal{B}}$.
 The bottom line, labelled ``$X$", represents the set $X$ in the definition of the expansion set 
 $\mathcal{B}$ (Definition \ref{definition:expansionsets}). Each rectangle represents an individual member of $\mathcal{B}$. The projection of a given labelled rectangle onto the $X$-axis determines the support of that rectangle's label. Our representation of elements of $\mathcal{B}$ by labelled rectangles is exactly as in Example \ref{example:figure}.
 
 
 \begin{figure}[!b]
\begin{center}
\begin{tikzpicture}
\node at (4,1.5){\huge $\leftrightarrow$};

\draw[black](0,0) rectangle (1.33,3); 
\draw[black](0,0) rectangle (1.33,.75);
\node at (.665, .375){\small $b_{1}$};
\node at (.665,1.125){\small $b_{1}$};
\node at (.4,1.875){\small $b_{5}$};
\node at (1.065,1.875){\small $b_{6}$};
\node at (1.065,2.625){\small $b_{6}$};
\node at (.2,2.625){\small $\hat{b}$};
\node at (.6,2.625){\small $\tilde{b}$};
\node at (-.4,.375){\small $v_{0}$};
\node at (-.4,1.125){\small $v_{1}$};
\node at (-.4,1.875){\small $v_{2}$};
\node at (-.4,2.625){\small $v_{3}$};
\draw[black](0,0) rectangle (1.33,1.5);
\draw[black](.8,1.5) rectangle (1.33,2.25);
\draw[black](0,1.5) rectangle (.8,2.25);
\draw[black](0,2.25) rectangle (.8,3);
\draw[black](0,2.25) rectangle (.4,3);

\draw[black](1.33,0) rectangle (3,.75); 
\node at (2.165,.375){\small $b_{2}$};
\node at (1.605,1.125){\small $b_{3}$};
\node at (1.605,1.875){\small $b_{3}$};
\node at (1.605,2.625){\small $b_{3}$};
\node at (2.44,1.125){\small $b_{4}$};
\node at (2.16,1.875){\small $b_{7}$};
\node at (2.16,2.625){\small $b_{7}$};
\node at (2.715,1.875){\small $b_{8}$};
\node at (2.715,2.625){\small $b_{8}$};
\draw[black](1.33,.75) rectangle (1.88,1.5);
\draw[black](1.33,0) rectangle (3,3);
\draw[black](1.88,.75) rectangle (3,1.5);
\draw[black](1.33,1.5) rectangle (3,2.25);
\draw[black](1.33,1.5) rectangle (1.88,3);
\draw[black](1.33,1.5) rectangle (2.43,3);

\draw[black,dotted](0,0)--(0,-.5);
\draw[black,dotted](1.33,0)--(1.33,-.5);
\draw[black,dotted](3,0)--(3,-.5);

\draw[black, ->](-.2,-.5) -- (3.2,-.5); 
\node at (1.5,-.8){$X$};

\draw[black](5.5,0) rectangle (6.83,3); 
\draw[black](5.5,0) rectangle (6.83,.75); 
\draw[black](5.5,0) rectangle (6.83,1.5); 
\draw[black](6.3,1.5) rectangle (6.83,2.25);
\draw[black](5.5,1.5) rectangle (6.3,2.25);
\draw[black](5.5,2.25) rectangle (6.3,3);
\draw[black](5.5,2.25) rectangle (5.9,3);
\node at (6.165, .375){\small $b_{1}$};
\node at (6.165,1.125){\small $b_{1}$};
\node at (5.9,1.875){\small $b_{5}$};
\node at (6.565,1.875){\small $b_{6}$};
\node at (6.565,2.625){\small $b_{6}$};
\node at (5.7,2.625){\small $\hat{b}$};
\node at (6.1,2.625){\small $\tilde{b}$};
\node at (5.1,.375){\small $v_{10}$};
\node at (5.1,1.125){\small $v_{11}$};
\node at (5.1,1.875){\small $v_{12}$};
\node at (5.1,2.625){\small $v_{13}$};

\draw[black,dotted](5.5,0)--(5.5,-.5);
\draw[black,dotted](6.8,0)--(6.8,-.5);
\draw[black,dotted](7.33,0)--(6.8,-.5);
\draw[black,dotted](9,0)--(8.47,-.5);

\draw[black](7.33,0) rectangle (9,.75); 
\draw[black](7.33,.75) rectangle (7.88,1.5); 
\draw[black](7.33,0) rectangle (9,3); 
\draw[black](7.88,.75) rectangle (9,1.5);
\draw[black](7.33,1.5) rectangle (9,2.25);
\draw[black](7.33,1.5) rectangle (7.88,3);
\draw[black](7.33,1.5) rectangle (8.43,3);
\node at (8.165,.375){\small $b_{2}$};
\node at (7.605,1.125){\small $b_{3}$};
\node at (7.605,1.875){\small $b_{3}$};
\node at (7.605,2.625){\small $b_{3}$};
\node at (8.44,1.125){\small $b_{4}$};
\node at (8.16,1.875){\small $b_{7}$};
\node at (8.16,2.625){\small $b_{7}$};
\node at (8.715,1.875){\small $b_{8}$};
\node at (8.715,2.625){\small $b_{8}$};
\node at (9.4,.375){\small $v_{20}$};
\node at (9.4,1.125){\small $v_{21}$};
\node at (9.4,1.875){\small $v_{22}$};
\node at (9.4,2.625){\small $v_{23}$};

\draw[black, ->](5.3,-.5) -- (8.7,-.5); 
\node at (7,-.8){$X$};

\end{tikzpicture}

\end{center} 

\caption{The figure illustrates how an expansion sequence (on the left) factors naturally into multiple expansion sequences (appearing on the right). The indicated correspondence is reversible.}
\label{expset} 
\end{figure}
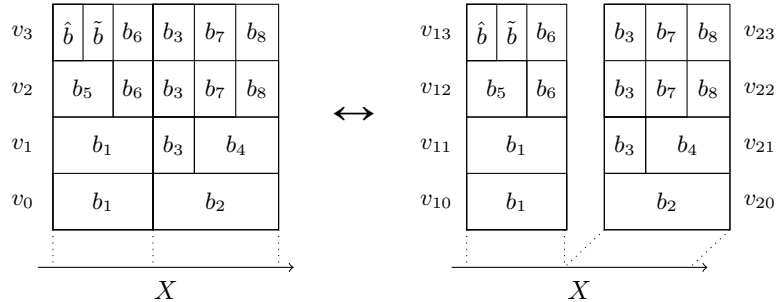

 To say that $v_{0}, v_{1}, v_{2}, v_{3}$ is an expansion sequence means that each vertex is in the upward star of the previous vertex. This means, for instance, that $\{ b_{1} \} \in \mathcal{E}(b_{1})$ and 
 $\{ b_{3}, b_{4} \} \in \mathcal{E}(b_{2})$ because $v_{0}, v_{1}$ is an expansion sequence. (The first inclusion $\{ b_{1} \} \in \mathcal{E}(b_{1})$ is a consequence of Definition \ref{definition:expansionsets}(1), while 
 the inclusion $\{ b_{3}, b_{4} \} \in \mathcal{E}(b_{2})$ is assumed to be part of the definition of $\mathcal{E}$ that is particular to this expansion set.) Note that it is not necessarily the case that
 $\{ b_{3}, b_{7}, b_{8} \} \in \mathcal{E}(b_{2})$; this may be true or not, depending on the specific definition of $\mathcal{E}$. 
 
Lemma \ref{lemma:expchar}(1) says that an expansion sequence $v_{0}, \ldots, v_{n}$ can be split into multiple expansion sequences,
each beginning with $\{ b \}$, for some $b \in v_{0}$. The right half of the figure illustrates this process. The two rectangles from the right half of the figure represent the expansion sequence
\[ r_{b_{i}}(v_{0}), r_{b_{i}}(v_{1}), r_{b_{i}}(v_{2}), r_{b_{i}}(v_{3}), \]
for $i=1,2$, where the first (left-hand) rectangle on the right represents the $i=1$ case, while 
the other represents the $i=2$ case. We note that, for $i=1,2$ and $j=0,1,2,3$, we have
$r_{b_{i}}(v_{j}) = v_{ij}$. Conversely, the expansion sequences on the right half of the diagram can be merged into the single sequence appearing on the left. 

Note that an expansion sequence, geometrically speaking, represents a (weakly) ascending edge-path in $\Delta_{\mathcal{B}}$. (``Weakly" because repetitions are allowed.) It is by no means the case that the vertices in an expansion sequence are the vertices of a simplex in $\Delta_{\mathcal{B}}$. In the figure, $\{ v_{0}, v_{1}, v_{2}, v_{3} \}$
is a simplex of $\Delta_{\mathcal{B}}$ exactly if $v_{10} \leq v_{11} \leq v_{12} \leq v_{13}$ in $\mathcal{E}(b_{1})$
and $v_{20} \leq v_{21} \leq v_{22} \leq v_{23}$ in $\mathcal{E}(b_{2})$ (see Definition \ref{definition:deltaB}(2)). As matters stand, it is not even guaranteed that
$v_{22} \in \mathcal{E}(b_{2})$, for instance (as noted above).  
\end{example}

\begin{definition}
\label{definition:relativeupward}
(Relative ascending stars and links)
Let $v', v''$ be vertices such that $v' \nearrow v''$. We let
$st_{\uparrow}(v', v'')$ be the subcomplex of $st_{\uparrow}(v')$ that is spanned by the vertices $u$ such that $u \nearrow v''$. We let 
$lk_{\uparrow}(v',v'')$ be the link of $v'$ in $st_{\uparrow}(v',v'')$. These are \emph{relative ascending stars} and \emph{links}, respectively.
\end{definition}

\begin{proposition} \label{proposition:relascstarasproduct} 
(Relative ascending stars as products) Let $v = \{ b_{1}, \ldots, b_{k} \}$ be a vertex, and let $v \nearrow v'$. 

The function $p$ from Proposition \ref{proposition:ascstarasproduct}
restricts to an isomorphism 
\[ p_{\mid}: st_{\uparrow}(v,v') \rightarrow
st_{\uparrow}(b_{1}, r_{b_{1}}(v')) \times
\ldots \times
st_{\uparrow}(b_{k}, r_{b_{k}}(v')) \]  of partially ordered sets.
\end{proposition}

\begin{proof}
We will argue that $p_{\mid}$ is well-defined and surjective. The remaining properties will then follow directly from the fact that $p$ is an isomorphism (Proposition \ref{proposition:ascstarasproduct}).

We first prove that $p_{\mid}$ is well-defined; i.e., that the given  codomain of $p_{\mid}$ is correct. Let $u \in st_{\uparrow}(v;v')$. It follows that there is an expansion sequence 
\[ v, u, v_{2}, \ldots, v_{n} = v', \]
for some $n$. Applying Lemma \ref{lemma:expchar}(1), we get expansion sequences 
\[ b_{i}, r_{b_{i}}(u), r_{b_{i}}(v_{2}), \ldots, r_{b_{i}}(v_{n}) = r_{b_{i}}(v') \]
for $i = 1, \ldots, k$. It follows directly that 
$r_{b_{i}}(u) \in st_{\uparrow}(b_{i}, r_{b_{i}}(v'))$, for $i = 1, \ldots, k$, so $p_{\mid}$ is indeed well-defined.

Next, for $i=1, \ldots, k$, assume that 
$u_{i} \in st_{\uparrow}(b_{i},r_{b_{i}}(v'))$. It follows that, for each $i$, we have an expansion sequence
\[ b_{i} = u_{i0}, u_{i} = u_{i1}, u_{i2}, u_{i3}, \ldots, u_{in} = r_{b_{i}}(v'). \]
(We can arrange that all of the sequences have length $n$ simply by padding each to a common length, by repeating terms as necessary.) Letting $v_{j} = u_{1j} \cup \ldots \cup u_{kj}$ for $j=0, \ldots, n$, Lemma \ref{lemma:expchar}(2) gives us an expansion sequence
\[ v = v_{0}, v_{1}, \ldots, v_{n} = v', \]
where $p(v_{1}) = (u_{1}, \ldots, u_{k})$, proving surjectivity. 
\end{proof}

\begin{theorem}
\label{theorem:topologyoflinksandstars}
(Topology of ascending links and stars)
Let $v = \{ b_{1}, \ldots, b_{k} \}$ be a vertex. We have homeomorphisms
\[
|st_{\uparrow}(v)| \cong \prod_{i=1}^{k}|st_{\uparrow}(b_{i})|;  \quad \quad
|lk_{\uparrow}(v)| \cong \bigast_{i=1}^{k}|lk_{\uparrow}(b_{i})|. \]
If $v \nearrow v'$, we additionally have homeomorphisms
\[ |st_{\uparrow}(v,v')| \cong \prod_{i=1}^{k}|st_{\uparrow}(b_{i}, r_{b_{i}}(v')|;  \quad \quad
|lk_{\uparrow}(v,v')| \cong \bigast_{i=1}^{k}|lk_{\uparrow}(b_{i},r_{b_{i}}(v')|, \] 
where all of the spaces in question are given the compactly generated topology. \qed
\end{theorem}

\begin{proof}
The first two homeomorphisms are the result of applying Theorem 5.1 from \cite{Walker} and using
the abstract isomorphism from Proposition \ref{proposition:ascstarasproduct}. Theorem 5.1(a) from \cite{Walker} applies to the star,
while Theorem 5.1(b) applies to the join.

Similarly, the last two homeomorphisms are the result of applying Theorem 5.1 from \cite{Walker} and using the abstract isomorphism from Proposition \ref{proposition:relascstarasproduct}.
\end{proof}

\subsubsection{Topology of partitioned downward links and stars in $\Delta_{\mathcal{B}}$}

\begin{definition}
\label{definition:partitionedlink} 
(Partitioned downward links and stars)
Let $v = \{ b_{1}, \ldots, b_{k} \}$ be a vertex, and let $\mathcal{P} = \{ v_{1}, \ldots, v_{\ell} \}$ be a partition of $v$. We define
$st_{\downarrow}^{\mathcal{P}}(v)$ to be the subcomplex of
$st_{\downarrow}(v)$ spanned by vertices $u = \{ b'_{1}, \ldots, b'_{m} \}$ such that, for each $b'_{i} \in u$, there is some $v_{j} \in \mathcal{P}$ such that $supp(b'_{i}) \subseteq supp(v_{j})$.

The subcomplex $st_{\downarrow}^{\mathcal{P}}(v)$ is the \emph{partitioned downward star of $v$ relative to $\mathcal{P}$}. The link of $v$ in $st_{\downarrow}^{\mathcal{P}}(v)$ is the 
\emph{partitioned downward link of $v$ relative to $\mathcal{P}$}, and is denoted $lk_{\downarrow}^{\mathcal{P}}(v)$. 
\end{definition}

\begin{example} \label{example:quick}
(A discussion of $lk_{\downarrow}^{\mathcal{P}}(v)$ in a simple case)
Let $v = \{ b_{1}, \ldots, b_{k} \}$, and let $\mathcal{P} = \{ \{ b_{1}, b_{2} \},
\{ b_{3}, \ldots, b_{k} \} \}$. A vertex of $lk_{\downarrow}(v)$ is a collection 
$v' = \{ b'_{1}, \ldots, b'_{\ell} \}$ such that $v' \neq v$ and
$r_{b'_{j}}(v) \in \mathcal{E}(b'_{j})$, for each $b'_{j} \in v'$. We might think of $v'$ as arising from $v$ by a finite set of disjoint \emph{contractions}, where the contractions in question replace the subset $r_{b'_{j}}(v)$ with $b'_{j}$. (This is the reverse of ``expansion".)

The vertices of $lk_{\downarrow}^{\mathcal{P}}(v)$ consist of only such $v'$ with the property that each contraction takes place \emph{within} a given member of $\mathcal{P}$,
never straddling more than one member of the partition. For instance, if $v' \in lk_{\downarrow}(v)$ is such that $r_{b'_{j}}(v) = \{ b_{2}, b_{3} \}$, then 
$v' \not \in lk_{\downarrow}^{\mathcal{P}}(v)$. 
\end{example}

\begin{definition} \label{definition:meet}
(The meet of partitions)
If $\mathcal{P}_{1}$ and $\mathcal{P}_{2}$ are partitions of a set $S$, then we define
\[ \mathcal{P}_{1} \wedge \mathcal{P}_{2} = \{ U \cap V \mid
U \cap V \neq \emptyset; U \in \mathcal{P}_{1}; V \in \mathcal{P}_{2} \}. \]
The set $\mathcal{P}_{1} \wedge \mathcal{P}_{2}$ is the \emph{meet} of $\mathcal{P}_{1}$ and $\mathcal{P}_{2}$.
\end{definition}

\begin{proposition} \label{proposition:intersections}
(Intersections of partitioned downward links and stars)
Let $\mathcal{P}_{1}, \ldots, \mathcal{P}_{\gamma}$ be partitions of the vertex $v$. Let $\mathcal{P} = \mathcal{P}_{1} \wedge \ldots \wedge \mathcal{P}_{\gamma}$. We have the following equalities:
\[
\bigcap_{i=1}^{\gamma} lk_{\downarrow}^{\mathcal{P}_{i}}(v)
= lk_{\downarrow}^{\mathcal{P}}(v); \quad \quad
\bigcap_{i=1}^{\gamma} st_{\downarrow}^{\mathcal{P}_{i}}(v)
= st_{\downarrow}^{\mathcal{P}}(v). \]
\end{proposition}

\begin{proof}
We will prove the second equality; the first equality follows easily.

It suffices to check that the given subcomplexes have the same underlying vertex sets.
Let $u = \{ b'_{1}, \ldots, b'_{m} \}$ be a vertex of $st_{\downarrow}^{\mathcal{P}}(v)$. Let $i$ be arbitrary; we show that $u \in st_{\downarrow}^{\mathcal{P}_{i}}(v)$. For an arbitrary $b'_{j} \in u$, there is some $v_{\alpha} \in \mathcal{P}$ such that $supp(b'_{j}) \subseteq supp(v_{\alpha})$. By the definition of meet, 
\[ v_{\alpha} = v_{\alpha_{1}} \cap \ldots \cap v_{\alpha_{\gamma}}, \] where $v_{\alpha_{\ell}} \in \mathcal{P}_{\ell}$, for each $\ell \in \{ 1, \ldots, \gamma \}$. If follows that $supp(b'_{j}) \subseteq supp(v_{\alpha_{i}})$. Thus, $u \in st_{\downarrow}^{\mathcal{P}_{i}}(v)$. This proves the reverse inclusion.

Conversely, suppose that $u = \{ b'_{1}, \ldots, b'_{m} \} \in \cap_{i=1}^{\gamma} st_{\downarrow}^{\mathcal{P}_{i}}(v)$. Let $b'_{j} \in u$; it follows that, for each $\mathcal{P}_{i}$, there is some $v_{\alpha_{i}} \in \mathcal{P}_{i}$ such that $supp(b'_{j}) \subseteq supp(v_{\alpha_{i}})$. Thus, letting $v_{\alpha}$ denote the intersection of the $v_{\alpha_{i}}$, we find that $supp(b'_{j}) \subseteq 
supp(v_{\alpha})$. By the definition of meet, $v_{\alpha} \in \mathcal{P}$, so $u \in st_{\downarrow}^{\mathcal{P}}(v)$.
\end{proof}

\begin{definition} \label{definition:morerestrictions}
(The functions $r_{v_{\alpha}}$)
Let $v \in \mathcal{V}_{\mathcal{B}}$, and let $\mathcal{P} = \{ v_{1}, \ldots, v_{\ell} \}$ be a partition of $v$. For each vertex $u$
in $st_{\downarrow}^{\mathcal{P}}(v)$ and for each $i \in \{1, \ldots, \ell \}$, we set
\[ r_{v_{i}}(u) = \{ b' \in u \mid supp(b') \subseteq supp(v_{i}) \}. \]
\end{definition}

\begin{theorem} \label{theorem:desctopology} 
(Topology of partitioned descending downward links and stars)
Let $\mathcal{P} = \{ v_{1}, \ldots, v_{k} \}$ be a partition of the vertex $v$. We have homeomorphisms
\[
|st_{\downarrow}^{\mathcal{P}}(v)| \cong \prod_{\beta=1}^{k}|st_{\downarrow}(v_{\beta})|;  \quad \quad
|lk_{\downarrow}^{\mathcal{P}}(v)| \cong \bigast_{\beta=1}^{k}|lk_{\downarrow}(v_{\beta})|, \]
provided that all of the spaces in question are given the compactly generated topology.
\end{theorem}

\begin{proof}
We will prove the existence of the homeomorphism between stars in some detail, and then indicate the proof in the case of the links.

For each $\beta \in \{ 1, \ldots, k \}$, we choose a linear order 
on the vertices of $st_{\downarrow}(v_{\beta})$. The linear order can be chosen so that it increases with the height; i.e., if $|w_{1}| < |w_{2}|$, then $w_{1} < w_{2}$ in the linear order. We define the simplicial product
$\prod_{\beta = 1}^{k} st_{\downarrow}(v_{\beta})$ as follows: 
\begin{enumerate}
\item The vertices are $k$-tuples in
$\prod_{\beta = 1}^{k} (st_{\downarrow}(v_{\beta}))^{0}$. 
\item An $n$-simplex is a collection of $n+1$ distinct vertices
$\{ w_{0}, w_{1}, w_{2}, \ldots, w_{n} \}$ satisfying certain additional conditions. We write 
$w_{i} = (u_{i1}, \ldots, u_{ik})$ for $i=0, \ldots, n$. We require that
\begin{enumerate}
\item each collection $\{ u_{0\beta}, u_{1\beta}, \ldots, u_{n\beta} \}$ be a simplex in $st_{\downarrow}(v_{\beta})$, for $\beta = 1, \ldots, k$, and
\item the sequence $u_{0\beta}, u_{1\beta}, \ldots, u_{n\beta}$ be non-decreasing with respect to the chosen linear order, for $\beta = 1, \ldots, k$.
\end{enumerate}
We emphasize that the sequence $u_{0\beta}, \ldots, u_{n\beta}$ may contain repetitions, for some or all $\beta$.
\end{enumerate}

Simplices in $\prod_{\beta=1}^{k} st_{\downarrow}(v_{\beta})$ can be usefully represented by matrices, as in Figure \ref{matrix}.

It is a standard fact that 
\[ \left| \prod_{\beta=1}^{k} L_{\beta} \right| \cong
 \prod_{\beta=1}^{k} |L_{\beta}|, \]
 for any finite collection of simplicial complexes $\{ L_{\beta} \}$, provided that the spaces in question are given the compactly generated topology. Thus, to establish the homeomorphism, it suffices to show that
 $st_{\downarrow}^{\mathcal{P}}(v)$ is simplicially isomorphic to the above simplicial product.
 
 \begin{figure}[!t]
\begin{center}
\begin{tikzpicture}
\draw[black](0,0) rectangle (.5,2.5);
\draw[black](0,0) rectangle (3.5,.5);
\draw[black](0,0) rectangle (3.5,2.5);
\draw[black](.5,0) rectangle (1,2.5);
\draw[black](0,.5) rectangle (3.5,1);
\draw[black](0,2) rectangle (3.5,2.5);
\draw[black](3,0) rectangle (3.5,2.5);
\node at (-.4,.25){\footnotesize $w_{0}$};
\node at (-.4,.75){\footnotesize $w_{1}$};
\node at (-.4,1.6){\large $\vdots$};
\node at (-.4,2.25){\footnotesize $w_{n}$};
\node at (.25,2.8){\footnotesize $v_{1}$};
\node at (.75,2.8){\footnotesize $v_{2}$};
\node at (3.25,2.8){\footnotesize $v_{k}$};
\node at (.25,.25){\footnotesize $u_{01}$};
\node at (.75,.25){\footnotesize $u_{02}$};
\node at (3.25,.25){\footnotesize $u_{0k}$};
\node at (.25,.75){\footnotesize $u_{11}$};
\node at (.75,.75){\footnotesize $u_{12}$};
\node at (3.25,.75){\footnotesize $u_{1k}$};
\node at (.25,2.25){\footnotesize $u_{n1}$};
\node at (.75,2.25){\footnotesize $u_{n2}$};
\node at (3.25,2.25){\footnotesize $u_{nk}$};
\node at (.25,1.6){\large $\vdots$};
\node at (.75,1.6){\large $\vdots$};
\node at (2,1.6){\large $\iddots$};
\node at (3.25,1.6){\large $\vdots$};
\node at (2,.75){\large $\ldots$};
\node at (2,.25){\large $\ldots$};
\node at (2,2.25){\large $\ldots$};
\node at (2,2.8){\large $\ldots$};
\end{tikzpicture}
\end{center}
\caption{The figure illustrates a matrix notation for $n$-simplices in the simplicial product. Each row is a vertex in $\prod_{\beta=1}^{k} st_{\downarrow}(v_{\beta})$. The $\beta$th column is a simplex in $st_{\downarrow}(v_{\beta})$, possibly with repetitions. Height weakly increases as we read a column from bottom to top.}
\label{matrix} 
\end{figure}
 We define a map 
\[ \phi: u \mapsto (r_{v_{1}}(u), r_{v_{2}}(u), \ldots, r_{v_{k}}(u)) \]
on the vertices $st_{\downarrow}^{\mathcal{P}}(v)$. It is straightforward to see that the right-hand expression is a vertex in our simplicial product, and indeed that the map is bijective on vertices (the map that sends a $k$-tuple to the union of its entries is a two-sided inverse). We must show that the map takes simplices to simplices, and that every simplex in the simplicial product is the image of a simplex from
$st_{\downarrow}^{\mathcal{P}}(v)$.

Let $S = \{ u_{0}, \ldots, u_{n} \}$ be an $n$-simplex in $st_{\downarrow}^{\mathcal{P}}(v)$. We will assume that $S$ is in canonical order so that, in particular,
\[ r_{b}(u_{0}) \leq r_{b}(u_{1}) \leq \ldots \leq r_{b}(u_{n}) \]
in $\mathcal{E}(b)$, for each $b \in u_{0}$ (by Definition \ref{definition:deltaB}).
We claim that
\[ r_{v_{\beta}}(S) := \{ r_{v_{\beta}}(u_{0}), r_{v_{\beta}}(u_{1}), \ldots, r_{v_{\beta}}(u_{n}) \} \] is a simplex in $st_{\downarrow}(v_{\beta})$, for $\beta = 1, \ldots, k$. Indeed, let $b \in r_{v_{\beta}}(u_{0})$. It follows that  
$b \in u_{0}$, so that $r_{b}(u_{0}) \leq \ldots, r_{b}(u_{n})$ in 
$\mathcal{E}(b)$, as noted above. Now note that, since $supp(b) \subseteq supp(v_{\beta})$, 
$r_{b}(r_{v_{\beta}}(u_{i})) = r_{b}(u_{i})$ for each $i$. Thus, 
\[ r_{b}(r_{v_{\beta}}(u_{0})) \leq r_{b}(r_{v_{\beta}}(u_{1})) \leq \ldots \leq r_{b}(r_{v_{\beta}}(u_{n})) \]
in $\mathcal{E}(b)$,
for all $b \in r_{v_{\beta}}(u_{0})$.
 The definition of $\Delta_{\mathcal{B}}$ (Definition \ref{definition:deltaB})
now implies that $r_{v_{\beta}}(S)$ is a simplex in $\Delta_{\mathcal{B}}$, for $\beta = 1, \ldots, k$. The fact that 
$r_{v_{\beta}}(S)$ is  a simplex in $st_{\downarrow}(v_{\beta})$ also follows easily, because we can apply essentially the same reasoning to
$r_{v_{\beta}}(S \cup \{ v \})$ (if $v \not \in S$) or to 
$r_{v_{\beta}}(S)$ (if $v = u_{n}$).

We must also check that the sequence
\[ r_{v_{\beta}}(u_{0}), \ldots, r_{v_{\beta}}(u_{n}) \]
is non-decreasing with respect to the linear order chosen above. This follows from the fact that height either increases as we move from term to term in this sequence (which implies also an increase in the linear order, by the hypothesis), or the sequence members repeat. This proves that the given sequence is non-decreasing with respect to the chosen linear order.

It follows from all of this that the map $\phi$ sends simplices to simplices. The fact that each simplex in the product arises as the image of a simplex under $\phi$ can be derived from the matrix notation for the simplices in the product: given such a matrix representing a simplex $\sigma$ in the product, we let the union of the $i$th row be $w_{i}$. The set $\{ w_{0}, \ldots, w_{n}\}$ is the desired simplex $\sigma' \in st_{\downarrow}^{\mathcal{P}}(v)$ such that
$\phi(\sigma') = \sigma$. Thus, $\phi$ is an isomorphism of simplicial complexes, establishing the first homeomorphism.

In the case of the links, we will indicate how to define a homeomorphism between the realizations of $lk_{\downarrow}^{\mathcal{P}}(v)$  and $\bigast_{\beta=1}^{k}lk_{\downarrow}(v_{\beta})$. We first recall the definition of the abstract join, as it specializes to the current case:
\begin{itemize}
\item $\mathcal{V}\left(\bigast_{\beta=1}^{k} lk_{\downarrow}(v_{\beta})\right) = \bigcup_{\beta=1}^{k} \mathcal{V}(lk_{\downarrow}(v_{\beta}))$;
\item $\mathcal{S}\left(\bigast_{\beta=1}^{k} lk_{\downarrow}(v_{\beta})\right) = 
\left\{ S_{1} \cup \ldots \cup S_{k} \mid S_{\beta} = \emptyset \text{ or } S_{\beta} \in 
\mathcal{S}(lk_{\downarrow}(v_{\beta})) \right\} - \{ \emptyset \}$.
\end{itemize}
We define a map from the realization of $lk_{\downarrow}^{\mathcal{P}}(v)$ to $\bigast_{\beta = 1}^{k} lk_{\downarrow}(v_{\beta})$ by sending a vertex of $lk_{\downarrow}^{\mathcal{P}}(v)$ to a certain weighted sum of vertices from the target. We must first introduce some notation.
We let $r'_{v_{\beta}}(u) = r_{v_{\beta}}(u)$ if the latter is not $v_{\beta}$, and $r'_{v_{\beta}}(u) = 0$ if $r_{v_{\beta}}(u) = v_{\beta}$.
We let $wgt(u)$ denote the number of non-zero terms in the set
\[ \{ r'_{v_{\beta}}(u) \mid \beta = 1, \ldots, k \}. \]
With these definitions, the desired map is
\[ u \mapsto \frac{1}{wgt(u)} \sum_{\beta = 1}^{k} r'_{v_{\beta}}(u) \]
on vertices, and we extend linearly over simplices of 
$lk_{\downarrow}^{\mathcal{P}}(v)$. The resulting map is a \emph{subdivision map}, in the terminology of \cite{Walker} (pg. 99), and thus determines a homeomorphism between the simplicial realizations.
\end{proof}

\subsection{A criterion for $n$-connectivity of $\Delta^{f}_{\mathcal{B}}$}
\label{subsection:nconncrit}

\begin{lemma} \label{lemma:nconnectedness} (A sufficient condition for $n$-connectedness)
Assume that any two vertices in $\Delta^{f}_{\mathcal{B}}$ have a common upper bound with respect to $\nearrow$ (i.e., 
$((\Delta^{f}_{\mathcal{B}})^{0}, \nearrow)$ is a directed set). 
If whenever $v' \nearrow v''$ ($v' \neq v''$; $v', v'' \in \Delta^{f}_{\mathcal{B}}$), $lk_{\uparrow}(v',v'')$ is
$(n-1)$-connected, then $\Delta^{f}_{\mathcal{B}}$ is $n$-connected.
\end{lemma}

\begin{proof}
This is essentially Lemma 2.6 from \cite{FH2}, but with minor differences in notation. We will therefore only sketch the argument and direct the interested reader to \cite{FH2} for details.

We let $f: S^{k} \rightarrow \Delta^{f}_{\mathcal{B}}$ $(k \leq n)$ be a continuous map. The image $f(S^{k})$ is compact, and thus lies inside a finite subcomplex $K$ of $\Delta^{f}_{\mathcal{B}}$. 
It follows from our assumptions that the vertices of $K$ have a common upper bound $v''$ under the partial order $\nearrow$. The hypothesis on the connectivity 
of the relative ascending links $lk_{\uparrow}(v',v'')$ assures that we can always homotope $f$ upward, toward the common upper bound $v''$. After finitely many steps (by finiteness of the height of $v''$), we succeed at homotoping $f$ to the constant map at $v''$. Thus, $f$ is null-homotopic. 

\end{proof}

\begin{definition} \label{definition:nconnectedE}
($n$-connected expansion sets) We say that the expansion set 
$\mathcal{B}$ is \emph{$\mathit{n}$-connected} if
\begin{enumerate}
\item for every $b \in \mathcal{B}$ and $v \in \mathcal{V}$ such that $\{ b \} \, \nearrow \, v$, 
$lk_{\uparrow}(b,v)$ is $(n-1)$-connected;
\item the vertices of $\Delta^{f}_{\mathcal{B}}$ are a directed set with respect to the partial order $\nearrow$.
\end{enumerate}
\end{definition}

\begin{lemma} \label{lemma:joincon}
(The connectivity of joins)
Let $K_{1}, \ldots, K_{m}$ be a family of simplicial complexes. Suppose that each $K_{j}$ is $n_{j}$-connected ($n_{j} \geq -1$; ``$-1$-connected" means ``non-empty"). The join
\[ K_{1} \ast \ldots \ast K_{m} \]
is $(n_{1} + \ldots + n_{m} + 2m-2)$-connected. \qed
\end{lemma}

\begin{theorem}
\label{theorem:nconnectivity}
($n$-connectivity of $\Delta^{f}_{\mathcal{B}}$)
If $\mathcal{B}$ is an $n$-connected expansion set, then
$\Delta^{f}_{\mathcal{B}}$ is $n$-connected. If $\mathcal{B}$ is $n$-connected for all $n$, then $\Delta^{f}_{\mathcal{B}}$ is contractible.
\end{theorem}

\begin{proof}
In view of Lemma \ref{lemma:nconnectedness}, it suffices to show that
$lk_{\uparrow}(v_{1}, v_{2})$ is $(n-1)$-connected, for every 
pair of vertices $v_{1}, v_{2} \in \Delta^{f}_{\mathcal{B}}$ such that $v_{1} \nearrow v_{2}$ and $v_{1} \neq v_{2}$.

We assume that $v_{1} = \{ b_{1}, \ldots, b_{k} \}$ and $v_{2} = v'_{1} \cup \ldots \cup v'_{k}$, where $supp(b_{i}) = supp(v'_{i})$ for $i = 1, \ldots, k$. (Note that $P(v_{2})$ is a proper refinement of $P(v_{1})$ by the definition of $\nearrow$, so $v_{1}$ and $v_{2}$ necessarily take the above forms.) By Theorem \ref{theorem:topologyoflinksandstars}, we have a homeomorphism
\[ |lk_{\uparrow}(v_{1}, v_{2})|
\cong \bigast_{\beta=1}^{k} |lk_{\uparrow}(b_{\beta}, v'_{\beta})|. \]
By our assumptions, each factor on the right is either empty or $(n-1)$-connected, and at least one of the factors is $(n-1)$-connected. It follows from Lemma \ref{lemma:joincon} that $|lk_{\uparrow}(v_{1},v_{2})|$ is $(n-1)$-connected, completing the proof of the first statement.

The final statement follows from an application of Whitehead's Theorem. 
\end{proof}

\subsection{The natural filtration of $\Delta^{f}_{\mathcal{B}}$}
\label{subsection:filtration}

\begin{definition} \label{definition:thefiltration} (Finite-height subcomplexes of $\Delta^{f}_{\mathcal{B}}$)
For $n \in \mathbb{N}$, we let $\Delta^{f}_{\mathcal{B},n}$ denote the subcomplex of 
$\Delta^{f}_{\mathcal{B}}$ spanned by vertices $v$ of height no more than $n$. 
\end{definition}

\begin{lemma} \label{lemma:filtrationlemma} (The sequence $\{ \Delta^{f}_{\mathcal{B},n} \}$ as a filtration of $\Delta^{f}_{\mathcal{B}}$)
For each $n \in \mathbb{N}$, $\Delta^{f}_{\mathcal{B},n}$ is a finite-dimensional complex. We have
\[ \Delta^{f}_{\mathcal{B},1} \subseteq \Delta^{f}_{\mathcal{B},2} \subseteq
\ldots \subseteq \Delta^{f}_{\mathcal{B}}, \]
and 
\[ \Delta^{f}_{\mathcal{B}} = \bigcup_{n=1}^{\infty} \Delta^{f}_{\mathcal{B},n}. \]
\end{lemma}

\begin{proof}
The chain of inclusions $\Delta^{f}_{\mathcal{B},1} \subseteq \Delta^{f}_{\mathcal{B},2} \subseteq \ldots \Delta^{f}_{\mathcal{B}}$ is clear.

We first show that $\Delta^{f}_{\mathcal{B},n}$ is finite-dimensional.  Let $\{ v_{0}, \ldots, v_{k} \}$ be a simplex in $\Delta^{f}_{\mathcal{B},n}$. It must be that $|v_{i}| \neq |v_{j}|$, for $i \neq j$. Since each $|v_{i}|$ is a positive integer, we must have that $k \leq n-1$ by the Pigeonhole Principle; thus the dimension of a simplex in $\Delta^{f}_{\mathcal{B},n}$ is at most $n-1$.

Finally, note that if $\{ v_{0}, \ldots, v_{k} \}$ is a simplex in $\Delta^{f}_{\mathcal{B}}$, then there is some subscript $i$ such that $|v_{i}|$ is a maximum. Assuming that $|v_{i}| = m$, we have that
$\{ v_{0}, \ldots, v_{k} \}$ is a simplex of $\Delta^{f}_{\mathcal{B},m}$. It follows directly that 
$\Delta^{f}_{\mathcal{B}} \subseteq \cup_{n=1}^{\infty} \Delta^{f}_{\mathcal{B},n}$; the reverse inclusion is obvious.
\end{proof}

\begin{proposition} \label{proposition:connectfilt}
(Connectivity of the filtration in terms of descending links) Assume that $\Delta^{f}_{\mathcal{B}}$ is $n$-connected.
If the descending links $lk_{\downarrow}(v)$ of all vertices $v \in \Delta^{f}_{\mathcal{B}}$ of height at least $k$ are $n$-connected, then the maps on homotopy groups
\[ \pi_{j}(\Delta^{f}_{\mathcal{B},\ell - 1}) \rightarrow \pi_{j}(\Delta^{f}_{\mathcal{B}, \ell}) \] 
are isomorphisms, for $j=0, \ldots, n$ and $\ell \geq k$. 

In particular, $\Delta^{f}_{\mathcal{B},\ell-1}$ is $n$-connected, for $\ell \geq k$.
\end{proposition}

\begin{proof}
We omit the proof, which has become standard. The interested reader should see Proposition 7.6 from \cite{FH2}, or consult \cite{BB}, which is a ``classical" source for this observation.

\end{proof}

\subsection{Connectivity of the descending link}
\label{subsection:threetypes}
In this subsection, we begin to specialize the discussion somewhat. We introduce two natural subtypes of expansion set, which we call linear and cyclic (respectively). These correspond to groups of ``$F$"- and ``$T$"-type, respectively. We also give a rather general discussion of the descending link in the case that $\mathcal{B}$ is ``rich in contractions" and has the ``bounded contractions" property. In the presence of the latter properties, we show that the connectivity of $lk_{\downarrow}(v)$ tends to infinity with the height of $v$ (Corollary \ref{corollary:connectivityeasy}). In view of Proposition \ref{proposition:connectfilt}, this is the crucial step in proving the connectivity of the filtration $\{ \Delta^{f}_{\mathcal{B},n} \}$. 

\subsubsection{Two types of expansion set} 

\begin{definition} \label{definition:3types}
(Linear and cyclic expansion sets; consecutively ordered)
We say that the expansion set $\mathcal{B}$ is \emph{linear}
if:
\begin{enumerate}
\item [(L1)] the set $X$ is linearly ordered, and
\item [(L2)] for each $b \in \mathcal{B}$, $supp(b)$ is an interval of $X$.
\end{enumerate}
We say that the expansion set $\mathcal{B}$ is \emph{cyclic} if:

\begin{enumerate}
\item [(C1)] the set $X$ is cyclically ordered, and
\item [(C2)] for each $b \in \mathcal{B}$, $supp(b)$ is a (cyclic) interval of $X$.
\end{enumerate}
If $\mathcal{B}$ is linear (respectively, cyclic) and $v = \{ b_{1}, \ldots, b_{k} \} \in \mathcal{V}_{\mathcal{B}}$, we say that $v$ is \emph{consecutively ordered} if $supp(v)$ is an interval (respectively, an interval in the cyclic sense).
We say that $v$ is in \emph{linear order} (respectively, in \emph{cyclic order}) if 
$supp(b_{1}), supp(b_{2}), \ldots, supp(b_{k})$ are ordered from left to right, or if the same supports are ordered cyclically (respectively).

\begin{remark}
(Ordered and consecutively ordered)
Suppose that $\mathbb{B}$ is a linear expansion set over $\mathbb{R}$ and 
$\{ b_{i} \mid i \in \mathbb{Z}\} \subseteq \mathbb{B}$ is such that
$supp(b_{i}) = [i,i+1)$.

The vertex $v = \{ b_{1}, b_{2}, b_{3} \}$ is consecutively ordered ($supp(v) = [1,4)$), while $\{ b_{1}, b_{3} \}$
is not
(since its support is not an interval). We note that, in contrast, the definition of ``linearly ordered" simply describes a convention for listing members of a vertex $v$: if they are listed ``from left to right" (in a suitable sense), then the vertex is linearly ordered, but not otherwise. Thus, the vertex $v = \{ b_{1}, b_{3}, b_{5} \}$ is linearly ordered (since the supports are ordered from left to right), but $v = \{ b_{1}, b_{5}, b_{3} \}$ (the same vertex) is not. 
\end{remark}

\end{definition}

\begin{remark} (``Permutational" expansion sets) 
\label{remark:permutational}
It appears difficult to give a direct definition of ``permutational" expansion sets, so we will avoid doing so here. (These would be the expansion sets that model ``$V$''-like groups.) We will instead capture the ``$V$"-like nature of the expansion sets using other definitions  (like Definition \ref{definition:rich}, or the explicit definitions of the expansion sets given in Section \ref{section:construction}). 
\end{remark}

\subsubsection{The bounded and rich-in-contractions property}

\begin{definition} \label{definition:boundedcontractions}
(The bounded contractions property)
We say that the expansion set $\mathcal{B}$ has the \emph{bounded contractions property} if 
\[ \{ |v| \mid v \in \mathcal{E}(b), \text{ for some } b \in \mathcal{B} \} \]
is a bounded set of positive integers. 

In the event that the above set is bounded, our custom will be to let $C_{0}$ denote an upper bound.
\end{definition}

\begin{definition} \label{definition:rich}
(Rich in contractions) Let $\mathcal{B}$ be an expansion set. 
\begin{enumerate}
\item If $\mathcal{B}$ is linear or cyclic, we say that $\mathcal{B}$ is \emph{rich in contractions} if there is a constant $C_{1}$ such that, for every consecutively ordered $v \in \mathcal{V}_{\mathcal{B}}$ satisfying
$|v| \geq C_{1}$, $lk_{\downarrow}(v)$ is non-empty.

We may also say, for emphasis, that $\mathcal{B}$ is \emph{rich in contractions in the linear sense (respectively, the cyclic sense)} in the above cases.
\item We say that $\mathcal{B}$ is \emph{rich in contractions in the permutational sense} if there is a constant $C_{1}$ such that, for every $v \in \mathcal{V}_{\mathcal{B}}$ satisfying $|v| \geq C_{1}$, $lk_{\downarrow}(v)$ is non-empty.

We may also say that $\mathcal{B}$ is \emph{rich in contractions}, if ``permutational" is understood from context.
\end{enumerate}
\end{definition}

\subsubsection{Generalities about the downward link}

\begin{remark}(Conventions and vocabulary)
\label{remark:convvoc}
Let $\{ u, v\}$ be a $1$-simplex in $\Delta_{\mathcal{B}}$. Assume that $\{ u, v \}$ is written in canonical order; i.e., that $v$ is a vertex in $st_{\uparrow}(u)$ and $u$ is a vertex in $lk_{\downarrow}(v)$. Let $u = \{ b'_{1}, \ldots, b'_{\ell} \}$ and $v = \{ b_{1}, \ldots, b_{k} \}$. We assume that $u$ and $v$ are written in either the linear or cyclic ordering if 
$\mathcal{B}$ is linear or cyclic, respectively. We will need to develop a bit of vocabulary that will be useful in the forthcoming analysis of the descending link.

For each $b'_{i} \in u$, $r_{b'_{i}}(v)$ is a subset of $v$, which we call the \emph{production of $b'_{i}$ in $v$}, or the \emph{production of $b'_{i}$} if the $v$ is understood. The production of $b'_{i}$ is \emph{trivial} if $r_{b'_{i}}(v) = \{ b'_{i} \}$; otherwise, it is \emph{non-trivial}. Since $u \neq v$, there is at least one $b'_{i}$ with non-trivial production. 

Due to the fact that $supp(b)$ is always an interval when $\mathcal{B}$ is linear or cyclic (by Definition \ref{definition:3types}), the production of each $b'_{i}$ is always consecutively ordered (in the appropriate sense) for such $\mathcal{B}$. (We note, though, that the $supp(b)$ is essentially arbitrary when $\mathcal{B}$ is not assumed to be linear or cyclic.)

Finally, we make the following observation. For a canonically-ordered simplex $\sigma = \{ u_{0}, \ldots, u_{n} \}$ in $lk_{\downarrow}(v)$, a necessary and sufficient condition that $\sigma$ be a simplex
in $lk_{\downarrow}^{\mathcal{P}}(v)$, where $\mathcal{P} = \{ v_{1}, \ldots, v_{k} \}$, is that
$u_{0}$ be a vertex in $lk_{\downarrow}^{\mathcal{P}}(v)$; i.e., that the production of
each member of $u_{0}$ be a subset of one of the $v_{j}$. This is because each
$P(u_{i})$ (for $i>0$) properly refines $P(u_{0})$, by Definition \ref{definition:deltaB} and 
Definition \ref{definition:expansionsets}(2). 
\end{remark}

\subsubsection{Connectivity of the downward link}

We now proceed to an analysis of the connectivity of the downward link. Our approach is to use nerves of covers (Definition \ref{definition:nerve}) and the Nerve Theorem (Theorem \ref{theorem:Nerve}). 

We note that the bounded contractions property and rich-in-contractions property (Definitions \ref{definition:boundedcontractions} and \ref{definition:rich}) reduce the entire discussion to a fairly simple matter regarding the combinatorics of finite sets, despite any impression of technicality to the contrary.

\begin{definition} \label{definition:nerve} (Nerve of a cover) Let $S$ be any set, and let $\mathcal{C}$ be a cover of $S$. We let $\mathcal{N}(\mathcal{C})$ denote the \emph{nerve of the cover $\mathcal{C}$}. The vertices of $\mathcal{N}(\mathcal{C})$ are the non-empty elements $C \in \mathcal{C}$. A subset $\{ C_{0}, \ldots, C_{k} \} \subseteq \mathcal{C}$ is a simplex if and only if $C_{0} \cap \ldots \cap C_{k} \neq \emptyset$.
\end{definition}

\begin{theorem} \label{theorem:Nerve} \cite{AB} (Nerve Theorem) Let $K$ be a simplicial complex and let 
$\{ K_{i} \}_{i \in I}$ be a family of subcomplexes of $K$ such that $K = \bigcup_{i \in I} K_{i}$. If every non-empty finite intersection $K_{i_{1}} \cap \ldots \cap K_{i_{t}}$ is  $(k-t+1)$-connected, then $K$ is $k$-connected if and only if $\mathcal{N}(\{K_{i}\}_{i\in I})$ is $k$-connected. \qed
\end{theorem}


\begin{definition} \label{definition:standardpartitions}
(Standard partitions) Let $v = \{ b_{1}, \ldots, b_{k} \}$. If $\mathcal{B}$ is linear or cyclic, then we assume that $v$ is in linear or cyclic order.
\begin{enumerate}
\item Assume that $i,j \in \{1, \ldots, k \}$.
If $1 \leq i < j \leq k$, we let 
\[ \mathcal{P}_{[i,j]} = \{ \{ b_{i}, b_{i+1}, \ldots, b_{j} \},  
v - \{ b_{i}, b_{i+1}, \ldots, b_{j} \} \}. \] If $j < i$, we let
\[ \mathcal{P}_{[i,j]} = \{ \{ b_{i}, \ldots, b_{k}, b_{1}, \ldots, b_{j} \}, v- \{ b_{i}, \ldots, b_{k}, b_{1}, \ldots, b_{j} \} \}. \]
The partition $\mathcal{P}_{[i,j]}$ is called a \emph{standard linear partition} if $i<j$; we will refer to all of the above collectively as \emph{standard cyclic partitions}.
In the first of the above cases ($i<j$), we say that $\{ b_{i}, \ldots, b_{j} \}$ is the \emph{principal interval} of $\mathcal{P}_{[i,j]}$. If $j<i$, then 
$\{ b_{i}, \ldots, b_{k}, b_{1}, \ldots, b_{j} \}$ is 
the \emph{principal interval} of $\mathcal{P}_{[i,j]}$. The \emph{length} of the principal interval of $\mathcal{P}_{[i,j]}$ is its cardinality, which is $j-i+1$ if $i<j$ and $k+j-i+1$ if $j<i$. The \emph{initial symbol} of $\mathcal{P}_{[i,j]}$ is $i$.

\item If $T \subseteq \{ 1, \ldots, k \}$, we let
\[ \mathcal{P}_{T} = \{ \{ b_{\ell} \mid \ell \in T \}, \{ b_{\ell} \mid \ell \not \in T \} \}. \]
We say that $\mathcal{P}_{T}$ is a \emph{standard partition} with \emph{principal set} $T$. The \emph{length} of $T$ is $|T|$. 
\end{enumerate}
\end{definition}

\begin{definition} (Standard covers)\label{definition:standardcovers}
Let $v = \{ b_{1}, \ldots, b_{k} \}$ be a vertex. If $\mathcal{B}$ is linear or cyclic, then we further assume that $v$ is both consecutively ordered and ordered in the appropriate sense (i.e., linearly or cyclically, respectively).
Let $C_{0}$ and $C_{1}$ be positive integers.
\begin{enumerate}
\item We let $\mathcal{C}^{\ell}(v,C_{0},C_{1})$
denote the set of all standard linear partitions of $v$ having principal intervals of length $C_{0}$ or less, and  initial symbols no more than $C_{1}$. 
\item We let $\mathcal{C}^{c}(v,C_{0},C_{1})$ denote the set of all standard cyclic partitions $\mathcal{P}$ of $v$ such that
\begin{enumerate}
\item the principal interval of $\mathcal{P}$ is of length $C_{0}$ or less, and
\item either $b_{1}$ is in the principal interval of $\mathcal{P}$, or the 
initial symbol of $\mathcal{P}$ is no more than $C_{1}$.
\end{enumerate} 
\item We let $\mathcal{C}(v,C_{0})$ denote the set of all standard partitions of $v$ having principal sets with length $C_{0}$ or less.
\end{enumerate}
\end{definition}

\begin{proposition} \label{proposition:naturalcovers} 
(Covers of the downward links)
Let $\mathcal{B}$ be an expansion set having both the bounded contractions property and the rich-in-contractions property, with constants $C_{0}$ and $C_{1}$, respectively, where $C_{0} \geq C_{1}$.

Let $v = \{ b_{1}, \ldots, b_{k} \}$ be a vertex; assume $k > C_{0}$. If $\mathcal{B}$ is linear or cyclic, then we further assume that $v$ is both consecutively ordered and ordered in the appropriate sense (i.e., linearly or cyclically, respectively).

We write $\mathcal{C}$ in place of $\mathcal{C}^{\ell}(v,C_{0},C_{1})$, $\mathcal{C}^{c}(v,C_{0},C_{1})$, or $\mathcal{C}(v,C_{0})$, according to whether $\mathcal{B}$ is linear or cyclic, or satisfies the rich-in-contractions property in the permutational sense, respectively.

Under these assumptions, $lk_{\downarrow}(v)$ is covered by 
\[ \{ lk_{\downarrow}^{\mathcal{P}}(v) \mid \mathcal{P} \in \mathcal{C} \}. \]
\end{proposition}

\begin{proof}
We first assume that $\mathcal{B}$ is linear, and that $\mathcal{B}$ satisfies the bounded contractions property and the
rich-in-contractions property, the latter in the linear sense. We let $C_{0}$ and $C_{1}$ respectively denote the associated constants. We now show that the set from the statement of the proposition covers $\lk_{\downarrow}(v)$.

Let $\sigma = \{ u_{0}, u_{1}, \ldots, u_{n} \}$ be a simplex in $lk_{\downarrow}(v)$, written in canonical order (so that 
$\sigma$ is a simplex in the ascending star of $u_{0}$, and indeed $\{ u_{0}, v \}$ is a simplex in the ascending star of $u_{0}$). Assume that $u_{0} = \{ b'_{1}, \ldots, b'_{\ell} \}$ is consecutively and linearly ordered. By the definition of $\Delta_{\mathcal{B}}$, the vertex $v$ 
may be obtained by replacing each of the $b'_{i}$ with the members of some vertex $\hat{u}_{i} \in \mathcal{E}(b'_{i})$.
If $\hat{u}_{i} = \{ b'_{i} \}$, then the production of $b'_{i}$ is trivial (in the terminology of Remark \ref{remark:convvoc}). 
We let $b'_{j}$ be the first member of $u_{0}$ with non-trivial production. It follows that
\[ b_{1} = b'_{1}; \quad b_{2} = b'_{2}; \quad \ldots \quad b_{j-1} = b'_{j-1}, \] 
while 
\[ \{ b_{j}, \ldots, b_{j+d} \} \in \mathcal{E}(b'_{j}), \]
where the height $d+1$ of the latter vertex is no more than $C_{0}$, and $d$ is at least $1$.

There are two cases: either $j \leq C_{1}$, or $j > C_{1}$. If $j \leq C_{1}$, then we note that
\[ \mathcal{P}_{[j,j+d]} = \{ \{ b_{j}, \ldots, b_{j+d} \}, v - \{ b_{j}, \ldots, b_{j+d} \} \} \]
is a standard linear partition such that the length of the principal interval is less than or equal to $C_{0}$, and the initial symbol is no more than $C_{1}$. Thus, $\mathcal{P}_{[j,j+d]} \in \mathcal{C}^{\ell}(v,C_{0},C_{1})$. We claim that, furthermore, 
$u_{0} \in lk_{\downarrow}^{\mathcal{P}_{[j,j+d]}}(v)$. Indeed, since the production of $b'_{j}$ is exactly the principal interval of $\mathcal{P}_{[j,j+d]}$, and the productions of distinct members of $u_{0}$ have disjoint support, it follows 
that the production of $b'_{i}$ is a subset of $v - \{ b_{j}, \ldots, b_{j+d} \}$, for each $i \neq j$. This proves the claim. Now since $u_{0} \in lk_{\downarrow}^{\mathcal{P}_{[j,j+d]}}(v)$, it follows that
$\sigma$ is a simplex in $lk_{\downarrow}^{\mathcal{P}_{[j,j+d]}}(v)$, by Remark \ref{remark:convvoc}.

If $j > C_{1}$, then we have
\[ b_{1} = b'_{1}, \ldots, b_{C_{1}} = b'_{C_{1}}. \]
Thus, each of $b'_{1}, \ldots, b'_{C_{1}}$ has trivial production, and this production is the entirety of the set
$\{ b_{1}, \ldots, b_{C_{1}} \} \subseteq v$. Every other production from a $b'_{i} \in u_{0}$ is a subset of
$\{ b_{C_{1}+1}, \ldots, b_{k} \}$. Note that 
\[ \mathcal{P}_{[1,C_{1}]} = \{ \{ b_{1}, \ldots, b_{C_{1}} \}, \{ b_{C_{1}+1}, \ldots, b_{k} \} \} \]
is a member of $\mathcal{C}^{\ell}(v,C_{0},C_{1})$ (since $C_{1} \leq C_{0}$ and the initial symbol $1$ is no more than $C_{1}$), and it follows from the above observations that
\[ u_{0} \in lk_{\downarrow}^{\mathcal{P}_{[1,C_{1}]}}(v), \]
so $\sigma$ is a simplex in $lk_{\downarrow}^{\mathcal{P}_{[1,C_{1}]}}(v)$, again by Remark \ref{remark:convvoc}.
Since $\sigma$ was arbitrary, we have now shown that 
\[ lk_{\downarrow}(v) \subseteq \bigcup_{\mathcal{P} \in \mathcal{C}^{\ell}(v,C_{0},C_{1})} lk_{\downarrow}^{\mathcal{P}}(v), \]
which proves the proposition in the linear case.

We next consider the cyclic case. By Remark \ref{remark:convvoc} (and as in the linear case), it suffices to show that every vertex $u$ in the descending link $lk_{\downarrow}(v)$ lies in 
$lk_{\downarrow}^{\mathcal{P}}(v)$, for some $\mathcal{P} \in \mathcal{C}^{c}(v,C_{0},C_{1})$. First, assume that $b_{1} \in v$ is part of the non-trivial production of some member $b'$ of $u$. In this case, the production in question involves no more than $C_{0}$ members of $v$, by the definition of the constant $C_{0}$. The production of all other members of $u$ is disjoint from that of $b'$.   It follows that there is some $\mathcal{P} \in \mathcal{C}^{c}(v,C_{0},C_{1})$ whose principal interval coincides with the production of $b'$ in $v$, and thus, $u \in lk_{\downarrow}^{\mathcal{P}}(v)$ (since the production of each $b'' \in u$ ($b'' \neq b'$) is disjoint from the principal interval of $\mathcal{P}$, by the above reasoning).
In the remaining case (in which $b_{1}$ is not part of a non-trivial production), one proceeds exactly as in the linear case. The cyclic case of the proposition follows.

In the remaining case, we assume that $\mathcal{B}$ satisfies the rich-in-contractions property in the permutational sense.
It again suffices to show that every vertex $u$ in $lk_{\downarrow}(v)$ lies in 
$lk_{\downarrow}^{\mathcal{P}}(v)$, for some $\mathcal{P} \in \mathcal{C}(v,C_{0})$. This follows from the fact that there is some $b' \in u$ with non-trivial production: taking a partition $\mathcal{P}$ with principal set equal to the production of $b'$, we easily conclude that $u \in lk_{\downarrow}^{\mathcal{P}}(v)$ (since all other productions are necessarily contained in the complement of the principal set). 
\end{proof}

\begin{theorem} \label{theorem:connectivityeasylinear}
(Connectivity of descending links in $\Delta_{\mathcal{B}}$: an easy sufficient condition for the linear case) Let $\mathcal{B}$ be a linear expansion set. Assume that $\mathcal{B}$ has the bounded contractions property and the rich-in-contractions property with constants $C_{0}$ and $C_{1}$, respectively, where ``rich-in-contractions" is in the  linear sense. We assume $C_{0} \geq C_{1}$. Assume that $v$ is consecutively ordered.

If \[ |v| \geq (n+2)C_{1} + (n+1)C_{0} - (n+1), \] for some $n \geq -1$, then 
$lk_{\downarrow}(v)$ is $n$-connected.
\end{theorem}

\begin{proof}
The proof uses the Nerve Theorem \ref{theorem:Nerve} and induction on $n$. We first note that $lk_{\downarrow}(v)$ is non-empty (i.e., $(-1)$-connected) when $|v| \geq C_{1}$, by the definition of the rich-in-contractions property. Thus, the theorem is true if $n=-1$.

Let $v = \{ b_{1}, \ldots, b_{k} \}$ be consecutively and linearly ordered.
We let $\mathcal{U}$ denote the cover of $lk_{\downarrow}(v)$ from Proposition \ref{proposition:naturalcovers}. We denote the members of $\mathcal{U}$, which are partitioned downward links of $v$, by $K_{i}$ (for various $i$), to simplify notation. For any finite subcollection 
\[ \mathcal{V} = \{ K_{1}, K_{2}, \ldots, K_{t} \}, \]
we define the \emph{complementary interval of $\mathcal{V}$} to be the subset of $v$ consisting of all $b_{j}$ such that 
$b_{j}$ is not in any of the principal intervals associated to the various $K_{i}$, nor is any $b_{j'}$ when $j' > j$. (Here each $K_{s}$ is
$lk_{\downarrow}^{\mathcal{P}_{s}}(v)$, for some standard linear partition $\mathcal{P}_{s}$. The complementary interval of $\mathcal{V}$ consists of all $b_{j}$ that lie ``to the right" of all of the principal intervals of the various $\mathcal{P}_{s}$.) We note that the initial symbol of each of the principal intervals in question is one of the numbers $1, \ldots, C_{1}$, and the length of each principal interval is at most $C_{0}$. It follows from this description that each complementary interval takes the form $\{ b_{k'}, b_{k'+1}, \ldots, b_{k} \}$, for
some $k'$ less than or equal to $C_{1} + C_{0}$. Thus, the complementary interval of $\mathcal{V}$ is consecutively ordered, and has at least $k - (C_{1} + C_{0} - 1)$ elements. 

Next, we let $\mathcal{P}$ denote the meet of the partitions $\mathcal{P}_{s}$, $s = 1, \ldots, t$. We have
\begin{align*}
|\cap_{s=1}^{t} K_{s}| &= |lk_{\downarrow}^{\mathcal{P}}(v)| \\
&\cong \bigast_{\beta} |lk_{\downarrow}(v_{\beta})|
\end{align*}
where the vertices $v_{\beta}$ are the members of the partition $\mathcal{P}$. The descending link of the complementary interval of
$\mathcal{V}$ is a factor in the latter join. (We note here that the complementary interval may not actually be a member of $\mathcal{P}$, but there will necessarily be a $v_{\hat{\beta}}$ of the form $v' \cup v''$, where $v''$ is the complementary interval, $v'$ is a subset of $v$, and 
$supp(v' \cup v'')$ is a disjoint union of $supp(v')$ and $supp(v'')$. In this case, we have
\[ |lk_{\downarrow}(v_{\hat{\beta}})| \cong |lk_{\downarrow}(v')| \ast |lk_{\downarrow}(v'')|, \]
so the desired conclusion stands.) It follows from all of this that \emph{$|\cap_{s=1}^{t} K_{s}|$ is at least as highly connected as $lk_{\downarrow}(v'')$, where $v''$ is the complementary interval of $\mathcal{V}$, and where $v''$ is consecutively ordered and has at least $k-(C_{0}+ C_{1}-1)$ members}.

Now we check the hypotheses of the Nerve Theorem (Theorem \ref{theorem:Nerve}). First, we consider the nerve $\mathcal{N}(\mathcal{U})$. 
Note that $\mathcal{U}$ is necessarily finite. It follows from the previous reasoning that an arbitrary intersection $|\cap K_{s}|$ of members of $\mathcal{U}$ is non-empty whenever $k - (C_{0} + C_{1} -1)$ is at least $C_{1}$. Thus, $\mathcal{N}(\mathcal{U})$ is isomorphic to a simplex when
$|v| = k \geq C_{0} + 2C_{1} -1$. In particular, $\mathcal{N}(\mathcal{U})$ is $n$-connected for arbitrary $n$ when $|v| \geq C_{0} + 2C_{1}-1$. 

Suppose that the theorem is true for some $n$. Thus, $lk_{\downarrow}(v)$ is $n$-connected when 
\[ |v| \geq (n+1)C_{0} + (n+2)C_{1} - (n+1), \]
and the specified hypotheses are satisfied. Assume that 
\[ |v| \geq (n+2)C_{0} + (n+3)C_{1} - (n+2). \]
It is sufficient, by Theorem \ref{theorem:Nerve}, to show that $t$-fold intersections of members of $\mathcal{U}$ are $(n-t+2)$-connected. However, 
such an intersection is at least as highly connected as the descending link on
the associated complementary interval, and the complementary interval in question has at least $(n+1)C_{0} + (n+2)C_{1} - (n+1)$ elements. Thus, the $t$-fold intersection is at least $n$ connected. This handles all cases except the case $t=1$ (i.e., the connectivity of an individual $K \in \mathcal{U}$). If $t=1$, then $K$ factors as a join in which one of the factors is the descending link on the principal interval, while another factor is the descending link on the complementary interval. Since the former is non-empty and the latter is at least
$n$-connected, $K$ is at least $(n+1)$-connected, as required. This completes the induction.
\end{proof}

\begin{theorem} \label{theorem:connectivityeasycyclic}
(Connectivity of descending links in $\Delta_{\mathcal{B}}$: an easy sufficient condition for the cyclic case) Let $\mathcal{B}$ be a cyclic expansion set. Assume that $\mathcal{B}$ has the bounded contractions property and the rich-in-contractions property with constants $C_{0}$ and $C_{1}$, respectively, where ``rich-in-contractions" is in the  cyclic sense. We assume $C_{0} \geq C_{1}$. Assume that $v$ is consecutively ordered.

The descending link $lk_{\downarrow}(v)$ is non-empty if $|v| \geq C_{1}$.
If \[ |v| \geq (n+2)C_{1} + (n+2)C_{0} - (n+2), \] for some $n \geq 0$, then 
$lk_{\downarrow}(v)$ is $n$-connected. 
\end{theorem}

\begin{proof}
It is clear, by the definition of $C_{1}$, that $lk_{\downarrow}(v)$ is non-empty if $|v| \geq C_{1}$.

Let $v = \{ b_{1}, \ldots, b_{k} \}$ be cyclically and consecutively ordered. The proof largely follows that of Theorem \ref{theorem:connectivityeasylinear}. We let $\mathcal{U}$ denote the cover from Proposition \ref{proposition:naturalcovers}, and we let $\mathcal{V}$ denote a particular subcollection
\[ \mathcal{V} = \{ K_{1}, \ldots, K_{t} \}, \]
where the $K_{s}$ are partitioned downward links from the cover $\mathcal{U}$. We define the \emph{complementary interval of $\mathcal{V}$}, denoted $v''$, to be the largest consecutively ordered subset
of
\[ v - (I_{1} \cup \ldots \cup I_{t}), \]
where $I_{s}$ denotes the principal interval of $\mathcal{P}_{s}$, and $K_{s} = lk_{\downarrow}^{\mathcal{P}_{s}}(v)$. By the definition of
$\mathcal{C}^{c}(v,C_{0},C_{1})$, the complementary interval of $\mathcal{V}$ has at least $k - (2C_{0} + C_{1} - 2)$ elements (since the only $b_{j}$ that lie in any $\mathcal{P}' \in \mathcal{C}^{c}(v,C_{0},C_{1})$ are
\[ b_{k-C_{0}+2}, \ldots, b_{k}, b_{1}, \ldots, b_{C_{0}+C_{1}-1}.) \]
Let $\mathcal{P}$ denote the meet of the partitions $\mathcal{P}_{1}, \ldots, \mathcal{P}_{k}$.
One has the same formula for the intersection of the $K_{s}$ as before:
\begin{align*}
| \cap_{s=1}^{t} K_{s}| &= |lk_{\downarrow}^{\mathcal{P}}(v)| \\
&\cong \bigast_{\beta} |lk_{\downarrow}(v_{\beta})|, 
\end{align*}
where the vertices $v_{\beta}$ are the members of the partition of $\mathcal{P}$. The descending link of the complementary interval of 
$\mathcal{V}$ is once again a factor of the above join (as in the proof of Theorem \ref{theorem:connectivityeasylinear}). Note also that the support of 
$v''$ is a linear interval, since $v''$ is consecutively ordered and $supp(v'')$ is a proper subset of the (cyclically ordered) $X$. 
It follows from this that \emph{$|\cap_{s=1}^{t} K_{s}|$ is at least as highly connected as $lk_{\downarrow}(v'')$, and the connectivity 
of $lk_{\downarrow}(v'')$ may be estimated from below using Theorem \ref{theorem:connectivityeasylinear}}.
 
Now we check the hypotheses of Theorem \ref{theorem:Nerve}. First, we consider $\mathcal{N}(\mathcal{U})$. If $|v| = k \geq 
2C_{1} + 2C_{0} - 2$, then $|v''| \geq C_{1}$, where $v''$ is the complementary interval for an arbitrary collection 
$\mathcal{V} \subseteq \mathcal{U}$. Thus, $\mathcal{N}(\mathcal{U})$ is a simplex when $|v| \geq 2C_{1} + 2C_{0} -2$, and thus
$n$-connected for all $n$.

We now prove the remainder of the theorem by induction on $n$. Consider first the base case, $n=0$. We assume that $|v| \geq 2C_{1} + 2C_{0} - 2$. By the previous paragraph, $\mathcal{N}(\mathcal{U})$ is $0$-connected, and so it suffices, by Theorem \ref{theorem:Nerve}, to prove that $t$-fold intersections are $(1-t)$-connected. This is a vacuous condition if $t \geq 3$. Suppose $t=2$. We must show that 
any intersection $K_{1} \cap K_{2}$ is non-empty ($K_{1}, K_{2} \in \mathcal{U}$). This is true since the intersection is at least as connected as $lk_{\downarrow}(v'')$ and 
\[ |v''| \geq (2C_{0} + 2C_{1} - 2) - (2C_{0} + C_{1} - 2) = C_{1}, \]
and thus $lk_{\downarrow}(v'')$ is non-empty. This settles the case $t=2$. If $t=1$, then $K_{1}$ is the join 
$|lk_{\downarrow}(I_{1}) \ast lk_{\downarrow}(v'')|$, where both factors are non-empty. Thus, 
$1$-fold intersections are $0$-connected, as desired. This completes the base case.

Now suppose that $lk_{\downarrow}(v)$ is $n$-connected if $|v| \geq (n+2)C_{0} + (n+2)C_{1} - (n+2)$. We consider 
$v$ such that $|v| \geq (n+3)(C_{0} + C_{1} - 1)$; we must show that $lk_{\downarrow}(v)$ is $(n+1)$-connected; it suffices, by Theorem \ref{theorem:Nerve}, to show that $t$-fold intersections of members from $\mathcal{U}$ are $(n-t+2)$-connected. 
Suppose first that $t \geq 2$. In this case, we know that $\cap_{s=1}^{t} K_{s}$ is at least as highly connected
as $lk_{\downarrow}(v'')$, where 
\[ |v''| = |v| - (2C_{0} + C_{1} - 2) \geq (n+1)C_{0} + (n+2)C_{1} - (n+1). \]
It follows from Theorem \ref{theorem:connectivityeasylinear} that $lk_{\downarrow}(v'')$ is $n$-connected. This proves that $t$-fold are $n$-connected when $t \geq 2$. We now must consider $1$-fold intersections. Thus, let $K \in \mathcal{U}$. The subcomplex $K$ factors as a join
\[ |K| = |lk_{\downarrow}(I)| \ast |lk_{\downarrow}(v'')|, \]
where $I$ is the principal interval of the partition $\mathcal{P}$ associated to $K$, $|lk_{\downarrow}(I)|$ is thus non-empty and $|lk_{\downarrow}(v'')|$ is $n$-connected. It follows that $K$ is $(n+1)$-connected, as required. This completes the induction.
\end{proof}

\begin{theorem} \label{theorem:connectivityeasyperm}
(Connectivity of descending links in $\Delta_{\mathcal{B}}$: an easy sufficient condition for the ``permutational" case) Let $\mathcal{B}$ be an expansion set. Assume that $\mathcal{B}$ has the bounded contractions property and the rich-in-contractions property with constants $C_{0}$ and $C_{1}$, respectively, where ``rich-in-contractions" is in the permutational sense. We assume $C_{0} \geq C_{1}$. 

The descending link $lk_{\downarrow}(v)$ is non-empty if $|v| \geq C_{1}$.
If \[ |v| \geq (2n+2)C_{0} + C_{1}, \] for some $n \geq 0$, then 
$lk_{\downarrow}(v)$ is $n$-connected. 
\end{theorem}

\begin{proof}
The proof is like those of Theorems \ref{theorem:connectivityeasylinear}
and \ref{theorem:connectivityeasycyclic}, but easier in most respects.
Let $v = \{b_{1}, \ldots, b_{k} \}$. We first note that $lk_{\downarrow}(v)$ is $(-1)$-connected if $|v| \geq C_{1}$.

We first consider an arbitrary subcollection $\mathcal{V}$ of $\mathcal{U}$, where 
$\mathcal{U}$ is the cover from Proposition \ref{proposition:naturalcovers}. Thus, 
\[ \mathcal{V} = \{ K_{1}, \ldots, K_{t} \}, \]
for some $t$, where each $K_{s}$ is a partitioned downward link $lk_{\downarrow}^{\mathcal{P}_{s}}(v)$, for a partition $\mathcal{P}_{s} \in \mathcal{C}(v,C_{0})$. We let $I_{1}, \ldots, I_{t}$ be the principal sets associated to the partitions $\mathcal{P}_{1}, \ldots, \mathcal{P}_{t}$. We define $v''$, the \emph{complementary set of $\mathcal{V}$}, to be the complement of $(I_{1} \cup \ldots \cup I_{t})$ in $v$. We note that
\[ |v''| \geq (|v| - tC_{0}). \]
The usual argument shows that $|\cap_{s=1}^{t} K_{s}|$ is at least as highly  connected as $lk_{\downarrow}(v'')$.

We now try to apply the Nerve Theorem, first considering the nerve $\mathcal{N}(\mathcal{U})$. Consider any $\mathcal{V} \subseteq \mathcal{U}$ with $t$ elements
$K_{1}, \ldots, K_{t}$ (as above). The complementary set $v''$ has at least $|v| - tC_{0}$ elements. Thus, if $|v| \geq tC_{0} + C_{1}$, we find that 
\[ |v''| \geq (tC_{0} + C_{1}) - tC_{0} = C_{1}, \]
So that $\{ K_{1}, \ldots, K_{t} \}$ is a simplex in $\mathcal{N}(\mathcal{U})$ when
$|v| \geq tC_{0} + C_{1}$. Thus, every $t$-element subset of $\mathcal{U}$ is a simplex of 
$\mathcal{N}(\mathcal{U})$ when $|v| \geq tC_{0} + C_{1}$, which makes $\mathcal{N}(\mathcal{U})$ the $(t-1)$-skeleton of a high-dimensional simplex under the latter hypothesis. Thus, $\mathcal{N}(\mathcal{U})$ is $(t-2)$-connected when $|v| \geq tC_{0} + C_{1}$, or, replacing $t$ by $t+2$, we find that $\mathcal{N}(\mathcal{U})$ is $t$-connected when $|v| \geq (t+2)C_{0} + C_{1}$.  

We now prove the theorem by induction, beginning with the base case $n=0$. Assume that $|v| \geq 2C_{0} + C_{1}$. It follows, by the previous reasoning, that $\mathcal{N}(\mathcal{U})$ is $0$-connected, so it suffices to show that $t$-fold intersections of members of $\mathcal{U}$ are $(1-t)$-connected. If $t=2$, we find that the intersection
$|K_{1} \cap K_{2}|$ is at least as connected as $|lk_{\downarrow}(v'')|$, where $|v''| \geq C_{1}$. Thus, $|K_{1} \cap K_{2}|$ is non-empty, as required. If $t=1$, then $K_{1}$ is 
a join
\[ lk_{\downarrow}(I) \ast lk_{\downarrow}(v''), \]
where both factors are non-empty, so $|K_{1}|$ is $0$-connected, as required. (Here $I$ is the principal interval associated to $K_{1}$.)

We proceed to the inductive step. Assume that, for $0 \leq m \leq n$, whenever $|v| \geq (2m+2)C_{0} + C_{1}$,
$lk_{\downarrow}(v)$ is $m$-connected. We assume that $|v| \geq (2n+4)C_{0}  + C_{1}$.
This easily implies that $\mathcal{N}(\mathcal{U})$ is $(n+1)$-connected; it therefore suffices to show that $t$-fold intersections of members of $\mathcal{U}$ are $(n-t+2)$-connected. Letting $2 \leq t \leq n+3$, we find that 
\[ |v''| \geq (2n -t +4)C_{0} + C_{1} \geq (2(n-t+2)+2)C_{0} + C_{1}. \]
It follows by induction that such intersections $(n-t+2)$-connected, as desired. If $t=1$, then $K$ factors as a join 
\[ lk_{\downarrow}(I) \ast lk_{\downarrow}(v''), \]
where $lk_{\downarrow}(I)$ is non-empty and $lk_{\downarrow}(v'')$ is $n$-connected. It follows that $K$ is $(n+1)$-connected, as desired.
\end{proof}

\begin{corollary} \label{corollary:connectivityeasy}
(Connectivity of the filtration $\{ \Delta_{\mathcal{B},m}^{f} \}$: an easy case) Let $n \geq 0$.
Let $\mathcal{B}$ be an expansion set such that $\Delta^{f}_{\mathcal{B}}$ is $n$-connected. Assume that $\mathcal{B}$ has the bounded contractions property and the rich-in-contractions property with constants $C_{0}$ and $C_{1}$, respectively, where ``rich-in-contractions" is interpreted in the appropriate sense. We assume that $C_{0} \geq C_{1}$. 
\begin{enumerate}
\item If $\mathcal{B}$ is linear, and 
\[ m \geq (n+2)C_{1} + (n+1)C_{0} - (n+2), \]
then $\Delta_{\mathcal{B},m}^{f}$ is $n$-connected.
\item If $\mathcal{B}$ is cyclic, and 
\[ m \geq (n+2)C_{1} + (n+2)C_{0} - (n+3), \]
then $\Delta_{\mathcal{B},m}^{f}$ is $n$-connected.
\item If $\mathcal{B}$ is permutational and 
\[ m \geq (2n+2)C_{0} + C_{1} -1, \]
then $\Delta_{\mathcal{B},m}^{f}$ is $n$-connected.
\end{enumerate}
\end{corollary}

\begin{proof}
This follows by combining the result of Proposition \ref{proposition:connectfilt} with the results of Theorems 
\ref{theorem:connectivityeasylinear},\ref{theorem:connectivityeasycyclic}, and 
\ref{theorem:connectivityeasyperm}, respectively.
\end{proof}

\section{Group actions on expansion sets}
\label{section:groupactions}

Each generalized Thompson group $\Gamma$ is piecewise determined by an inverse semigroup
$S$ acting on a set $X$. We describe how to create a larger inverse semigroup $\widehat{S}$, which contains $\Gamma$ as a subgroup, and also describe actions of $\widehat{S}$ on $\mathcal{B}$. The result is a ``piecewise" treatment of the actions of $\Gamma$ on $\Delta_{\mathcal{B}}$ and $\Delta^{f}_{\mathcal{B}}$. We consider sufficient conditions for the action of $\Gamma$ on $\Delta^{f}_{\mathcal{B}}$ to have cell stabilizers of type $F_{n}$ (Subsection \ref{subsection:stabilizers}) and to be cocompact (Subsection \ref{subsection:cocompact}). Subsection \ref{subsection:Brown} compares Brown's Finiteness Criterion (Theorem \ref{theorem:Brown}) with the results that have been proved so far; the result is a template for proving the $F_{\infty}$ property in a wide number of cases (Theorem \ref{theorem:template}). 

\subsection{Actions of inverse semigroups on expansion sets}
\label{subsection:actions}

\begin{definition} \label{definition:inversesemigroup} 
(Inverse semigroups of partial bijections; actions of inverse semigroups) 
Let $X$ be a set. A \emph{partial bijection} of $X$ is a bijection $h: A \rightarrow B$ between subsets of $X$. A collection $S$ of partial bijections of $X$ is called an \emph{inverse semigroup acting on $X$} if it is closed under inverses and compositions; i.e.,
\begin{enumerate}
\item if $h:A \rightarrow B$ is in $S$, then $h^{-1}:B \rightarrow A$ is also in $S$;
\item if $h: A \rightarrow B$ and $g:A' \rightarrow B'$ are members of $S$, then 
the composition $g \circ h: h^{-1}(A') \rightarrow g(A' \cap B)$ is also in $S$. 
\end{enumerate}
We note that the empty function is a zero element.
\end{definition}

\begin{lemma} \label{lemma:restrictions} (Closure under restriction to subdomains)
Let $S$ be an inverse semigroup acting on $X$. Let $s_{1}, s_{2} \in S$.
The restriction of $s_{1}$ to the domain of $s_{2}$ is also a member of $S$.
\end{lemma}

\begin{proof}
The restriction in question is $s_{1} s_{2}^{-1} s_{2}$, which is in $S$, since $S$ is closed under compositions and inverses.
\end{proof}

\begin{definition} \label{definition:union} 
(Disjoint union of partial bijections)
Let $s_{1}: A \rightarrow B$ and $s_{2}: C \rightarrow D$ be bijections, where $A$, $B$, $C$, and $D$ are subsets of a common set $X$.
Assume that $A \cap C = \emptyset$ and $B \cap D = \emptyset$. We define 
$s_{1} \amalg s_{2}: A \cup C \rightarrow B \cup D$ by the rules:
\[ (s_{1} \amalg s_{2})_{\mid A} = s_{1}; \quad (s_{1} \amalg s_{2})_{\mid C} = s_{2}. \]
The function $s_{1} \amalg s_{2}$ is the \emph{disjoint union} of $s_{1}$ and $s_{2}$. It is a bijection from 
$A \cup C$ to $B \cup D$.
\end{definition}

\begin{convention} \label{convention:SShatGamma} 
(The inverse semigroups $S$ and $\widehat{S}$, and the group $\Gamma$) 
We assume throughout this section that an expansion set $\mathcal{B}$ over $X$ and an inverse semigroup $S$ 
acting on $X$ are given. We will define a larger inverse semigroup $\widehat{S}$ and a group $\Gamma$ from this starting information.

We let $\widehat{S}$ be an inverse semigroup with the following properties:
\begin{enumerate}
\item $S \subseteq \widehat{S}$;
\item each $\hat{s}$ is a finite disjoint union of members of $S$;
\item $\widehat{S}$ is closed under disjoint unions;
\item $\widehat{S}$ has a non-empty subset of full-support elements; i.e., bijections from $X$ to $X$. Such elements obviously form a group, which we denote $\Gamma$, or $\Gamma_{\widehat{S}}$, if we wish to emphasize the role of $\widehat{S}$;
\item for each $b \in \mathcal{B}$, $supp(b)$ is the domain of some $\hat{s} \in \widehat{S}$. (We do not need to assume the converse; i.e., that every domain of an $\hat{s} \in \widehat{S}$ is the support of some $b$.)  
\end{enumerate}

If $\mathcal{B}$ is linear or cyclic, then $X$ has a linear or cyclic order topology. We will assume that each member of $\widehat{S}$ is
continuous and order-preserving (in the relevant sense). By (1), this forces each member of $S$ to be continuous and order-preserving.
Property (3) also has to be qualified slightly: we assume that, if $\hat{s}_{1}, \hat{s}_{2} \in \widehat{S}$, then $\hat{s}_{1} \amalg \hat{s}_{2} \in \widehat{S}$ if the latter is continuous and order-preserving, with continuous and order-preserving inverse.

We note that the group $\Gamma$ is the ultimate object of study in this paper.
\end{convention}

\begin{definition} \label{definition:actionsonexpansionsets}
(Actions of inverse semigroups on expansion sets)
Let $\widehat{S}$ be an inverse semigroup acting on $X$. Let $\mathcal{B} = (\mathcal{B}, X, supp,\mathcal{E})$ be an expansion set. An \emph{action $\cdot$ of $\widehat{S}$ on $\mathcal{B}$} satisfies the following:
\begin{enumerate}
\item for each $\hat{s} \in \widehat{S}$ and $b \in \mathcal{B}$, $\hat{s} \cdot b \in \mathcal{B}$ is defined if and only if $supp(b)$ is contained in  $dom(\hat{s})$;
 
\item for each $(\hat{s},b) \in \widehat{S} \times \mathcal{B}$ for which $\hat{s} \cdot b$ is defined, $supp(\hat{s} \cdot b) = \hat{s}(supp(b))$;
\item for $\hat{s}_{1}$, $\hat{s}_{2} \in \widehat{S}$ and $b \in \mathcal{B}$, 
\[ \hat{s}_{1} \cdot ( \hat{s}_{2} \cdot b ) = (\hat{s}_{1} \hat{s}_{2}) \cdot b ,\]
if both sides of the equation are defined. If $id_{A}$ is the identity function on some $A \subseteq X$, then
 \[ id_{A} \cdot b = b \]
if $supp(b) \subseteq A$;
\item for each pair $(\hat{s},b)  \in \widehat{S} \times \mathcal{B}$ for which $\hat{s} \cdot b$ is defined, the natural function 
$\hat{s}: \mathcal{E}(b) \rightarrow \mathcal{V}_{\mathcal{B}}$ determines an order preserving bijection $\hat{s}: \mathcal{E}(b) \rightarrow \mathcal{E}(\hat{s} \cdot b)$.
\end{enumerate}
\end{definition}

\begin{definition} \label{definition:subcomplexes}
(The subcomplex with support in $A \subseteq X$)
Let $\mathcal{B}$ be an expansion set over $X$, and let $A \subseteq X$. We
let $\Delta^{A}_{\mathcal{B}}$ be the subcomplex of $\Delta_{\mathcal{B}}$ spanned by vertices $v \in \mathcal{V}_{\mathcal{B}}$ such that $supp(v) = A$.
\end{definition}

\begin{remark}
The complex $\Delta^{A}_{\mathcal{B}}$ may be empty, since not every subset of $X$ is equal to the support of a vertex.
\end{remark}

\begin{theorem} \label{theorem:partialaction}
(The action of $\Gamma$ on $\Delta_{\mathcal{B}}$)
Let $\mathcal{B}$ be an expansion set. Let $\widehat{S}$ act on $\mathcal{B}$. Let $\Gamma$ denote the full-support subgroup of $\widehat{S}$. 

Each $\hat{s} \in \widehat{S}$ determines a simplicial isomorphism from
$\Delta^{A}_{\mathcal{B}}$ to $\Delta^{\hat{s}(A)}_{\mathcal{B}}$, provided that
$A \subseteq dom(\hat{s})$. 

In particular, the group $\Gamma$ acts on
$\Delta^{f}_{\mathcal{B}}$ and $\Delta_{\mathcal{B}}$ by simplicial automorphisms.
\end{theorem}

\begin{proof}
We extend the action of $\widehat{S}$  on $\mathcal{B}$ in the natural way to $2^{\mathcal{B}}$ (the power set of $\mathcal{B}$, of which $\mathcal{V}_{\mathcal{B}}$ is a subset) and to $2^{2^{\mathcal{B}}}$ (which contains $\mathcal{S}_{\mathcal{B}}$ as a subset).

We assume that $A \subseteq dom(\hat{s}) \subseteq X$. 
Let 
$v = \{ b_{1}, \ldots, b_{k} \}$ be a vertex in $\Delta^{A}_{\mathcal{B}}$. By definition,
$\hat{s} \cdot v = \{ \hat{s} \cdot b_{1}, \ldots, \hat{s} \cdot b_{k} \}$. Since $\hat{s}$ is injective and the supports $supp(b_{i})$ ($i=1, \ldots, k$) are pairwise disjoint, 
the sets $\hat{s}(supp(b_{i}))$ ($i=1, \ldots, k$) are also pairwise disjoint. Thus, 
by Definition \ref{definition:actionsonexpansionsets}(2), the sets 
$supp( \hat{s} \cdot b_{i})$ ($i=1, \ldots, k$) are pairwise disjoint, 
so $\hat{s} \cdot v$ is a vertex. Definition \ref{definition:actionsonexpansionsets}(2) also implies that 
\[ supp(\hat{s} \cdot v) = \hat{s}(supp(v)) = \hat{s}(A),\] so
$\hat{s} \cdot v \in \Delta^{\hat{s}(A)}_{\mathcal{B}}$. It follows that $\hat{s}$ sends vertices of $\Delta^{A}_{\mathcal{B}}$ to vertices of $\Delta^{\hat{s}(A)}_{\mathcal{B}}$.

Now let  $\sigma = \{ u_{0}, \ldots, u_{n} \}$ be an $n$-simplex in $\Delta^{A}_{\mathcal{B}}$, where 
$u_{0} = \{ b_{1}, \ldots, b_{k} \}$. 
We assume that $\sigma$ is in canonical order, so that, for each $b_{i} \in u_{0}$,
\[ \{ b_{i} \} \leq r_{b_{i}}(u_{1}) \leq \ldots \leq r_{b_{i}}(u_{n}) \]
is a chain in $\mathcal{E}(b_{i})$. We must show that $\hat{s} \cdot \sigma = \{ \hat{s} \cdot u_{0}, \ldots, \hat{s} \cdot u_{n} \}$ is a simplex. For this, it suffices to show that
\[ \{ \hat{s} \cdot b_{i} \} \leq r_{\hat{s} \cdot b_{i}}(\hat{s} \cdot u_{1}) \leq \ldots 
\leq r_{\hat{s} \cdot b_{i}}(\hat{s} \cdot u_{n}) \]
is a chain in $\mathcal{E}(\hat{s} \cdot b_{i})$, for each $i$. This follows from Definition \ref{definition:actionsonexpansionsets}(4) and from the fact that
$\hat{s} \cdot r_{b_{i}}(u_{j}) = r_{\hat{s} \cdot b_{i}}(\hat{s} \cdot u_{j})$, for $j =1, \ldots, n$.

Thus, $\hat{s}$ determines a simplicial map from $\Delta^{A}_{\mathcal{B}}$ to 
$\Delta^{\hat{s}(A)}_{\mathcal{B}}$. The fact that this map is an isomorphism is a consequence of the fact that $\hat{s}^{-1}$ determines a simplicial map from
$\Delta^{\hat{s}(A)}_{\mathcal{B}}$ to $\Delta^{A}_{\mathcal{B}}$, and the fact that these simplicial maps are two-sided inverses for each other (by Definition \ref{definition:actionsonexpansionsets}(3)). This proves the main statement.

The statements about $\Gamma$ follow immediately. 
\end{proof}

\subsection{Stabilizers of simplices}
\label{subsection:stabilizers}

\begin{definition} \label{definition:stabilizer}
(Stabilizers) 
Let $\cdot$ be an inverse semigroup action of $\widehat{S}$ on 
$\mathcal{B}$.
Let $b \in \mathcal{B}$. The \emph{stabilizer} of $b$, denoted $\widehat{S}_{b}$, is defined as follows:
\[ \widehat{S}_{b} = \{ \hat{s} \in \widehat{S} \mid \hat{s} \cdot b = b \text{ and } dom(\hat{s}) = supp(b) \}. \]
\end{definition}

\begin{lemma} \label{lemma:group}
(The stabilizer of each $b \in \mathcal{B}$ is a group)
Assume that $\widehat{S}$ satisfies Convention \ref{convention:SShatGamma}.
Let $\cdot$ be an inverse semigroup action of $\widehat{S}$ on 
$\mathcal{B}$. For each $b \in \mathcal{B}$, $\widehat{S}_{b}$ is a group.
\end{lemma}

\begin{proof}
Let $b \in \mathcal{B}$. By Convention \ref{convention:SShatGamma}(5), there is some $\hat{s} \in \widehat{S}$ such that $dom(\hat{s}) = supp(b)$. The function $\hat{s}^{-1}\hat{s}$ is equal to $id_{supp(b)}$, and so 
$\hat{s}^{-1}\hat{s} \cdot b = b$ by Definition \ref{definition:actionsonexpansionsets}(3).
Thus, 
$\widehat{S}_{b}$ is non-empty. The rest of the proof is an entirely routine consequence of properties (1)-(3) from Definition \ref{definition:actionsonexpansionsets}.
\end{proof}

\begin{proposition} \label{proposition:finiteindex}
(Cell stabilizers in $\Delta^{f}_{\mathcal{B}}$)
Assume that $\widehat{S}$ satisfies Convention \ref{convention:SShatGamma}. Let $\Gamma$ be the full-support subgroup of $\widehat{S}$.

\begin{enumerate} 
\item If $v= \{ b_{1}, \ldots, b_{k} \}$ is a vertex of $\Delta^{f}_{\mathcal{B}}$, then the stabilizer 
$\Gamma_{v}$ has a finite index subgroup that is isomorphic to
\[ \prod_{i=1}^{k} \widehat{S}_{b_{i}}. \] 
\item Assume that, for each $b \in \mathcal{B}$, the action of $\widehat{S}_{b}$ on $\mathcal{E}(b)$ has finite orbits; i.e., if $v \in \mathcal{E}(b)$, then 
\[ \{ g \cdot v \mid g \in \widehat{S}_{b} \} \]
is finite. If $\sigma$ is a simplex of $\Delta^{f}_{\mathcal{B}}$, and $v$ is the vertex of $\sigma$ having the least height, then the stabilizer $\Gamma_{\sigma}$ has finite index in $\Gamma_{v}$. 
\end{enumerate}
\end{proposition}

\begin{proof}
We first prove (1). Let $v = \{ b_{1}, \ldots, b_{k} \}$ be a vertex of $\Delta^{f}_{\mathcal{B}}$, and let 
$\Gamma_{v}$ be the stabilizer of $v$ in $\Gamma$. Each $\gamma \in \Gamma_{v}$ permutes $v$, resulting in a homomorphism $\phi: \Gamma_{v} \rightarrow S_{k}$, where $S_{k}$ denotes the symmetric group on $k$ symbols. Since $|S_{k}| < \infty$, the kernel, $\Gamma'_{v}$, has finite index in $\Gamma_{v}$. For each $i = 1, \ldots, k$, we then have a homomorphism $\pi_{i}: \Gamma'_{v} \rightarrow \widehat{S}_{b_{i}}$, which sends $\gamma$ to the restriction
$\gamma_{\mid supp(b_{i})}$. Since $supp(v) = X$, the images of these homomorphisms entirely determine $\gamma$; i.e., the homomorphism $\psi: \Gamma'_{v} \rightarrow \prod_{i=1}^{k} \widehat{S}_{b_{i}}$ defined by
\[ \gamma \mapsto (\gamma_{\mid supp(b_{1})}, \ldots, \gamma_{\mid supp(b_{k})}) \]
is injective. To prove surjectivity, we appeal to the fact that $\widehat{S}$ is closed under unions: if $(\alpha_{1}, \ldots, \alpha_{k}) \in \prod_{i=1}^{k} \widehat{S}_{b_{i}}$, we let 
$\gamma = \alpha_{1} \cup \ldots \cup \alpha_{k}$. Clearly, $\psi(\gamma) =
(\alpha_{1}, \ldots, \alpha_{k})$, proving surjectivity. The latter proof (of surjectivity) works in the cases of linear and cyclic expansion sets with very minor changes (one needs to be careful about the ordering), and the proof of injectivity is the same in all cases. This proves (1).

Now we prove (2). Let $\sigma = \{ v_{1}, \ldots, v_{n} \}$ be a simplex in $\Delta^{f}_{\mathcal{B}}$, with the $v_{i}$ arranged in the canonical order. We assume that $v_{1} = \{ b_{1}, \ldots, b_{k} \}$, and consider the vertex stabilizer group $\Gamma_{v_{1}}$. By (1), $\Gamma_{v_{1}}$ has a finite-index subgroup $\Gamma'_{v_{1}}$, which can be identified with the product $\prod_{i=1}^{k} \widehat{S}_{b_{i}}$. The vertex $v_{2}$ is a union
\[ r_{b_{1}}(v_{2}) \cup r_{b_{2}}(v_{2}) \cup \ldots \cup r_{b_{k}}(v_{2}), \]
by Definition \ref{definition:deltaB}, where each $r_{b_{i}}(v_{2}) \in \mathcal{E}(b_{i})$. Each $\gamma \in \Gamma
_{v_{1}}$ sends $r_{b_{i}}(v_{2})$ to another member of $\mathcal{E}(b_{i})$. There is thus a natural action
\[ \cdot : \Gamma'_{v_{1}}  \times \prod_{i=1}^{k} \mathcal{O}(r_{b_{i}}(v_{2})) \rightarrow  \prod_{i=1}^{k} \mathcal{O}(r_{b_{i}}(v_{2})), \]
where $\mathcal{O}(r_{b_{i}}(v_{2}))$ is an orbit under the action of $\widehat{S}_{b_{i}}$ on $\mathcal{E}(b_{i})$. 
Each set $\mathcal{O}(r_{b_{i}}(v_{2}))$ is finite by hypothesis, so a finite-index subgroup $\widetilde{\Gamma}_{v_{1}}$ acts trivially, and thus fixes $v_{2}$. After again passing to a finite index subgroup, we can find $\Gamma''_{v_{1}} \leq \Gamma'_{v_{1}}$ such that $\Gamma''_{v_{1}}$ fixes each member of $v_{2}$. We repeat this procedure, finding a sequence
\[ \Gamma'_{v_{1}} \geq \Gamma''_{v_{1}} \geq \ldots \Gamma^{(n)}_{v_{1}}, \]
where each $\Gamma^{(j)}_{v_{1}}$ fixes $v_{j}$ and each member of $v_{j}$, and each $\Gamma^{(j+1)}_{v_{1}}$ has finite index in $\Gamma^{(j)}_{v_{1}}$ ($j=1, \ldots, n-1$). Thus, $\Gamma^{(n)}_{v_{1}}$ is a finite-index subgroup of $\Gamma_{v_{1}}$ and $\Gamma^{(n)}_{v_{1}} \leq \Gamma_{\sigma}$. It follows that $[\Gamma_{v_{1}} : \Gamma_{\sigma}] < \infty$. 
\end{proof}

\begin{corollary} \label{corollary:Fnstabilizers}
(Finiteness properties of stabilizers)
Assume that $\widehat{S}$ satisfies Convention \ref{convention:SShatGamma}. Assume that, for each $b \in \mathcal{B}$, the action of $\widehat{S}_{b}$ on $\mathcal{E}(b)$ has finite orbits.

If each stabilizer group $\widehat{S}_{b}$ ($b \in \mathcal{B}$) is of type $F_{n}$, then 
every vertex- and cell-stabilizer in $\Delta^{f}_{\mathcal{B}}$ also has type 
$F_{n}$. 
\end{corollary}

\begin{proof}
This follows from Proposition \ref{proposition:finiteindex}, and from the fact that the class of type $F_{n}$ groups is closed under finite products and under passage to finite-index subgroups.
\end{proof}

\subsection{Cocompactness of the action on $\Delta^{f}_{\mathcal{B}}$}
\label{subsection:cocompact}

\begin{proposition} \label{proposition:findiminv} 
Each complex $\Delta^{f}_{\mathcal{B},n}$ is 
$\Gamma$-invariant. 
\end{proposition}

\begin{proof}
This follows directly from the fact that the action of $\Gamma$ preserves the subcomplex $\Delta^{f}_{\mathcal{B}}$, and acts in a height-preserving manner.  
\end{proof}

\begin{proposition} \label{proposition:cocompact} (A sufficient condition for cocompactness) Assume that $\widehat{S}$ satisfies Convention \ref{convention:SShatGamma}. If the action of $\Gamma$ on $\mathcal{B}$ has finitely many orbits, and the action of each 
$\widehat{S}_{b}$ on $\mathcal{E}(b)$ is cocompact,
then the action
of $\Gamma$ on each subcomplex $\Delta^{f}_{\mathcal{B},n}$ is cocompact.
\end{proposition}

\begin{proof}
We fix a specific subcomplex $\Delta^{f}_{\mathcal{B},n}$. We will first show that there are only finitely many orbits of vertices in this complex, relative to the action of $\Gamma$. The specifics of the proof depend in very minor ways on whether the expansion set in question is linear, cyclic, or permutational. We will focus on the case of permutational expansion sets. (In the present context, ``permutational" means that the
inverse semigroup follows Convention \ref{convention:SShatGamma}, without the additional restrictions imposed on the linear and cyclic cases.)

Since the action of $\Gamma$ on $\mathcal{B}$ has finitely many orbits, we can find a finite set $\{ t_{1}, \ldots, t_{\ell} \} \subseteq \mathcal{B}$ of orbit representatives (i.e., a selection of exactly one $t_{i}$ from each orbit).  To each vertex $v \in \Delta^{f}_{\mathcal{B}}$, we associate a \emph{type vector}, which is an $\ell$-tuple
$(a_{1}, \ldots, a_{\ell})$ such that
\[ a_{i} = |\{ b \in v \mid b \text{ is in the same }\Gamma \text{-orbit as } t_{i} \}|. \]
Clearly, $a_{1} + \ldots + a_{\ell} = |v|$ and $a_{i} \geq 0$ for all $i$. Since $|v| \leq n$, it follows easily that only finitely many vectors can be the type vector of a vertex $v \in \Delta^{f}_{\mathcal{B},n}$. It therefore suffices to show that two vertices with the same type vector are actually in the same $\Gamma$-orbit. (In the linear case, we would consider the \emph{type sequence}, where two ordered vertices 
$v_{1} = \{ b_{1}, \ldots, b_{k} \}$ and $v_{2} = \{ b'_{1}, \ldots, b'_{k} \}$ have the same type sequence
if $b_{j}$ and $b'_{j}$ are in the same orbit under the action of $\Gamma$, for $j=1, \ldots, k$. One would then show that 
two vertices with the same type sequence are in the same orbit, and there are clearly only finitely many type sequences under our hypotheses. The cyclic case involves an analogous idea.)

Thus, let $v$ and $v'$ be vertices of $\Delta^{f}_{\mathcal{B},n}$ such that $v$ and $v'$ have the same type vector. It follows directly that $v$ and $v'$ have the same height. We can write 
$v = \{ b_{1}, \ldots, b_{k} \}$ and $v' = \{ b'_{1}, \ldots, b'_{k} \}$ where, for $i=1, \ldots, k$, $b_{i}$ and $b'_{i}$ are in the same $\Gamma$-orbit. We can build $\gamma \in \Gamma$ carrying $v$ to $v'$ as follows. For each $i=1, \ldots, k$, we are given $\gamma_{i} \in \Gamma$ 
such that $\gamma_{i} \cdot b_{i} = b'_{i}$. It follows that 
$\gamma_{i \mid supp(b_{i})} \cdot b'_{i}$. Note that $\gamma_{i \mid supp(b_{i})} \in \widehat{S}$ by Lemma \ref{lemma:restrictions}. Now, using Convention \ref{convention:SShatGamma}(3), we can take the union of the various $\gamma_{i \mid supp(b_{i})} \in \widehat{S}$ to create $\gamma \in \Gamma$ such that $\gamma \cdot v = v'$. Thus, there are only finitely orbits of vertices in $\Delta^{f}_{\mathcal{B},n}$.

It now suffices to show that there are only finitely many different orbits of simplices in $\Delta^{f}_{\mathcal{B},n}$. We suppose the contrary. Since $\Delta^{f}_{\mathcal{B},n}$ is finite dimensional, it follows that, for some $m$, there is an infinite collection $\mathcal{C}_{1}$ of $m$-dimensional simplices such that no two members of $\mathcal{C}_{1}$ are in the same $\Gamma$-orbit. Since there are only finitely many orbits of vertices in $\Delta^{f}_{\mathcal{B},n}$, we can find an infinite subcollection $\mathcal{C}_{2} \subseteq \mathcal{C}_{1}$ such that, whenever $\sigma_{1}, \sigma_{2} \in \mathcal{C}_{2}$, the lowest-height vertices in $\sigma_{1}$ and $\sigma_{2}$ lie in the same orbit. After replacing each such simplex with a suitable translate, we have a collection $\mathcal{C}$ of $m$-dimensional simplices such that:
\begin{enumerate}
\item each $\sigma \in \mathcal{C}$ has a certain fixed $v$ as its lowest-height vertex;
\item no two members of $\mathcal{C}$ are in the same $\Gamma$-orbit.
\end{enumerate}

We will now argue for a contradiction. It will be useful to develop a bit of notation for simplices in 
$\Delta^{f}_{\mathcal{B}}$. Let $\sigma = \{ w_{0}, \ldots, w_{m} \}$ be an $m$-simplex in
$\Delta^{f}_{\mathcal{B}}$ and let $w_{0} = \{ b_{1}, \ldots, b_{k} \}$. To the simplex $\sigma$, we associate an $(m+1) \times k$ matrix $M_{\sigma}$, such that the $ij$-entry (for $0 \leq i \leq m$ and $1 \leq j \leq k$) is $r_{b_{j}}(w_{i})$. The $i$th row of the matrix represents the vertex $v_{i}$, while the $j$th column represents expansions from $b_{j}$. (If we make the convention that the $0$th row is bottommost, and the $m$th is on the top, Figure \ref{matrix} illustrates the basic idea.) According to Definition \ref{definition:deltaB},
the matrix $M_{\sigma}$ has the following two properties:
\begin{enumerate}
\item no two rows of $M_{\sigma}$ are equal (since the vertices $v_{0}, \ldots, v_{m}$ are distinct);
\item the $j$th column, read from bottom to top (i.e., from row $0$ to row $m$), is a non-decreasing sequence (i.e., one possibly including repetitions) in $\mathcal{E}(b_{j})$, where the $0j$th entry is $\{ b_{j} \}$.
\end{enumerate}
Two such matrices represent the same simplex if and only if 
one can be obtained from the other by permuting the columns.

The matrices $M_{\sigma}$ for $\sigma \in \mathcal{C}$ all have the same $0$th (i.e., bottom) row, possibly after permutation of the columns (since all  simplices $\sigma \in \mathcal{C}$ have the same least-height vertex
$w_{0}$). We will assume, for the sake of simplicity, that the columns are ordered in such a way that the $0$th rows are all identical. If we delete repetitions from the $j$th column of $M_{\sigma}$, we obtain a simplex in $\mathcal{E}(b_{j})$, which we will call the \emph{simplex underlying the $j$th column} of $M_{\sigma}$. 

Now, still assuming that $\mathcal{C}$ is as above, we are forced to conclude, by the assumption that $\widehat{S}_{b}$ always acts cocompactly on $\mathcal{E}(b)$, that there is an infinite subcollection $\mathcal{C}'$ such that, if $\sigma_{1}, \sigma_{2} \in \mathcal{C}'$, then the simplex underlying the $j$th column of $M_{\sigma_{1}}$ is in the same $\widehat{S}_{b_{j}}$-orbit as the simplex underlying the $j$th column of $M_{\sigma_{2}}$, for $j=1, \ldots, k$. Since $\mathcal{C}' \subseteq \mathcal{C}$, no two members of $\mathcal{C}'$ lie in the same $\Gamma$-orbit. After replacing the family $\mathcal{C}'$ with suitable translates under the action of
$\Gamma'_{v_{0}}$, we arrive at an infinite family $\mathcal{C}''$ such that, if $\sigma_{1}, \sigma_{2} \in \mathcal{C}''$, then the simplex underlying the $j$th column of $M_{\sigma_{1}}$ is identical to the simplex underlying the $j$th column of $M_{\sigma_{2}}$, for $j=1, \ldots, k$. Moreover, no two members of $\mathcal{C}''$ lie in the same $\Gamma$-orbit. 

Let $\sigma_{1}, \sigma_{2} \in \mathcal{C}''$. The above discussion implies that, for each $j \in \{ 1, \ldots, k \}$, the simplex underlying the $j$th column of $M_{\sigma_{1}}$ is equal to the simplex underlying the $j$th column of $M_{\sigma_{2}}$. This does not mean that the matrices $M_{\sigma_{1}}$ and $M_{\sigma_{2}}$ are identical, since the repeated entries in the columns may occur in different rows. The map  from columns to underlying simplices is finite-to-one, however, which implies that we can find an infinite set $\mathcal{C}'''$ of distinct simplices $\sigma$ such that all of the matrices $M_{\sigma}$ are identical. This is impossible, and it follows that the action of $\Gamma$ on 
$\Delta^{f}_{\mathcal{B},n}$ is cocompact.
\end{proof}

\subsection{The finiteness criterion and a template for establishing type $F_{\infty}$}
\label{subsection:Brown}

Here we briefly recall Brown's Finiteness Criterion, and compare the necessary hypotheses with our results so far.

\begin{theorem} \cite{Brown} \label{theorem:Brown} (Brown's Finiteness Criterion)
Let $X$ be a CW-complex. Let $G$ be a group acting on $X$. If
\begin{enumerate}
\item $X$ is $(n-1)$-connected; 
\item $G$ acts cellularly on $X$, and
\item there is a filtration $X_{1} \subseteq X_{2} \subseteq \ldots \subseteq X_{k} \subseteq \ldots \subseteq X$
such that
\begin{enumerate}
\item $X = \bigcup_{k=1}^{\infty} X_{k}$;
\item $G$ leaves each $X_{k}^{(n)}$ invariant and acts cocompactly on each $X_{k}^{(n)}$; 
\item each $p$-cell stabilizer has type $F_{n-p}$, and
\item for sufficiently large $k$, $X_{k}$ is $(n-1)$-connected,
\end{enumerate}
\end{enumerate}
then $G$ is of type $F_{n}$. \qed
\end{theorem}

\begin{theorem} \label{theorem:template} (A template for proving that $\Gamma$ has type $F_{\infty}$)
Let $\mathcal{B}$ be an expansion set over $X$; let $\widehat{S}$ be an inverse semigroup acting on $\mathcal{B}$. Let $\Gamma$ be the full-support subgroup of $\widehat{S}$.
We can let $\mathcal{B}$ be linear, or cyclic, or neither. 
Assume that Convention \ref{convention:SShatGamma} is in place. If
\begin{enumerate}
\item the vertices of $\Delta^{f}_{\mathcal{B}}$ are a directed set with respect to $\nearrow$;
\item for each $b \in \mathcal{B}$ and $v \in \mathcal{V}_{\mathcal{B}}$ such that $\{ b \} \nearrow v$, 
$lk_{\uparrow}(b,v)$ is contractible (when $\{ b \} \neq v$);
\item each stabilizer group $\widehat{S}_{b}$ ($b \in \mathcal{B}$) has type $F_{\infty}$, and each set $\mathcal{E}(b)$ is finite;
\item the action of $\Gamma$ on $\mathcal{B}$ has finitely many orbits;
\item $\mathcal{B}$ is rich in contractions and has the bounded contractions property,
\end{enumerate}
then $\Gamma$ has type $F_{\infty}$.
\end{theorem}

\begin{proof}
Properties (1) and (2) imply that $\mathcal{B}$ is an $n$-connected expansion set for all $n$ (Definition \ref{definition:nconnectedE}). It follows from Theorem \ref{theorem:nconnectivity} that $\Delta^{f}_{\mathcal{B}}$ is contractible. By Theorem \ref{theorem:partialaction}, $\Gamma$ acts by simplicial automorphisms on $\Delta^{f}_{\mathcal{B}}$. 

We note that $\{ \Delta^{f}_{\mathcal{B},k} \}$ is a 
filtration of $\Delta^{f}_{\mathcal{B}}$ such that
\[ \bigcup_{k=1}^{\infty} \Delta^{f}_{\mathcal{B},k}  = \Delta^{f}_{\mathcal{B}}, \] by Lemma \ref{lemma:filtrationlemma}, and that $\Gamma$ acts cocompactly on each $\Delta^{f}_{\mathcal{B},k}$ by Proposition \ref{proposition:cocompact}, and by properties (3) and (4).
Property (3) implies that cell stabilizers under the action of $\Gamma$ have type $F_{\infty}$ (by Corollary \ref{corollary:Fnstabilizers}). The connectivity of the complexes $\Delta^{f}_{\mathcal{B},k}$ tends to infinity with $k$ by Corollary \ref{corollary:connectivityeasy}. 

It follows from Theorem \ref{theorem:Brown} that $\Gamma$ has type $F_{\infty}$.
\end{proof}

\begin{remark}
(An alternate hypothesis)
In place of assumption (3) from Theorem \ref{theorem:template},
we can instead assume:
\begin{enumerate}
\item[(3')] each cell stabilizer group $\Gamma_{\sigma}$ has type $F_{\infty}$, and each
$\widehat{S}_{b}$ acts cocompactly on $\mathcal{E}(b)$.
\end{enumerate}
The conclusion is the same; i.e., $\Gamma$ has type $F_{\infty}$. We will need to do this in the case of the Lodha-Moore group.
\end{remark}

\begin{remark} \label{remark:discussion}
If $b_{1}, b_{2} \in \mathcal{B}$ are in the same $\widehat{S}$-orbit, then 
the groups $\widehat{S}_{b_{1}}$ and $\widehat{S}_{b_{2}}$ are easily seen to be conjugate. Similarly, each relative ascending link
$lk_{\uparrow}(b_{1},v_{1})$ is isomorphic to some relative ascending link
of the form $lk_{\uparrow}(b_{2},v_{2})$.

It follows that, in the presence of property (4) from 
\ref{theorem:template}, properties (2) and (3) require only a finite amount of checking (although (2) must still be proved for an arbitrary $v$). Notice also that the bounded contractions property is a consequence of (3) and (4).
\end{remark}

\section{Construction of expansion sets} \label{section:construction}

In this section, we turning an algebraic input (in the form of an inverse semigroup $S$ acting on $X$) into a topological output. The basic construction (Subsection \ref{subsection:basic}) is very general, but, as outlined in 
Subsection \ref{subsection:properties}, it lacks some of the properties demanded by 
Theorem \ref{theorem:template}. 

In the final subsections, we consider applications to various generalized Thompson groups,
including $V$, $nV$, R\"{o}ver's group, and the Lodha-Moore group. Our goal is to indicate some of the modifications that are needed in order to apply Theorem \ref{theorem:template} successfully. The list of examples is not intended to be exhaustive.

\subsection{The basic construction of an expansion set}
\label{subsection:basic}

Let $S$ be an inverse semigroup acting on the set $X$. We can assume that $X$ is linearly or cyclically ordered, or not. 
We will build
\begin{itemize}
\item an expansion set $\mathcal{B}$ over $X$ and
\item a larger inverse semigroup $\widehat{S}$ acting 
on $\mathcal{B}$
\end{itemize}
 from the given input; if $X$ is given a linear or cyclic ordering, we will build a linear or cyclic $\mathcal{B}$. Our construction will satisfy Convention \ref{convention:SShatGamma}. The resulting full-support groups $\Gamma$ will be $F$-like, $T$-like, or $V$-like, according to whether $\mathcal{B}$ is linear, cyclic, or neither.

Thus, let $X$ and $S$ be given. We let $\widehat{S}$ be the inverse semigroup determined by conditions (1)-(4) from 
Convention \ref{convention:SShatGamma}. (Recall that these 
conditions differ slightly, depending upon whether $X$ is endowed with an order, and that (4) imposes a constraint on $S$.) We let 
\[ \mathcal{D}^{+}_{S} = \{ D \mid D = dom(s),
\text{ for some }s \in S-\{0\} \}.\] 
These are the non-empty \emph{domains} of $S$. Define
\[ \mathcal{A} = \{ (\hat{s},D) \in \widehat{S} \times
\mathcal{D}^{+}_{S}\mid D \subseteq dom(\hat{s}) \}. \]
We impose an equivalence relation $\sim$ on $\mathcal{A}$ as follows:
\[ (\hat{s}_{1},D_{1}) \sim (\hat{s}_{2},D_{2}) \quad \Leftrightarrow \quad (\hat{s}_{2} \circ s)_{\mid D_{1}} = \hat{s}_{1 \mid D_{1}}, \]
where $s \in S$ is a bijection from $D_{1}$ to $D_{2}$. (The image of $s$ is itself necessarily a domain, since it is the domain of $s^{-1}$, which is in $S$ because $S$ is closed under inverses.) Equivalently, the equivalence relation $\sim$ is defined by a commutative diagram as in Figure \ref{commdiagram}.
\begin{figure}[!t]
\begin{center}
\begin{tikzpicture}
\draw[black,->] (.5,.6) -- (2.5,1.43);
\draw[black,->] (.5,2.5) -- (2.5,1.66);
\draw[black,->] (.2,2.25) -- (.2,.75);
\draw[black, ->] (1,1.8) arc (90:360:.3);
\node at (.2,2.5){$D_{1}$};
\node at (.2,.5){$D_{2}$};
\node at (-.05,1.5){$s$};
\node at (2.75,1.53){$X$};
\node at (1.5,.65){$\hat{s}_{2}$};
\node at (1.5,2.35){$\hat{s}_{1}$};
\end{tikzpicture}
\end{center}
\caption{The above commutative diagram determines the equivalence relation on $\mathcal{A}$.
The vertical arrow is labelled by a bijection from $S$, while the remaining arrows are members of $\widehat{S}$.}
\label{commdiagram}  
\end{figure}
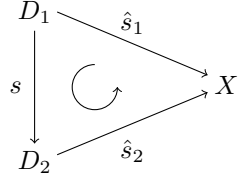
Let $\mathcal{B}$ be the set of equivalence classes of $\sim$, members of which are denoted by square brackets; i.e.,
\[ \mathcal{B} = \{ [\hat{s},D] \mid (\hat{s},D) \in \mathcal{A} \}. \]
For a given $b= [\hat{s},D]$, we define $supp(b) = \hat{s}(D)$. This assignment is easily seen to be well-defined. We note that, as required by Convention \ref{convention:SShatGamma}(5), $supp(b)$ is the domain
of some $\tilde{s} \in \widehat{S}$. This is because: (1) we can assume, up to equivalence, that $dom(\hat{s}) = D$, since an inverse semigroup is closed under restriction to subdomains (by Lemma \ref{lemma:restrictions}) and $S \subseteq \widehat{S}$, and (2) the image of $\hat{s}$ is the domain of $\hat{s}^{-1}$. It follows that $supp(b) = dom(\hat{s}^{-1})$, so that parts (1)-(5) of Convention \ref{convention:SShatGamma} are satisfied. 

We must next define the partially ordered set 
$\mathcal{E}(b)$, for each $b \in \mathcal{B}$.  It will be useful to introduce a preliminary definition. Let $b' = [\hat{s},D] \in v$, where $v$ is a vertex. Let $h:\widetilde{D} \rightarrow D$ be a bijection, where $h \in S$, and let $\mathcal{P} \subseteq \mathcal{D}^{+}_{S}$ ($|\mathcal{P}| \geq 2$) be a finite partition of $\widetilde{D}$ into finitely many domains. The act of replacing $b'$ with the members of the set
\[ \{ [\hat{s}h,E] \mid E \in \mathcal{P} \} \]
is called performing an \emph{elementary expansion} at $b'$. (In simpler terms, an elementary expansion is performed by first replacing $[\hat{s},D]$ with any equivalent pair $[\hat{s}',D']$, and then ``subdividing'' $[\hat{s}',D']$ using a partition of $D'$ into finitely many domains.) The definition of $\mathcal{E}(b)$ is inductive:
\begin{itemize}
\item $\{ b \} \in \mathcal{E}(b)$;
\item if $v \in \mathcal{E}(b)$ and $v'$ is the result of performing an elementary expansion at some $b \in v$,
then $v' \in \mathcal{E}(b)$ and we write $v \preceq v'$. 
\end{itemize}
The partial order $\leq$ on $\mathcal{E}(b)$ is 
obtained by taking the transitive closure of $\preceq$. It is not difficult to see that $\mathcal{E}(b)$ satisfies all of the requirements of Definition \ref{definition:expansionsets}, so
$\mathcal{B}$ is an expansion set. This is the basic construction of an expansion set $\mathcal{B}$.

We still need to show that $\widehat{S}$ acts on $\mathcal{B}$. For a given $f \in \widehat{S}$ and
$[g,D] \in \mathcal{B}$ such that $g(D) \subseteq dom(f)$, we let 
\[ f \cdot [g,D] = [fg,D]. \]
(We leave $f \cdot [g,D]$ undefined if $g(D) \not \subseteq dom(f)$.) Properties (1)-(4) from Definition \ref{definition:actionsonexpansionsets} are easily checked.
We have therefore produced the desired $\mathcal{B}$, $\widehat{S}$, and action of $\widehat{S}$ on $\mathcal{B}$. (Note that, in the linear or cyclic cases, one requires that supports be connected. But this is also an easy consequence of Convention \ref{convention:SShatGamma} and the definition of $\mathcal{B}$.) 

\subsection{Properties of the basic construction}
\label{subsection:properties}

Let us compare the properties of the basic construction with the hypotheses of Theorem \ref{theorem:template}. 

It is easiest to consider (1) and (2) from Theorem \ref{theorem:template} together. 

\begin{proposition} \label{proposition:1and2}
The basic construction satisfies properties (1) and (2) from
Theorem \ref{theorem:template}

\end{proposition}

\begin{proof}
Let $[f,D] \in \mathcal{B}$. We can assume, up to equivalence, that $dom(f) = D$ by Lemma \ref{lemma:restrictions}. Property (2) of Convention \ref{convention:SShatGamma} implies that $f$ is a finite disjoint union of members of $S$, so there is a finite partition 
$\mathcal{P} = \{ D_{1}, \ldots, D_{m} \}  
\subseteq \mathcal{D}^{+}_{S}$ of $D$ such that $f_{\mid D_{i}} = s_{i}$ for some $s_{i} \in S$, for each $i=1,\ldots,m$. We let $E_{i}$ ($i=1,\ldots,m$) denote the images of the $s_{i}$. There is an elementary  expansion that replaces the pair $[f,D]$ with the members of the set
\[ \{ [s_{1}, D_{1}], \ldots, [s_{m}, D_{m}] \}. \]
Next, note that, for each $i$, $[s_{i},D_{i}] = [id_{E_{i}},E_{i}]$. (One sets $D_{1} = D_{i}$, $D_{2} = E_{i}$, $s = s_{i}$, 
$\hat{s}_{1} = s_{i}$, and $\hat{s}_{2} = id_{E_{i}}$ in Figure \ref{commdiagram}.)

It follows from this that, for each $b \in \mathcal{B}$, there is some
$v \in \mathcal{E}(b)$ such that 
\[ v = \{ [id_{E_{1}},E_{1}], \ldots, [id_{E_{m}},E_{m}] \}, \]
for some partition $\mathcal{P}' = \{ E_{1}, \ldots, E_{m} \}$ of $supp(b)$. A vertex $v$ of the above form can essentially be identified with the partition $\mathcal{P}'$; we let the notation reflect this by writing $v_{\mathcal{P}'}$ in place of $v$. It is straightforward to show that two vertices $v_{\mathcal{P}_{1}}$ and $v_{\mathcal{P}_{2}}$ with the same support both expand to $v_{\mathcal{P}_{1} \wedge \mathcal{P}_{2}}$ by a sequence of elementary expansions. (The crucial observation here is that the intersection of two domains is a domain, by Lemma \ref{lemma:restrictions}.) It follows that each partial order
$\mathcal{E}(b)$ is a directed set, and that the vertex set of 
$\Delta^{f}_{\mathcal{B}}$ is directed with respect to the order $\nearrow$. Properties (1) and (2) from 
Theorem \ref{theorem:template} follow directly.
\end{proof}

The situation regarding property (3) from Theorem \ref{theorem:template} is a bit murkier. We can prove:

\begin{proposition}
\label{proposition:iso}
(Isomorphism type of $\widehat{S}_{b}$)
For each $b = [f,D] \in \mathcal{B}$ (where $\mathcal{B}$ denotes the basic construction), we have an isomorphism of groups
\[ \widehat{S}_{b} \cong \{ s \in S \mid dom(s) = im(s) = D \}, \]
where the operation on the right is composition of functions.
\end{proposition}

\begin{proof}
We denote the group on the right side by $S_{D}$. Let $b' = [id_{D},D]$.
By Remark \ref{remark:discussion}, it suffices to show that 
the stabilizer $\widehat{S}_{b'}$ is isomorphic $S_{D}$. (Note that 
$S_{D}$ is indeed a group: since $D \in \mathcal{D}^{+}_{S}$ by the definition of $\mathcal{B}$, there is some $s$ having domain $D$, and it follows that $s^{-1}s = id_{D} \in S_{D}$. Thus, $S_{D}$ is non-empty. The properties of $S$ as an inverse semigroup guarantee the remaining group properties.)

We show that, indeed $\widehat{S}_{b'} = S_{D}$. First, let 
$\hat{s} \in \widehat{S}_{b'}$. It follows that $\hat{s}$ is a self-bijection of the set $D$, and that $[\hat{s},D] = [id_{D},D]$. Letting 
$D_{1} = D_{2} = D$, $\hat{s}_{1} = \hat{s}$, $\hat{s}_{2} = id_{D}$ in Figure \ref{commdiagram}, we find that $\hat{s}$ is the only possible label of the vertical arrow. Thus $\hat{s} \in S$, and so
$\hat{s} \in S_{D}$. The reverse inclusion $S_{D} \subseteq \widehat{S}_{b'}$ follows directly from the inclusion $S \subseteq \widehat{S}$.
\end{proof}

We also have:
\begin{proposition}
(The action of $\widehat{S}_{b}$ on $\mathcal{E}(b)$ is usually not cocompact) \label{proposition:nococompact} 
If the set
\[ \{ |v| \mid v \in \mathcal{E}(b) \} \]
is infinite, then the action of $\widehat{S}_{b}$ on $\mathcal{E}(b)$ is not cocompact (in particular, $\mathcal{E}(b)$  is not finite).
\end{proposition}

\begin{proof}
This follows immediately from the fact that the action of $\widehat{S}_{b}$ is height-preserving.
\end{proof}

We can prove the following restricted version of (4) from Theorem \ref{theorem:template}:

\begin{proposition} \label{proposition:orbits} 
(Orbits under the action of $\widehat{S}$)
Two members $b_{1} = [f_{1},D_{1}], b_{2}= [f_{2},D_{2}] \in \mathcal{B}$ (where $\mathcal{B}$ denotes the basic construction) are in the same $\widehat{S}$-orbit if and only if there is
some $s \in S$ such that $dom(s) = D_{1}$ and $im(s) = D_{2}$.
\end{proposition}

\begin{proof}
Suppose first that there is $s \in S$ such that $s:D_{1} \rightarrow D_{2}$ is a bijection. We have
\[ b_{2} = [f_{2},D_{2}] = [f_{2}s,D_{1}] \]
by the definition of $\sim$. It follows that $f_{2}sf^{-1}_{1} \in \widehat{S}$ sends $b_{1}$ to $b_{2}$. 

Conversely, suppose that there is $\hat{s} \in \mathcal{S}$ such that
$\hat{s} \cdot b_{1}= b_{2}$. We assume that $dom(f_{1}) = D_{1}$ and $dom(f_{2}) = D_{2}$, without loss of generality. The definition of $\sim$ implies that
$f^{-1}_{2} \hat{s} f_{1} \in S$, and has the required domain $D_{1}$ and image $D_{2}$.
\end{proof}

Finally, we note that (5) from Theorem \ref{theorem:template} will almost never be satisfied by the basic construction:

\begin{proposition}
\label{proposition:nobound} 
The bounded contractions property is never satisfied when
\[ \{ |v| \mid v \in \mathcal{E}(b) \} \]
for some $b \in \mathcal{E}(b)$.
\end{proposition}

\begin{proof}
This is true as a matter of definition, whether $\mathcal{B}$ is the basic construction or not.
\end{proof}

\begin{remark} (Strategy for defining $\mathcal{B}$)
\label{remark:strategy}
The results of this subsection suggest some of the strategy involved in modifying the basic construction:
\begin{enumerate}
\item It is best to keep the inverse semigroup $S$ as small as possible, since
the choice of $S$ directly affects the equivalence relation $\sim$ and also
the stabilizer subgroups $\widehat{S}_{b}$ (the latter by Proposition 
\ref{proposition:iso}). In practice, though, one may not have a great deal of flexibility on this matter.

Nothing prevents us from simply letting $S = \widehat{S}$ in the basic construction of $\mathcal{B}$. The problem with doing so is that stabilizer groups become very large: indeed, the stabilizer group of $[id_{X},X]$ would be the entire full-support group $\Gamma$ with this set-up. Considerations like these are the reason to make a distinction between $S$ and $\widehat{S}$ in the first place.

 The ``$S$-structures" from \cite{FH2} were devised exactly in order to limit the size of the stabilizer groups $\widehat{S}_{b}$. We will avoid a systematic treatment, but refer the interested reader to \cite{FH2} for details. Our treatment of R\"{o}ver's group (Subsection \ref{subsection:rover}) uses some similar ideas.

\item It is likewise better to make $\mathcal{E}(b)$ as small as possible, by (3).
The challenge is to shrink the size of $\mathcal{E}(b)$ without losing properties (1) and (2) from Theorem \ref{theorem:template}, while keeping property (4) from Definition \ref{definition:actionsonexpansionsets}. It is often possible to achieve this in practice, as we will soon see.

\end{enumerate}
\end{remark}

\begin{remark} \label{remark:comments}(Further comments)

\begin{enumerate}
\item In \cite{FH2}, two domains $D_{1}, D_{2} \in \mathcal{D}^{+}_{S}$ are said to have the \emph{same domain type} essentially if there is a bijection $s: D_{1} \rightarrow D_{2}$ such that $s \in S$. (The actual definition involves $S$-structures, but it is very similar to the above in spirit.) With this definition, we can restate Proposition \ref{proposition:orbits} as follows: two members 
$[f_{1},D_{1}], [f_{2},D_{2}]$ are in the same $\widehat{S}$-orbit if and only if $D_{1}$ and $D_{2}$ have the same domain type.

We note that the $\widehat{S}$-orbit of $b \in \mathcal{B}$ is not necessarily quite the same 
as the $\Gamma$-orbit of $b$. The potential difference arises from the fact that a given 
$\hat{s} \in \widehat{S}$ need not extend to a full-support $\gamma \in \Gamma$. This is a minor issue in practice.

\item Instead of proving (4) from Theorem \ref{theorem:template}, it is sometimes easier to show directly that 
$\Gamma$ acts on the vertices of $\Delta^{f}_{\mathcal{B},k}$ with finitely many orbits. A close reading of the proof of Proposition \ref{proposition:cocompact} will show that the latter property can replace (4) in the hypothesis of Proposition \ref{proposition:cocompact}. 

\item The ``rich-in-contractions" condition may be the most restrictive hypothesis in 
Theorem \ref{theorem:template}. While this condition figures heavily in our analysis of the descending link, it is not really necessary: in \cite{FH2} the authors consider groups like $QV$, which fail to have the ``rich-in-contractions" property. (The first proof that $QV$ has type $F_{\infty}$ seems to be due to \cite{QV}. This extended some results from \cite{Nucinkis}, where the groups $QF$ and $\widetilde{Q}V$ (among others) were shown to have type $F_{\infty}$.) The descending link in the associated complexes is instead analyzed with the aid of ``type vectors", which encode the domain types that occur in a given vertex $v \in \Delta_{\mathcal{B}}$. The resulting analysis is still entirely combinatorial in nature, but more involved. We refer the reader to Example 8.12 from \cite{FH2} for further details.

\end{enumerate}

\end{remark}

\subsection{Thompson's group $V$} \label{subsection:thompson} 

\subsubsection{The set-up} \label{subsubsection:setup} 

Let $X$ be the set of all infinite binary strings, and let $X^{fin}$ denote the set of all finite binary strings, including the empty string. For given $\omega_{1}, \omega_{2} \in X^{fin}$, we let $\sigma_{\omega_{1}}^{\omega_{2}}$  denote the partial transformation of $X$ that replaces the prefix $\omega_{1}$ of an infinite binary string $\omega$ with the prefix $\omega_{2}$. If the $\omega$ in question does not begin with the prefix $\omega_{1}$, then $\sigma_{\omega_{1}}^{\omega_{2}} (\omega)$ is not defined; i.e., 
 \[ \sigma^{\omega_{2}}_{\omega_{1}}(\omega) = 
  \begin{cases} \omega_{2}a_{1}a_{2}\ldots & \text{if } \omega = \omega_{1}a_{1}a_{2}\ldots \\

 \text{undefined} & \text{otherwise} \end{cases} \]
 where the $a_{i}$ are binary digits ($0$ or $1$). We let $0$ also denote the null transformation of $X$, which by definition is undefined for every input. It is then straightforward to check that the set
 \[ S = \{ \sigma^{\omega_{2}}_{\omega_{1}} \mid \omega_{1}, \omega_{2} \in X^{fin} \} \cup \{ 0 \} \]
 is an inverse semigroup acting on $X$. For a given finite string $\alpha \in X^{fin}$, we let 
 \[ X_{\alpha} = \{ \omega \in X \mid \omega = \alpha a_{1}a_{2}\ldots \}. \]
Thus, $X_{\alpha}$ is the set of infinite binary strings beginning with $\alpha$. Clearly, the set of all $X_{\alpha}$ ($\alpha \in X^{fin}$) is the set of non-empty
 domains $\mathcal{D}^{+}_{S}$. 
 
  We now extend $S$ to a larger inverse semigroup $\widehat{S}$, following Convention \ref{convention:SShatGamma}. If we choose to order $X$ linearly (i.e., lexicographically), then Convention \ref{convention:SShatGamma} leads us to an inverse semigroup $\widehat{S}$ whose full-support subgroup is Thompson's group $F$; if we choose the natural cyclic ordering, then we are led to a definition of Thompson's group $T$. We will ignore any order for the sake of this discussion; thus Convention \ref{convention:SShatGamma} defines $\widehat{S}$ to be the unrestricted disjoint union of members of $S$, and the full-support subgroup of $\widehat{S}$ will therefore be Thompson's group $V$.  
This choice affects the discussion in only minor ways.
 
The expansion set $(\mathcal{B}, X, supp, \mathcal{E})$ is defined as in Subsection \ref{subsection:basic}, but with the difference that $\mathcal{E}(b)$ is greatly restricted.  For each $b = [f,X_{\alpha}] \in \mathcal{B}$, we define
 \[ \mathcal{E}(b) = \{ \{ [f,X_{\alpha}] \}, \{ [f,X_{\alpha0}], [f,X_{\alpha1}] \} \}. \]
(The definition is the same for $F$, $T$, and $V$.) Note that $\alpha0$ denotes the concatenation of the string $\alpha$ with the additional symbol $0$; similarly for $\alpha1$.  
The partial order on $\mathcal{E}(b)$ is defined by the inequality
 \[ \{ [f,X_{\alpha}] \} \leq \{  [f,X_{\alpha0}], [f,X_{\alpha1}] \}, \]
which indeed is the only partial ordering that is compatible with Definition \ref{definition:expansionsets}(1).
 
\subsubsection{Tree-pair description} \label{subsubsection:treepair} 

It may be helpful to attach some intuition to the ideas of \ref{subsubsection:setup}. We will do this using the familiar language of tree-pairs \cite{CFP}. 

Let $b = [f,X_{\alpha}] \in \mathcal{B}$. We can represent $b$ by a special type of tree-pair $(T_{1},T_{2})$, as follows:
\begin{itemize}
\item $T_{1}$ and $T_{2}$ are finite binary trees;
\item the tree $T_{2}$ has at least as many leaves as $T_{1}$;
\item the vertex $\alpha$ is either an interior node or a leaf in $T_{1}$;
\item all leaves below the node $\alpha$ in $T_{1}$ are numbered $1, \ldots, n$ consecutively from left to right (for some $n$), with each leaf receiving a label (if $\alpha$ itself is a leaf, then it is the only leaf of $T_{1}$ to be numbered);
\item if $T_{1}$ has $n$ numbered leaves, then an (unordered) subset of the leaves of $T_{2}$ receive a label from the set
$\{1, 2, \ldots, n \}$.
\end{itemize}

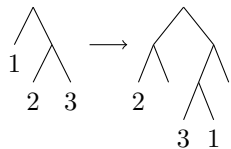
\begin{figure}[!b]

\begin{center}
\begin{tikzpicture}
\draw[black](.5,1.5) -- (.25,1); 
\node at (.25,.75){$1$};
\draw[black](.5,1.5) -- (1,.5);
\draw[black](.5,.5) -- (.75,1);
\node at (.5,.25){$2$};
\node at (1,.25){$3$};

\draw[black,->](1.25,1) -- (1.75,1); 

\draw[black](2.5,1.5) -- (2.1,1); 
\draw[black](2.1,1) -- (1.9,.5); 
\draw[black](2.1,1) -- (2.3,.5); 
\draw[black](2.5,1.5) -- (2.9,1); 
\draw[black](2.9,1) -- (2.7,.5); 
\draw[black](2.9,1) -- (3.1,.5); 
\draw[black](2.7,.5) -- (2.5,0); 
\draw[black](2.7,.5) -- (2.9,0); 
\node at (1.9,.25){$2$};
\node at (2.5,-.25){$3$};
\node at (2.9,-.25){$1$};
\end{tikzpicture}

\end{center} 

\caption{The given tree-pair represents a member $[g,X]$ of $\mathcal{B}$, the expansion set associated to Thompson's group $V$.}
\label{fig1}
\end{figure}

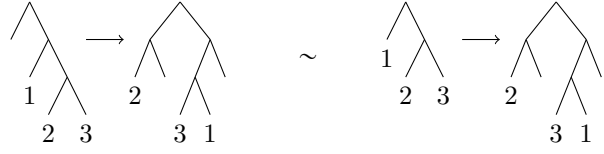
\begin{figure}[!t] 
\begin{center}
\begin{tikzpicture}
\draw[black](5.5,1.5) -- (5.25,1); 
\node at (5.25,.75){$1$};
\draw[black](5.5,1.5) -- (6,.5);
\draw[black](5.5,.5) -- (5.75,1);
\node at (5.5,.25){$2$};
\node at (6,.25){$3$};

\draw[black,->](6.25,1) -- (6.75,1); 

\draw[black](7.5,1.5) -- (7.1,1); 
\draw[black](7.1,1) -- (6.9,.5); 
\draw[black](7.1,1) -- (7.3,.5); 
\draw[black](7.5,1.5) -- (7.9,1); 
\draw[black](7.9,1) -- (7.7,.5); 
\draw[black](7.9,1) -- (8.1,.5); 
\draw[black](7.7,.5) -- (7.5,0); 
\draw[black](7.7,.5) -- (7.9,0); 
\node at (6.9,.25){$2$};
\node at (7.5,-.25){$3$};
\node at (7.9,-.25){$1$};

\draw[black](.5,1.5) -- (.25,1); 
\draw[black](.5,1.5) -- (1,.5);
\draw[black](.5,.5) -- (.75,1);
\draw[black](1,.5) -- (.75,0); 
\draw[black](1,.5) -- (1.25,0); 
\node at (.5,.25){$1$};
\node at (.75,-.25){$2$};
\node at (1.25,-.25){$3$};

\draw[black,->](1.25,1) -- (1.75,1); 

\draw[black](2.5,1.5) -- (2.1,1); 
\draw[black](2.1,1) -- (1.9,.5); 
\draw[black](2.1,1) -- (2.3,.5); 
\draw[black](2.5,1.5) -- (2.9,1); 
\draw[black](2.9,1) -- (2.7,.5); 
\draw[black](2.9,1) -- (3.1,.5); 
\draw[black](2.7,.5) -- (2.5,0); 
\draw[black](2.7,.5) -- (2.9,0); 
\node at (1.9,.25){$2$};
\node at (2.5,-.25){$3$};
\node at (2.9,-.25){$1$};

\node at (4.2,.75){$\sim$};

\end{tikzpicture}
\end{center}

\caption{On the left, we have tree notation for a pair $[f,X_{1}]$ and, on the right, tree notation for the equivalent pair 
$[f \sigma_{\epsilon}^{1}, X]$.}
\label{fig2}
\end{figure}

In Figure \ref{fig1}, we have pictured a typical member $[g,X]$ of $\mathcal{B}$, represented in the above way. Interior nodes and leaves of a finite binary tree have the obvious labellings by finite binary strings. The topmost vertex (called the \emph{root}) is labelled by the empty string. If the label of a given vertex is $\omega$, then its left and right children are labelled $\omega0$ and $\omega1$, respectively. With this convention, the (implicit) labels on the leaves of the left tree are $0$, $10$, and $11$ (reading from left to right). The labels for the leaves of the right tree are $00$, $01$, $100$, $101$, and $11$. The numbers $1,2,3$ define the transformations $\sigma_{\omega_{1}}^{\omega_{2}}$ that ``locally determine" $g$. Here, the ``$1$"s determine the transformation $\sigma_{0}^{101}$, while the ``$2$"s and ``$3$"s determine
$\sigma_{10}^{00}$ and $\sigma_{11}^{100}$, respectively. The transformation $g$ is defined by the equation
\[ g = \sigma_{0}^{101} \coprod \sigma_{10}^{00} \coprod \sigma_{11}^{100}.\]

\begin{figure}[!b]
\begin{center}
\begin{tikzpicture}
\draw[black](.5,1.5) -- (.25,1); 
\node at (.25,.75){$1$};
\draw[black](.5,1.5) -- (1,.5);
\draw[black](.5,.5) -- (.75,1);
\node at (.5,.25){$2$};
\node at (1,.25){$3$};

\draw[black,->](1.25,1) -- (1.75,1); 

\draw[black](2.5,1.5) -- (2.1,1); 
\draw[black](2.1,1) -- (1.9,.5); 
\draw[black](2.1,1) -- (2.3,.5); 
\draw[black](2.5,1.5) -- (2.9,1); 
\draw[black](2.9,1) -- (2.7,.5); 
\draw[black](2.9,1) -- (3.1,.5); 
\draw[black](2.7,.5) -- (2.5,0); 
\draw[black](2.7,.5) -- (2.9,0); 
\node at (2.9,-.25){$1$};
\node at (1.9,.25){$2$};
\node at (2.5,-.25){$3$};

\node at (4,.75){{\huge $\leadsto$}};

\draw[black](5.5,2.5) -- (5.25,2); 
\node at (5.25,1.75){$1$};
\draw[black](5.5,2.5) -- (5.75,2);
\draw[black,dashed](5.75,2) -- (6,1.5);
\draw[black,dashed](5.5,1.5) -- (5.75,2);

\draw[black,->](6.25,2) -- (6.75,2); 

\draw[black](7.5,2.5) -- (7.1,2); 
\draw[black,dashed](7.1,2) -- (6.9,1.5); 
\draw[black,dashed](7.1,2) -- (7.3,1.5); 
\draw[black](7.5,2.5) -- (7.9,2); 
\draw[black](7.9,2) -- (7.7,1.5); 
\draw[black](7.9,2) -- (8.1,1.5); 
\draw[black](7.7,1.5) -- (7.5,1); 
\draw[black](7.7,1.5) -- (7.9,1); 

\node at (7.9,.75){$1$};

\draw[black](5.5,.25) -- (5.25,-.25); 

\draw[black](5.5,.25) -- (6,-.75);
\draw[black](5.5,-.75) -- (5.75,-.25);
\node at (5.5,-1){$2$};
\node at (6,-1){$3$};

\draw[black,->](6.25,-.25) -- (6.75,-.25); 

\draw[black](7.5,.25) -- (7.1,-.25); 
\draw[black](7.1,-.25) -- (6.9,-.75); 
\draw[black](7.1,-.25) -- (7.3,-.75); 
\draw[black](7.5,.25) -- (7.9,-.25); 
\draw[black](7.9,-.25) -- (7.7,-.75); 
\draw[black](7.9,-.25) -- (8.1,-.75); 
\draw[black](7.7,-.75) -- (7.5,-1.25); 
\draw[black](7.7,-.75) -- (7.9,-1.25); 
\node at (6.9,-1){$2$};
\node at (7.5,-1.5){$3$};

\end{tikzpicture}
\end{center}
\caption{This figure illustrates the process of expansion from the pair $[g,X]$ of Figure \ref{fig1}. The top tree pair on the right is the restriction of $g$ to the left half of $X$, and the bottom pair is the restriction of $g$ to the right. The dashed carets indicate a portion of the tree pair that may be omitted without any loss of information.}
\label{fig3}
\end{figure}
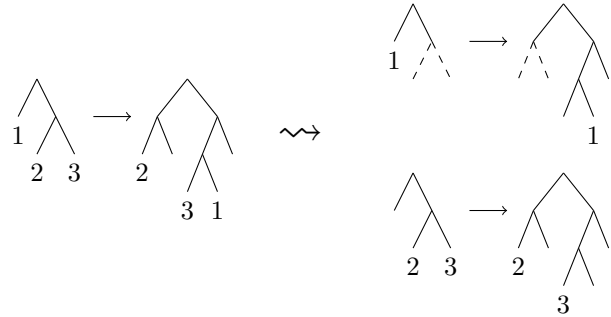

It is completely straightforward to read the support of a pair $[f,X_{\omega}]$ from its tree pair representative. One simply considers the labelled leaves in the range tree, which are unchanged under $\sim$, and therefore lead to a well-defined support. (Unlabelled leaves are not part of the image of $f$.) For instance, 
$supp([g,X]) = X_{00} \cup X_{100} \cup X_{101}$ if $[g,X]$ is as in Figure \ref{fig1}.

We note that the above tree-pair representatives are not unique. In addition to the usual rules for cancelling carets (as in \cite{CFP}), any given tree pair 
has a representative of the form $[g,X]$, up to $\sim$, because of the identity
\[ [f,X_{\alpha}] = [f\sigma_{\epsilon}^{\alpha},X], \]
where $\epsilon$ denotes the binary empty string. See Figure \ref{fig2}.
We note that the equivalence relation $\sim$ enables us to modify the domain tree by a suitable ``shift", but the range tree is unchanged.

The tree pair description of $\mathcal{E}$ is very simple: for a given pair $[g,X]$, expansion from $[g,X]$ consists of replacing 
$[g,X]$ with the restrictions to the left and right halves of $X$; i.e., $[g,X]$ is replaced with $[g_{\mid X_{0}}, X_{0}]$ and $[g_{\mid X_{1}}, X_{1}]$. Figure \ref{fig3} illustrates the resulting tree pairs if the original pair $[g,X]$ is the one from Figure \ref{fig1}. The tree pairs on the right side of Figure \ref{fig3} are simply obtained by deleting the appropriate labels, where, for instance, the ``$2$" and ``$3$" are deleted from the top tree pair since the leaves in the domain tree with those labels lie outside of $X_{0}$. We similarly delete the label ``$1$" from the bottom tree pair since the leaf of the domain tree with that label lies outside of $X_{1}$.  

Finally, we note for future reference that expansion is fully reversible; i.e., if two pairs $[g_{1},X]$ and $[g_{2},X]$ have disjoint support, then one can find a pair $[g,X]$ such that simple expansion from the latter pair creates $[g_{1},X]$ and $[g_{2},X]$. (We are stating this for the group $V$, but it would remain true for $F$ and $T$ in a suitably restricted sense.) The process is straightforward: one simply rewrites $[g_{1},X]$ and $[g_{2},X]$ in such a way that the second coordinates do not overlap, and then ``takes the union". (In fact, there are at least two different ways to do this.) We illustrate the process of contraction in Figure \ref{fig4}. The idea here is to alter the pair $[g_{1},X]$ into an equivalent pair with the domain $X_{0}$, and similarly alter $[g_{2},X]$ into a pair with domain $X_{1}$.

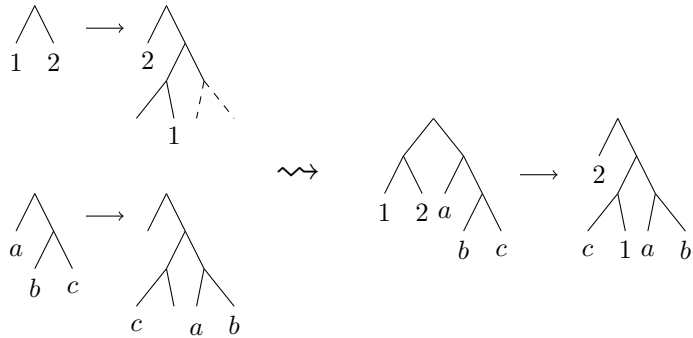
\begin{figure}[!t]
\begin{center}
\begin{tikzpicture}

\draw[black](.75,5.25) -- (1,5.75); 
\draw[black](1.25,5.25) -- (1,5.75); 
\node at (.75,5){$1$};
\node at (1.25,5){$2$};

\draw[black,->](1.68,5.45) -- (2.18,5.45); 

\draw[black](2.75,5.75) -- (2.5,5.25); 
\draw[black](2.75,5.75) -- (3,5.25); 
\draw[black](3,5.25) -- (2.75,4.75); 
\draw[black](3,5.25) -- (3.25,4.75); 
\draw[black](2.75,4.75) -- (2.35,4.25); 
\draw[black](2.75,4.75) -- (2.85,4.25);  
\draw[black,dashed](3.25,4.75) -- (3.15,4.25);   
\draw[black,dashed](3.25,4.75) -- (3.65,4.25); 
\node at (2.5,5.02){$2$};
\node at (2.85,4.02){$1$};

\draw[black](.75,2.75) -- (1,3.25); 
\draw[black](1.25,2.75) -- (1,3.25); 
\draw[black](1.25,2.75) -- (1,2.25); 
\draw[black](1.25,2.75) -- (1.5,2.25); 
\node at (.75,2.5){$a$};
\node at (1,2){$b$};
\node at (1.5,2){$c$};

\draw[black,->](1.68,2.95) -- (2.18,2.95); 

\draw[black](2.75,3.25) -- (2.5,2.75); 
\draw[black](2.75,3.25) -- (3,2.75); 
\draw[black](3,2.75) -- (2.75,2.25); 
\draw[black](3,2.75) -- (3.25,2.25); 
\draw[black](2.75,2.25) -- (2.35,1.75); 
\draw[black](2.75,2.25) -- (2.85,1.75);  
\draw[black](3.25,2.25) -- (3.15,1.75);   
\draw[black](3.25,2.25) -- (3.65,1.75); 
\node at (3.15,1.45){$a$};
\node at (3.65,1.5){$b$};
\node at (2.35,1.5){$c$};

\node at (4.5,3.5){{\huge $\leadsto$}}; 


\draw[black](6.3,4.25) -- (5.9,3.75); 
\draw[black](6.3,4.25) -- (6.7,3.75); 

\draw[black](5.65,3.25) -- (5.9,3.75); 
\draw[black](6.15,3.25) -- (5.9,3.75); 
\node at (5.65,3){$1$};
\node at (6.15,3){$2$};

\draw[black](6.45,3.25) -- (6.7,3.75); 
\draw[black](6.95,3.25) -- (6.7,3.75); 
\draw[black](6.95,3.25) -- (6.7,2.75); 
\draw[black](6.95,3.25) -- (7.2,2.75); 
\node at (6.45,3){$a$};
\node at (6.7,2.5){$b$};
\node at (7.2,2.5){$c$};

\draw[black,->](7.45,3.5) -- (7.95,3.5); 

\draw[black](8.75,4.25) -- (8.5,3.75); 
\draw[black](8.75,4.25) -- (9,3.75); 
\draw[black](9,3.75) -- (8.75,3.25); 
\draw[black](9,3.75) -- (9.25,3.25); 
\draw[black](8.75,3.25) -- (8.35,2.75); 
\draw[black](8.75,3.25) -- (8.85,2.75);  
\draw[black](9.25,3.25) -- (9.15,2.75);   
\draw[black](9.25,3.25) -- (9.65,2.75); 
\node at (8.5,3.5){$2$};
\node at (8.35,2.5){$c$};
\node at (8.85,2.5){$1$};
\node at (9.15,2.5){$a$};
\node at (9.65,2.5){$b$};

\end{tikzpicture}
\end{center}
\caption{This figure illustrates contraction, the reverse of expansion. The tree pairs on the left represent members of $\mathcal{B}$, both expressed in canonical form. To contract, we replace these pairs with equivalent ones having domains $X_{0}$ and $X_{1}$, respectively, and then take the union.}
\label{fig4}
\end{figure}

\subsubsection{Proof that $V$ has type $F_{\infty}$}

\begin{theorem} \cite{Brown} \label{theorem:VFinfinity} 
The group $V$ has type $F_{\infty}$.
\end{theorem}

\begin{proof}
We check the hypotheses of Theorem \ref{theorem:template}, beginning with the easiest ones.
First, we note that properties (1)-(5) from Convention \ref{convention:SShatGamma} are satisfied, because the construction of $\mathcal{B}$ is almost identical to the basic construction
in Subsection \ref{subsection:basic}. The proof that $\widehat{S}$ acts on $\mathcal{B}$ (also part of Convention \ref{convention:SShatGamma}) is entirely straightforward, so all parts of Convention \ref{convention:SShatGamma} are satisfied.

We note that the rich-in-contractions property is satisfied with $C_{1} = 2$ by Figure \ref{fig4} and the surrounding discussion. The bounded contractions property is likewise satisfied with 
constant $C_{0} = 2$, since the largest height of any vertex in any $\mathcal{E}(b)$ is $2$. Thus (5) from Theorem \ref{theorem:template} is satisfied.

Property (3) from Theorem \ref{theorem:template} is also easy to check: the group
$\widehat{S}_{b}$ is isomorphic to $S_{X}$ by Proposition \ref{proposition:iso}, and the latter group is trivial by the definition of $S$. The finiteness of $\mathcal{E}(b)$ is clear. 
This establishes (3).

We next consider (4) from Theorem \ref{theorem:template}. We claim that the action of $\Gamma$ 
on $\mathcal{B}$ has exactly two orbits: one containing the elements $[g,X]$ such that
$g(X) = X$, and the other consisting of the elements $[g,X]$ such that $g(X)$ is any proper subset of $X$. (Recall that $[g,X]$ is the form a general member of $\mathcal{B}$, up to equivalence, by the discussion in \ref{subsubsection:treepair}.) Indeed, it is rather clear that an element from one class cannot be in the $\Gamma$-orbit of an element from the other. Suppose that $[g_{1},X], [g_{2},X]$ are both such that $g_{1}(X) = X = g_{2}(X)$. We can then simply take $g_{2}g_{1}^{-1}$. Indeed, the latter is a member of $\Gamma$ by hypothesis, and it carries $[g_{1},X]$ to $[g_{2},X]$. Thus, any two full-support members of $\mathcal{B}$ are in the same orbit. If $[g_{1},X]$ and $[g_{2},X]$ are such that $g_{1}(X)$ and $g_{2}(X)$ are both proper subsets of $X$, then it is an easy exercise to find $f_{1}$ and $f_{2} \in \Gamma$ such that $f_{1}g_{1}(X) = f_{2}g_{2}(X) = X_{1}$. We can then define $\hat{g} \in \Gamma$ equal to  $f_{2}g_{2}g_{1}^{-1}f_{1}^{-1}$ on $X_{1}$, and define $\hat{g} = id_{X_{0}}$ on $X_{0}$. It follows that $\hat{g} \cdot [f_{1}g_{1},X] = [f_{2}g_{2},X]$, which shows that $[g_{1},X]$ and $[g_{2},X]$ are in the same $\Gamma$-orbit. This proves (4). 

Next consider (2) from Theorem \ref{theorem:template}. If $\{ b \} \nearrow v$ and $\{ b \} \neq v$, then there is a sequence
\[ \{ b \} = v_{0}, v_{1}, \ldots, v_{m} = v, \]
where each $v_{i}$ is a member of the upward link of $v_{i-1}$, for $i=1,\ldots, m$, where $m>0$. We note that $v_{1}$ is uniquely identified by this condition: it is the height-$2$ element of $\mathcal{E}(b)$, which is indeed the entirety of $lk_{\uparrow}(b)$. It follows that
$lk_{\uparrow}(b,v)$ is a single point, and therefore contractible.

Property (1) from Theorem \ref{theorem:template} can be verified in much the same way as in Proposition \ref{proposition:1and2}. The only difference is that, in Proposition \ref{proposition:1and2}, one has that a vertex of the form $v_{\mathcal{P}} = \{ [id_{E_{1}}, E_{1}], \ldots, [id_{E_{m}},E_{m}] \}$ lies in the set $\mathcal{E}(b)$, whereas, in general, we have only
\[ \{ b \}= v_{0}, v_{1}, \ldots, v_{m} = v_{\mathcal{P}}, \]
where each $v_{i}$ is in the upward link of $v_{i-1}$. (The distinction here is that we have an ascending edge-path connecting $\{ b\}$ to $v_{\mathcal{P}}$, rather than a single edge, as in the basic construction.) In any case, one has
$\{ b \} \nearrow v_{\mathcal{P}}$ and one argues just as in Proposition \ref{proposition:1and2} to prove that the vertices of $\Delta^{f}_{\mathcal{B}}$ are a directed set.
\end{proof}

\subsection{The Brin-Thompson groups $nV$} \label{subsection:brin} 

We next turn to the Brin-Thompson groups $nV$. We will concentrate on the group $2V$, the remaining groups (for $n>2$) being defined similarly.

Define $X$, $X^{fin}$, and $\sigma_{\alpha_{1}}^{\alpha_{2}}$ exactly as in Subsection \ref{subsection:thompson}. We define $S$ as follows:
\[ S = \{ \sigma_{\alpha_{1}}^{\alpha_{2}} \times \sigma_{\beta_{1}}^{\beta_{2}} \mid
\alpha_{1}, \alpha_{2}, \beta_{1}, \beta_{2} \in X^{fin} \} \cup \{ 0 \}. \]
The elements of $S$ act on $X^{2}$ coordinate-wise. That is, 
$\sigma_{\alpha_{1}}^{\alpha_{2}} \times \sigma_{\beta_{1}}^{\beta_{2}}$ acts on a pair $(\omega_{1}, \omega_{2}) \in X^{2}$ by replacing prefixes coordinate-by-coordinate. The transformation
$\sigma_{\alpha_{1}}^{\alpha_{2}} \times \sigma_{\beta_{1}}^{\beta_{2}}$ is undefined on the input $(\omega_{1},\omega_{2})$ if $\omega_{1}$ does not begin with $\alpha_{1}$ or if $\omega_{2}$ does not begin with 
$\alpha_{2}$. It is straightforward to check that $S$ is an inverse semigroup acting on $X^{2}$.

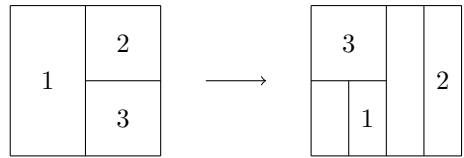
\begin{figure}[!h]
\begin{center}
\begin{tikzpicture}

\draw[black](0,0) -- (2,0);
\draw[black](0,0) -- (0,2);
\draw[black](0,2) -- (2,2);
\draw[black](2,0) -- (2,2);

\draw[black](1,0) -- (1,2);
\draw[black](1,1) -- (2,1);

\node at (.5,1){$1$};
\node at (1.5,1.5){$2$};
\node at (1.5,.5){$3$};

\draw[black,->](2.6,1) -- (3.4,1); 

\draw[black](4,0) -- (6,0);
\draw[black](4,0) -- (4,2);
\draw[black](4,2) -- (6,2);
\draw[black](6,0) -- (6,2);

\draw[black](5,0) -- (5,2);
\draw[black](5.5,0) -- (5.5,2);
\draw[black](4.5,0) -- (4.5,1);
\draw[black](4,1) -- (5,1);

\node at (4.75,.5){$1$};
\node at (5.75,1){$2$};
\node at (4.5,1.5){$3$};

\end{tikzpicture}
\end{center}
\caption{Here we have given a box-pair diagram that represents a member of 
$\widehat{S}$. The large left and right squares represent the product $X \times X$ (of the Cantor set with itself). The smaller rectangles represent subsets of $X \times X$, all of which have the form $X_{\omega_{1}} \times X_{\omega_{2}}$, for appropriate $\omega_{1}$ and $\omega_{2}$. The numberings of the boxes indicate the correspondence between subrectangles.} 
\label{fig5}
\end{figure}

\begin{figure}[!b]
\begin{center}
\begin{tikzpicture}

\draw[black] (1.4,0) -- (0,1.4);
\draw[black] (1.4,0) -- (2.8,1.4);
\draw[black] (0,1.4) -- (1.4, 2.8);
\draw[black] (2.8,1.4) -- (1.4,2.8);

\draw[gray] (1.4,0) -- (1.4,2.8);

\filldraw[lightgray] (1.1,-.8) rectangle (1.7,-.2);
\draw[black](1.1,-.8) rectangle (1.7,-.2);

\filldraw[lightgray] (-.8,1.1) rectangle (-.2,1.7);
\draw[black](-.8,1.1) rectangle (-.2,1.7);
\draw[black](-.8,1.4) -- (-.2,1.4); 

\filldraw[lightgray] (3,1.1) rectangle (3.6,1.7);
\draw[black](3,1.1) rectangle (3.6,1.7);
\draw[black](3.3,1.1) -- (3.3,1.7); 

\filldraw[lightgray] (1.1,3) rectangle (1.7,3.6);
\draw[black](1.1,3) rectangle (1.7,3.6);
\draw[black](1.1,3.3) -- (1.7,3.3); 
\draw[black](1.4,3) -- (1.4,3.6); 

\end{tikzpicture}
\end{center}
\caption{Here is a picture of $\mathcal{E}([f, X^{2}])$. The bottommost vertex (labelled by a shaded square) represents 
$\{ [f, X^{2}] \}$. The left and right vertices are labelled by two different dyadic subdivisions of the square: the left vertex is the result of a horizontal bisection of the square, while the right vertex  is the result of a vertical bisection of the square. The top vertex combines both bisections, resulting in four equal pieces.}
\label{fig6} 
\end{figure}
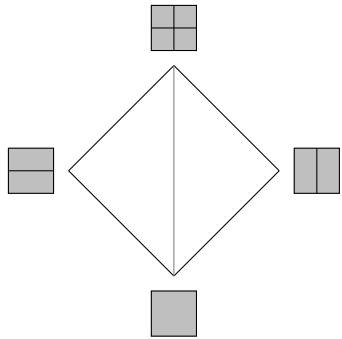

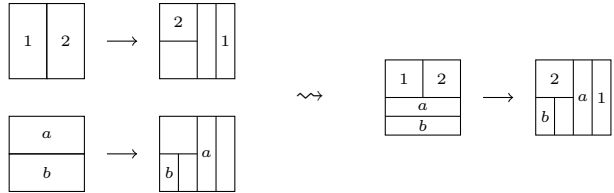
\begin{figure}[!b]
\begin{center}
\begin{tikzpicture}

\draw[black] (0,0) rectangle (1,.5);
\draw[black] (0,.5) rectangle (1,1); 
\node at (.5,.25){{\tiny $b$}};
\node at (.5,.75){{\tiny $a$}};

\draw[black,->](1.3,.5) -- (1.7,.5); 

\draw[black] (2,0) rectangle (3,1);
\draw[black] (2.5,0) -- (2.5,1);
\draw[black] (2.75,0) -- (2.75,1);
\draw[black] (2.25,0) -- (2.25, .5);
\draw[black] (2,.5) -- (2.5,.5);
\node at (2.125,.25){{\tiny $b$}};
\node at (2.625,.5){{\tiny $a$}};

\draw[black] (0,1.5) rectangle (.5,2.5);
\draw[black] (.5,1.5) rectangle (1,2.5); 
\node at (.25,2){{\tiny $1$}};
\node at (.75,2){{\tiny $2$}};

\draw[black,->](1.3,2) -- (1.7,2); 

\draw[black] (2,1.5) rectangle (3,2.5);
\draw[black] (2.5,1.5) -- (2.5,2.5);
\draw[black] (2.75,1.5) -- (2.75,2.5);
\draw[black] (2,2) -- (2.5,2);
\node at (2.875,2){{\tiny $1$}};
\node at (2.25,2.25){{\tiny $2$}};


\node at (4,1.25){{\large $\leadsto$}}; 

\draw[black] (5,.75) rectangle (6,1.75);
\draw[black] (5,1.25) -- ( 6,1.25);
\draw[black] (5.5, 1.25) -- (5.5, 1.75);
\draw[black] (5,1) -- (6,1);
\node at (5.25,1.5){{\tiny $1$}};
\node at (5.75,1.5){{\tiny $2$}};
\node at (5.5,1.125){{\tiny $a$}};
\node at (5.5,.875){{\tiny $b$}};

\draw[black,->](6.3,1.25) -- (6.7,1.25); 

\draw[black] (7,.75) rectangle (8,1.75);
\draw[black] (7.5,.75) -- (7.5,1.75);
\draw[black] (7.75,.75) -- (7.75,1.75);
\draw[black] (7.25,.75) -- (7.25,1.25);
\draw[black] (7,1.25) -- (7.5,1.25);
\node at (7.25,1.5){{\tiny $2$}};
\node at (7.875,1.25){{\tiny $1$}};
\node at (7.625, 1.25){{\tiny $a$}};
\node at (7.125, 1){{\tiny $b$}};

\end{tikzpicture}
\end{center}
\caption{Here we have depicted a contraction of two typical pairs $[f,X^{2}]$ and $[g,X^{2}]$ (on the top left and bottom left, respectively). We can compress the domains of each vertically, up to equivalence, which results in pairs with the domains $X \times X_{1}$ and $X \times X_{0}$, respectively. A straightforward check shows that the pair $[h,X^{2}]$ (on the right) expands to produce
$[f,X^{2}]$ and $[g,X^{2}]$: one simply restricts $h$ to the top and bottom halves of $X^{2}$ (i.e., the required expansion is the left vertex in Fig \ref{fig6}). Note that one could also produce a contraction by an analogous horizontal compression of $[f,X^{2}]$ and $[g,X^{2}]$, and the roles of $[f,X^{2}]$ and $[g,X^{2}]$ could also be reversed.} 

\label{fig7}
\end{figure}

We let $\widehat{S}$ be the inverse semigroup satisfying Convention \ref{convention:SShatGamma}. We note that no order is under consideration, so the construction of $\mathcal{B}$ will be ``permutational". A member of $\widehat{S}$ can be represented by a box-pair diagram, as in Figure \ref{fig5}. The given diagram
represents the disjoint union
\[ f = \left( \sigma_{0}^{01} \times \sigma_{\epsilon}^{0} \right) \coprod
\left( \sigma_{1}^{11} \times \sigma_{1}^{\epsilon} \right) \coprod
\left( \sigma_{1}^{0} \times \sigma_{0}^{1} \right), \]
where the factors of the above disjoint union represent ``$1$", ``$2$", and ``$3$" from Figure \ref{fig5}, respectively. Thus, the subrectangle
$X_{0} \times X$ (labelled by ``$1$" in the left box) is moved to the subrectangle $X_{01} \times X_{0}$ in the right box. The remaining factors send $X_{1} \times X_{1}$ and $X_{1} \times X_{0}$ to $X_{11} \times X$ and $X_{0} \times X_{1}$, respectively. On each piece, the transformation $f$ consists of stretching or compressing horizontally and vertically, while leaving left-to-right and bottom-to-top orientations unchanged. An unlabelled subrectangle in the right-hand square indicates a portion of $X^{2}$ that lies outside of the image of $f$.

The full-support group $\Gamma$ is the same as the group $2V$ from \cite{BrinHighD}. 
Our construction of the expansion set $\mathcal{B}$ follows that of the basic construction 
from Subsection \ref{subsection:basic}, the sole difference again being 
the definition of $\mathcal{E}(b)$. Each $b \in \mathcal{B}$ can be written in the form
$[f,X^{2}]$ up to equivalence. The members of $\mathcal{E}(b)$ (and their partial ordering)
can best be described with the aid of Figure \ref{fig6}. The bottommost vertex represents 
the vertex $\{ [f, X^{2}] \}$. The left vertex $v_{\ell}$ is
$\{ [f, X \times X_{i}] \mid i \in \{ 0, 1\} \}$, the right vertex $v_{r}$ is 
$\{ [f, X_{i} \times X] \mid i \in \{ 0, 1 \} \}$, and the top vertex $v_{t}$
is
$\{ [f, X_{i} \times X_{j}] \mid i,j \in \{ 0,1 \} \}$. The simplicial realization 
 is a topological square.

In the case of the groups $nV$ for $n>2$, one has a similar picture. The main difference is that the sets $\mathcal{E}(b)$ are cubes of dimension $n$. Each vertex in $\mathcal{E}(b)$ that is adjacent to $\{ b \}$ represents a bisection of the set $X^{n}$ in one of the coordinate directions. The other vertices in $\mathcal{E}(b)$ combine these subdivisions in an obvious way. If we let
an $n$-tuple $(a_{1}, \ldots, a_{n}) \in \{ 0,1 \}^{n}$ label each corner of the $n$-cube, where $b$ is the corner with the label $(0,0,\ldots, 0)$, then the
corner $(a_{1}, \ldots, a_{n})$ has a ``cut" in each coordinate direction for which $a_{i} =1$. 

\begin{theorem} \cite{nV}
\label{theorem:nVFinfinity}
The groups $nV$ have type $F_{\infty}$, for all $n \geq 2$.
\end{theorem}

\begin{proof}
The proof is very much like the proof of Theorem \ref{theorem:VFinfinity}. 

We begin with (5) from Theorem \ref{theorem:template}. The bounded contractions property holds with the constant $C_{0} = 2^{n}$. This follows from the description of $\mathcal{E}(b)$ in the immediately preceding discussion. The rich-in-contractions property holds with 
constant $C_{1} = 2$ (i.e., $C_{1}$ is independent of $n$). Figure \ref{fig7} illustrates how to contract any two members of $\mathcal{B}$. The figure shows the $n=2$ case, but the basic result is true for all $n$. Thus, (5) is satisfied.

We note that the stabilizer groups $\widehat{S}_{b}$ are all trivial, and the sets 
$\mathcal{E}(b)$ are finite. It follows that property (3) from Theorem \ref{theorem:template}
holds.

Properties (1) and (4) from Theorem \ref{theorem:template} are satisfied for essentially the same reasons as in the proof of Theorem \ref{theorem:VFinfinity}. 

Property (2) is somewhat more complicated to check. We consider the case $n=2$, and indicate the general idea for larger $n$. We can assume, by Remark \ref{remark:discussion}, that 
$b = [id_{X^{2}}, X^{2}]$. If $n=2$, the relative ascending link
$lk_{\uparrow}(b,v)$ is a complex spanned by some subset of three particular vertices, namely the top, left, and right corners from Figure \ref{fig6}. It is not difficult to see that, if
$lk_{\uparrow}(b,v)$ contains the top vertex $v_{t}$, then it must also contain both the left and the right vertices, since both of these are clearly less than $v_{t}$ in the partial order
$\nearrow$. The link in this case is therefore a cellulated line segment, and therefore contractible. If $lk_{\uparrow}(b,v)$ is a single vertex (either the left or the right), then
it is clearly contractible. Thus, the only potential problem is that $lk_{\uparrow}(b,v)$ consists of only the left and right vertices, but not the top vertex. This, however, is easily seen to be impossible: the vertex $v$ represents some common refinement of the partitions represented by the left and right vertices, which must be a refinement of the partition represented by the top vertex; thus, $v_{t} \nearrow v$. It follows that $lk_{\uparrow}(b,v)$ is necessarily contractible.

For larger $n$, the proof is similar; indeed, the corners of the $n$-cube have a natural lattice 
structure, and one can argue that $lk_{\uparrow}(b,v)$ is closed under least upper bounds. 
An idea like this was used in \cite{nV}.
\end{proof}

\subsection{A Nekrashevych-R\"{o}ver group} \label{subsection:rover}

We next consider a group first by R\"{o}ver \cite{Rover}, generalizations of which were considered by Nekrashevych \cite{Nek}.
R\"{o}ver's group combines the Grigorchuk group \cite{Grigorchuk} and Thompson's group $V$ into one. 

We need to develop some basic conventions for describing automorphisms of the rooted infinite binary tree $T$. The ends of such a tree can be identified with the Cantor set $X = \prod_{i=1}^{\infty} \{ 0, 1 \}$, and therefore each $g \in Aut(T)$ determines a function $g: X \rightarrow X$. If $g_{1}, g_{2} \in Aut(T)$, then we let $(g_{1}, g_{2})$ be the automorphism defined by the rule
\[ (g_{1},g_{2})(\omega) = \left\{ \begin{array}{ll} 0g_{1}(\omega_{1}) & \text{ if }\omega = 0\omega_{1} \\
1g_{2}(\omega_{1}) & \text{ if }\omega = 1\omega_{1} \end{array} \right. . \]
Thus, $(g_{1},g_{2})$ fixes the first symbol of $\omega$, acting either as $g_{1}$ or $g_{2}$ on the remainder of $\omega$, according to whether the first symbol was $0$ or $1$, respectively. 

The Grigorchuk group $G_{0}$ is generated by four automorphisms: $a$, $b$, $c$, and $d$. 
The generator $a$ alters the first entry of any input 
$\omega \in X$, replacing a $1$ with a $0$ or a $0$ with a $1$. The remainder of $\omega$ is unchanged. The generators $b$, $c$, and $d$ can be simultaneously defined using the above pair notation:
\[ b = (a,c), \quad 
c = (a,d), \quad  
d = (1,b). \]
We let $K$ denote the group $\{ 1, b, c, d \}$. This is indeed a group, which is well-known to be isomorphic to the elementary abelian group of order $4$.

Let 
\[ S = \{ \sigma_{\epsilon}^{\omega_{2}} g \sigma_{\omega_{1}}^{\epsilon} \mid g \in G_{0}; \, \omega_{1}, \omega_{2} \in X^{fin} \}. \]
We again let $\widehat{S}$ denote the collection of all finite disjoint unions of members of $S$. The full-support group $\Gamma$
is R\"{o}ver's group. 

The construction of the expansion set $\mathcal{B}$ uses a different definition of the equivalence relation $\sim$ between pairs. The domains $\mathcal{D}^{+}_{S}$ still take the form $X_{\alpha}$, where $\alpha \in X^{fin}$. 
We write $(f,X_{\alpha}) \sim (g,X_{\beta})$ if there $h: X_{\alpha} \rightarrow X_{\beta}$ such that 
$h = \sigma_{\epsilon}^{\beta}k \sigma_{\alpha}^{\epsilon}$ for some $k \in K$, and
\[ g_{\mid X_{\beta}} \circ h = f_{\mid X_{\alpha}}. \]
 The proof that $\sim$ is an equivalence relation is just like the argument in Subsection  \ref{subsection:basic}, and relies on the fact that $K$ is a group. As always, we define 
$supp([f,X_{\alpha}]) = f(X_{\alpha})$. 

The sets $\mathcal{E}(b)$ are defined as indicated in Figure \ref{fig8}. The bottommost vertex is $b = [f,X_{\alpha}]$. 
The left vertex represents the standard subdivision into left and right halves, exactly as in Thompson's groups $F$, $T$, and $V$. The
right vertex also represents a subdivision into left and right halves, but precomposed with a twisting factor
$\sigma_{\epsilon}^{\alpha1}a\sigma_{\alpha1}^{\epsilon}$ on the right half. The latter transformation acts below the node $\alpha1$ in the same way that $a$ acts at the root. The top vertex is a three-fold subdivision of $[f,X_{\alpha}]$, in which one first performs the left-right subdivision, and then applies this basic subdivision to the right half. The verification of (3) from Definition \ref{definition:expansionsets} comes down to the observation that the basic left-right subdivision of
$[f\sigma_{\epsilon}^{\alpha1}a\sigma_{\alpha1}^{\epsilon},X_{\alpha1}]$ undoes the twisting, resulting in the pairs
$[f, X_{\alpha10}]$ and $[f,X_{\alpha11}]$.  

\begin{theorem} \label{theorem:roever} \cite{BelkMatucci} (The R\"{o}ver group has type $F_{\infty}$)
The group $\Gamma$ has type $F_{\infty}$.
\end{theorem}

\begin{proof}
We note that the construction of this subsection is different from the previous ones, in that
we are using a heavily restricted subset of $S$ in order to define the equivalence relation. 
Indeed, there are only four possibilities for the transformation $h: X_{\alpha} \rightarrow X_{\beta}$ under our definition of $\sim$, while there are infinitely many bijections
$s: X_{\alpha} \rightarrow X_{\beta}$ such that $s \in S$. In \cite{FH2}, this restriction is a special instance of an ``$S$-structure". Here we have described the equivalence relation without direct reference to that idea.

The properties (1) and (3)-(5) are easy to check. Property (5) is satisfied with constants
$C_{0}=3$ and $C_{1}=2$, respectively. (The ``rich-in-contractions" property works in exactly the same way as in $V$.) The stabilizer groups $\widehat{S}_{b}$ have order $4$.

The most substantial part of the proof involves checking (2). We omit the argument, which is Claim 6.25 in \cite{FH2}.
\end{proof}

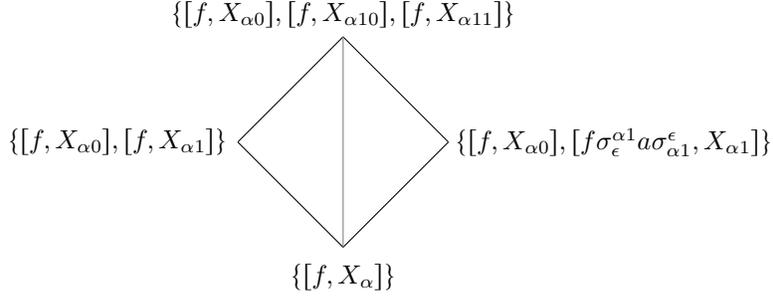
\begin{figure}[!t]
\begin{center}
\begin{tikzpicture}

\draw[black] (1.4,0) -- (0,1.4); 
\draw[black] (1.4,0) -- (2.8,1.4); 
\draw[black] (0,1.4) -- (1.4, 2.8); 
\draw[black] (2.8,1.4) -- (1.4,2.8); 
\node at (1.4,-.4){$\{ [f,X_{\alpha}] \}$};
\node at (-1.6,1.4){$\{ [f,X_{\alpha0}], [f,X_{\alpha1}] \}$};
\node at (5,1.4){$\{ [f,X_{\alpha0}], [f\sigma_{\epsilon}^{\alpha1}a\sigma_{\alpha1}^{\epsilon},X_{\alpha1}] \}$};
\node at (1.4,3.1){$\{ [f,X_{\alpha0}], [f,X_{\alpha10}], [f,X_{\alpha11}] \}$};

\draw[gray] (1.4,0) -- (1.4,2.8);

\end{tikzpicture}
\end{center}
\caption{The partially ordered set $\mathcal{E}(b)$ is a topological square; it is the realization of the Hasse diagram depicted here.}
\label{fig8}
\end{figure}

\subsection{The Lodha-Moore group and generalizations} \label{subsection:LM} 

In \cite{LM}, Lodha and Moore defined a group, which we will call the \emph{Lodha-Moore group}, and denote by $G$. Lodha and Moore \cite{LM} showed that the group $G$ is a finitely presented non-amenable group with no nonabelian free subgroups and can be defined by a presentation with only three generators and nine relators. Lodha \cite{Lodha} later showed that $G$ in fact has type $F_{\infty}$. This made $G$ the first-known $F_{\infty}$ non-amenable group with no nonabelian free subgroups.

The Lodha-Moore group is a group of piecewise projective homeomorphisms of the line. We will now describe a related group, here denoted $G'$, which is such that $G$ is an ascending HNN extension of $G'$ \cite{F(LM)}. (It follows, in particular, that $G$ has type $F_{\infty}$ if $G'$ does.) The group $G'$ is a group of piecewise projective homeomorphisms of the half-line $X = [0,\infty)$. Consider the inverse semigroup $S$ generated by the transformations $A:[0,1) \rightarrow [0,1/2)$, 
$B:[0,1) \rightarrow [1/2,1)$, $C:[0,1) \rightarrow [0,1)$, and $T:[0,\infty) \rightarrow [1,\infty)$, which are defined as follows:
\[ 
A(x) = \frac{x}{x+1}; \quad
B(x) = \frac{1}{2-x}; \quad
C(x) = \frac{2x}{x+1}; \quad 
T(x) = x+1.
\] 
We can consider an inverse semigroup $\widehat{S}$ by observing the ``linear" version of Convention \ref{convention:SShatGamma}, so that each transformation in $\widehat{S}$ is a homeomorphism (with respect to the usual topology on $\mathbb{R}$) between its domain and codomain. We define $G'$ to be the full-support group $\Gamma$ of homeomorphisms of $[0,\infty)$.

The set of domains $\mathcal{D}^{+}_{S}$ appears to be very difficult to describe. This presents a problem when one tries to apply the methods of \cite{FH2}, which make the set of \emph{all} domains in $\mathcal{D}^{+}_{S}$ part of the definition of the set $\mathcal{A}$. One idea from \cite{F(LM)} (and the main departure from the methods of \cite{FH2}) is to build from a simpler foundation. Let $I = [0,1)$, and let $\{ A, B \}^{\ast}$ denote the set of all positive words in the generators $A$ and $B$, including the empty word. We let 
\[ \mathcal{D}^{+}_{S,gen} = \{ T^{k}\omega(I) \mid \omega \in \{ A, B \}^{\ast}; k \geq 0 \} \cup 
\{ [k,\infty) \mid k \in \mathbb{Z}; k \geq 0 \}. \]
We call $\mathcal{D}^{+}_{S,gen}$ the set of \emph{generating domains}. There is a straightforward correspondence between the domains 
$\omega(I)$, for $\omega \in \{ A, B \}^{\ast}$ and the nodes of the infinite binary tree, which can be described inductively. 
The root of the tree corresponds to $I$. If a given node of the tree corresponds to the interval 
$\omega(I)$, then the left and right children correspond to $\omega A(I)$ and $\omega B(I)$, respectively. For instance, the interval $AB(I)$ is
$[1/3,1/2)$, which corresponds to the right child of the left child of the root.
Thus, the correspondence described here must not be confused with the usual correspondence between standard dyadic intervals and nodes of the infinite binary tree as described in \cite{CFP}, which assigns the interval $[1/4,1/2]$ to the same node. 
  
We now proceed to the description of the expansion set $\mathcal{B}$. We consider pairs
$(f,D)$ and $(g,E)$, where $f,g \in \widehat{S}$ and $D$ and $E$ are generating domains. The equivalence relation on such pairs is defined exactly as before: $(f,D) \sim (g,E)$ if
there is some $h \in S$ such that $h:D \rightarrow E$ is a bijection and
\[ g_{\mid E} \circ h = f_{\mid D}. \]
The support of a pair $[f,D]$ is $f(D)$.

The sets $\mathcal{E}(b)$ take two forms, depending on the form of $b$.
If $b = [f,[k,\infty)]$, then $\mathcal{E}(b)$ is a line segment. The lower vertex is $\{ b \}$, while the upper vertex is
\[ \{ [f, [k,k+1)], [f, [k+1,\infty) ] \}. \]
If $b = [f,D]$ for some $D \in \mathcal{D}^{+}_{S,gen}$ of the form $T^{k}\omega(I)$, then we can, up to equivalence, express $b$ in the form $[g,I]$. The simplicial complex $\mathcal{E}(b)$ is a simplicial cone on a cellulated real line, whose vertices are in the set $\frac{1}{2}\mathbb{Z}$. See Figure \ref{fig9}.
\begin{figure} [!h]
\begin{center}
\begin{tikzpicture}

\filldraw[lightgray] (4,1.5) -- (0,3) -- (1,3.25) -- cycle;
\filldraw[lightgray] (4,1.5) -- (2,3) -- (3,3.25) -- cycle;
\filldraw[lightgray] (4,1.5) -- (4,3) -- (5,3.25) -- cycle;
\filldraw[lightgray] (4,1.5) -- (6,3) -- (7,3.25) -- cycle;
\filldraw[lightgray] (4,1.5) -- (1,3.25) -- (2,3) -- cycle;
\filldraw[lightgray] (4,1.5) -- (3,3.25) -- (4,3) -- cycle;
\filldraw[lightgray] (4,1.5) -- (5,3.25) -- (6,3) -- cycle;
\filldraw[lightgray] (4,1.5) -- (7,3.25) -- (8,3) -- cycle;

\draw[gray, thick] (4,1.5) -- (0,3);
\draw[gray, thick] (0,3) -- (1,3.25);
\draw[gray, thick] (4,1.5) -- (1,3.25);
\draw[gray, thick] (1,3.25) -- (2,3);
\draw[gray, thick] (4,1.5) -- (2,3);
\draw[gray, thick] (2,3) -- (3,3.25);
\draw[gray, thick] (4,1.5) -- (3,3.25);
\draw[gray, thick] (3,3.25) -- (4,3);
\draw[gray, thick] (4,1.5) -- (4,3);
\draw[gray, thick] (4,3) -- (5,3.25);
\draw[gray, thick] (4,1.5) -- (5,3.25);
\draw[gray, thick] (5,3.25) -- (6,3);
\draw[gray, thick] (4,1.5) -- (6,3);
\draw[gray, thick] (6,3) -- (7,3.25);
\draw[gray, thick] (4,1.5) -- (7,3.25);
\draw[gray, thick] (7,3.25) -- (8,3);
\draw[gray, thick] (4,1.5) -- (8,3);

\node at (0,3.25){-2};
\node at (2,3.25){-1};
\node at (4,3.25){0};
\node at (6,3.25){1};
\node at (8,3.25){2};
\node at (4,1.25){b};

\filldraw [gray] (0,2.25) circle (1.5pt);
\filldraw [gray] (-.5,2.25) circle (1.5pt);
\filldraw [gray] (-1,2.25) circle (1.5pt);

\filldraw [gray] (8,2.25) circle (1.5pt);
\filldraw [gray] (8.5,2.25) circle (1.5pt);
\filldraw [gray] (9,2.25) circle (1.5pt);
\end{tikzpicture}
\end{center}
\caption{The infinite cellulated fan is the simplicial realization of $\mathcal{E}([g,I])$. Vertices that are higher on the page have a greater height (in the sense of Definition \ref{definition:height}): vertices with integral labels have height $2$, while the ``fractional" vertices have height $3$. The length of a maximal ascending chain in $\mathcal{E}([g,I])$ is $3$.}
\label{fig9}
\end{figure}
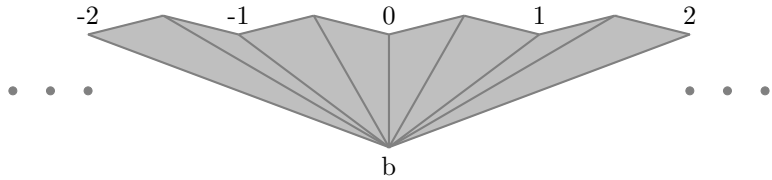
The vertices with integral labels $k$ represent vertices of the following form:
\[ \{ [gC^{k}, A(I)], [gC^{k}, B(I)] \}. \]
Such vertices represent a left-right subdivision of $[g,I]$ with a ``twisting" factor of $C^{k}$.
A ``fractional" vertex $k-1/2$ is obtained by subdividing the first member of the above vertex in half:
\[ \{ [gC^{k}, AA(I)], [gC^{k}, AB(I)], [gC^{k},B(I)] \}. \]

\begin{theorem} \cite{F(LM)}
The group $G'$ has type $F_{\infty}$.
 \end{theorem}
 
\begin{proof}
We discuss the group $G'$ under the headings (1)-(5) from Theorem \ref{theorem:template}.

Property (1) is somewhat more complicated than normal. The directed set property is proved in Theorem 4.12 of \cite{F(LM)}; this crucially uses a few semigroup identities from Lemma 3.3 of \cite{F(LM)}.

 Properties (3'), (4), and (5) pose few difficulties. The stabilizer groups $\widehat{S}_{b}$ are infinite cyclic when $b=[g,I]$, or trivial when $b = [g,[0,\infty)]$. It follows from Proposition \ref{proposition:finiteindex} that every vertex stabilizer group $\Gamma_{v}$ has a finite-index free abelian subgroup. This shows that every cell stabilizer $\Gamma_{\sigma}$ is a subgroup of a virtually free abelian group, and therefore has type $F_{\infty}$. (In fact, a closer analysis reveals that vertex stabilizers are always free abelian, not merely virtually so.) The cocompactness of the action of
 $\widehat{S}_{b} \cong \mathbb{Z}$ ($b=[g,I]$) on
 the cellulated fan of Figure \ref{fig9} is rather clear, and the cocompactness is trivial when $b=[g,[0,\infty)]$. This proves (3'). Property (4) is satisfied with three orbits; property (5) is satisfied with the constants
 $C_{0} = 3$ and $C_{1} =2$. 
 
Property (2) is, by far, the most difficult to check. We omit a detailed discussion, but note that, 
for $b=[id_{I},I]$, the relative ascending link $lk_{\uparrow}(b,v)$ is always a cellulated line segment in the fan pictured in Figure \ref{fig9}.  (Property (2) is trivial to check when
$b=[id,[0,\infty)]$; the relative ascending link is always a point in this case.) The proof occupies a large part of \cite{F(LM)} and it is adapted from some of the ideas of \cite{LM}.
\end{proof}
 
\begin{remark}
\label{remark:final}
One feature of the proof from \cite{F(LM)} (and the argument in this paper) is that it handles the ``$F$", ``$T$", and ``$V$" versions of the Lodha-Moore group equally well. In particular,
\cite{F(LM)} proved that some of the ``$V$" versions of the Lodha-Moore group also have type $F_{\infty}$, which answered a question from \cite{LodhaZ}, where the ``$T$" versions 
of the Lodha-Moore group were proved to have type $F_{\infty}$.

Finally, we note that Kodama \cite{Kodama} introduced various $n$-ary generalizations of the Lodha-Moore group, and proved that the groups in question are finitely presented counterexamples to the von Neumann-Day problem. Moawad \cite{Moawad} used methods similar to the ones we have outlined in this paper to show that certain similar groups have type $F_{\infty}$. The groups in question contain Kodama's groups as proper subgroups. We note that the question of $F_{\infty}$ remains open for the groups in \cite{Kodama}. 
\end{remark}

\bibliographystyle{plain}
\bibliography{biblio}

\end{document}